\documentclass[reqno,final,a4paper]{amsart}
\usepackage[utf8]{inputenc}
\usepackage[shortlabels]{enumitem}
\usepackage[notref,notcite]{showkeys}
\usepackage{amsmath}
\usepackage{amssymb}
\usepackage{amsthm}
\usepackage{abstract}
\usepackage{xcolor}
\usepackage{comment}
\usepackage{mathtools}
\usepackage{graphicx}
\usepackage{caption}
\usepackage{subcaption}
\usepackage{float}
\usepackage{hyperref}
\usepackage{enumitem}
\usepackage{tikz}
\usepackage{bbm}
\usepackage[normalem]{ulem}

\renewenvironment{abstract}{%
  \noindent\textbf{Abstract.}%
}{}

\makeatletter
\def\namedlabel#1#2{\begingroup
   \def\@currentlabel{#2}%
   \label{#1}\endgroup
}
\makeatother
\usepackage{dsfont}
\hypersetup{
    colorlinks=true,
    linkcolor=blue,
    filecolor=magenta,      
    urlcolor=cyan
    }
\allowdisplaybreaks
\makeatletter
\newcommand*{\barfix}[2][.175ex]{%
  \mathpalette{\@barfix{#1}}{#2}%
}
\newcommand*{\@barfix}[3]{%
  \vbox{%
    \kern#1\relax
    \hbox{$#2#3\m@th$}%
  }%
}
\makeatother

\usepackage[capitalise]{cleveref}
\newtheorem{theorem}{Theorem}
\newtheorem{thm}{Theorem}[section]
\newtheorem{corollary}[thm]{Corollary}
\newtheorem{lemma}[thm]{Lemma}
\newtheorem{proposition}[thm]{Proposition}
\newtheorem{claim}[thm]{Claim}

\theoremstyle{definition}

\newtheorem{definition}[thm]{Definition}
\newtheorem{remark}[thm]{Remark}

\newcommand{\cA}{\mathcal{A}}
\newcommand{\cB}{\mathcal{B}}
\newcommand{\cC}{\mathcal{C}}
\newcommand{\cD}{\mathcal{D}}
\newcommand{\cE}{\mathcal{E}}
\newcommand{\cF}{\mathcal{F}}
\newcommand{\cG}{\mathcal{G}}

\newcommand{\cI}{\mathcal{I}}

\newcommand{\cM}{\mathcal{M}}

\newcommand{\cP}{\mathcal{P}}

\newcommand{\cS}{\mathcal{S}}
\newcommand{\cT}{\mathcal{T}}

\newcommand{\cY}{\mathcal{Y}}

\newcommand{\dtv}{d_{\mathrm{TV}}}
\newcommand{\tmix}{t_{\mathrm{mix}}}
\newcommand{\mult}{\mathrm{mult}}

\newcommand{\whp}{%
    \textbf{whp}
}

\usepackage[margin=1in]{geometry}

\def\e{\mathrm{e}}
\def\Id{\mathrm{Id}}
\def\eps{\varepsilon}

\def\vbad{V_{\mathrm{bad}}}

\newcommand{\1}{\mathbbm{1}}

\title{Diameter and mixing time of the giant component in the percolated hypercube}

\author[Anastos]{Michael Anastos}

\address{Institute of Science and Technology Austria (ISTA), Klosterneurburg 3400, Austria}
\email{michael.anastos@ist.ac.at}

\author[Diskin]{Sahar Diskin}

\address{D-MATH ETH Z\"urich, R\"amistrasse 101, 8092 Z\"urich, Switzerland}
\email{sahardiskinmail@gmail.com}

\author[Lichev]{Lyuben Lichev}

\address{Institute of Statistics and Mathematical Methods in Economics, TU Wien, A-1040 Vienna, Austria}
\email{lyuben.lichev@tuwien.ac.at}

\author[Zhukovskii]{Maksim Zhukovskii}

\address{School of Computer Science, University of Sheffield, UK}
\email{m.zhukovskii@sheffield.ac.uk}

\thanks{Anastos was supported by the Austrian Science Fund (FWF) grant No.~10.55776/ESP3863424. Diskin was supported in part by the BSF Prof. Rahamimoff Travel Grant for Young Scientists (T-2023202). Lichev was supported by the Austrian Science Fund (FWF) grant No.~10.55776/ESP624. For open access purposes, the authors
have applied a CC BY public copyright license to any author-accepted manuscript version arising from this
submission.}

\begin{document}

\maketitle

\begin{abstract}
We consider bond percolation on the $d$-dimensional binary hypercube with $p=c/d$ for fixed $c>1$. 
We prove that the typical diameter of the giant component $L_1$ is of order $\Theta(d)$, and the typical mixing time of the lazy random walk on $L_1$ is of order $\Theta(d^2)$. This resolves long-standing open problems of Bollob\'as, Kohayakawa and \L{}uczak from 1994, and of Benjamini and Mossel from 2003.

A key component in our approach is a new tight large deviation estimate on the number of vertices in $L_1$ whose proof includes several novel ingredients: a structural description of the residue outside the giant component after sprinkling, a tight quantitative estimate on spreadness properties of the giant in the hypercube, and a stability principle which rules out the disintegration of large connected sets under thinning. This toolkit further allows us to obtain optimal bounds on the expansion in~$L_1$.
\end{abstract}

\section{Introduction}
The $d$-dimensional binary hypercube $Q^d$ is the graph with vertex set $\{0,1\}^d$ where an edge connects two vertices if and only if they differ in a single coordinate. The hypercube is thus a $d$-regular bipartite graph on $2^d$ vertices.
For $p=p(d)\in [0,1]$, the $p$-percolated hypercube $Q^d_p$ is obtained by retaining every edge of $Q^d$ independently and with probability $p$.

The study of the percolated hypercube $Q^d_p$ was pioneered by Burtin~\cite{B77} and Sapo\v{z}enko \cite{S67} who showed that $p=1/2$ is a sharp threshold for connectivity of $Q^d_p$. 
A more drastic change of behaviour of typical instances of the model happens around $p=1/d$.
Erd\H{o}s and Spencer~\cite{ES79} observed that when $p=c/d$ with $c<1$, $Q^d_p$ typically contains only components of order $O(d)$, and Ajtai, Koml\'os, and Szemer\'edi~\cite{AKS81} completed the picture by showing that $Q^d_p$ typically contains a giant connected component $L_1$ with linearly many vertices when~$c>1$. 
A subsequent strengthening of these results by Bollob\'as, Kohayakawa and \L{}uczak \cite{BKL92} showed that, when $c>1$, the number of vertices in $L_1$ is typically $y2^d + o(2^d)$ where $y:=y(c)$ is the survival probability of a Galton-Watson process with offspring distribution Poisson($c$), and typically all other (smaller) components have order $O(d)$, mirroring a similar phenomenon in the Erd\H{o}s-R\'enyi random graph $G(n,p)$~\cite{ER60}.

Despite the similarities, many properties of the giant component of the random graph $G(n,c/n)$ with $c>1$ were better understood thanks to the homogeneous nature of the host graph $K_n$.
Two key examples are the typical \textit{diameter} of the giant component and the asymptotic \textit{mixing time of a simple random walk} on the giant component.
The typical diameter of the giant component of $G(n,c/n)$ with $c>1$ was determined very precisely in a line of research~\cite{CL01,FR07,RW10} and turns out to be of asymptotic order $\Theta_c(\log n)$.
The mixing time of a simple random walk, roughly speaking, measures the time needed by a simple random walk to forget its starting point (for a formal definition, see \Cref{s: mixing time}).
The asymptotic mixing time of a simple random walk on the giant component in $G(n,c/n)$ with $c>1$ is of order $\Theta_c((\log n)^2)$~\cite{BKW14,FR08}.

The remarkable similarity between the Erd\H{o}s-R\'enyi random graph and the percolated hypercube with the same (constant) average degree raises the question whether the typical diameter and the typical mixing time of a simple random walk on the giant component $L_1$ in $Q^d_p$ have the same quantitative behaviour.
Already in 1994, Bollob\'as, Kohayakawa, and \L{}uczak \cite[Problem 15]{BKL94d} asked about the order of the typical diameter of $L_1$ and, in particular, if it is polynomial in $d$. 
The latter question was answered in a breakthrough of Erde, Kang, and Krivelevich~\cite{EKK22} who showed -- using expansion properties of the hypercube -- that the typical diameter of $L_1$ is of order $O(d^3)$.
This result was later improved by Diskin, Erde, Kang, and Krivelevich~\cite{DEKK24} to $O(d (\log d)^2)$. 
In each of~\cite{DEKK24,EKK22} it was asked if, similarly to the giant of $G(n,c/n)$, $\Theta_c(d)$ is the right dependency on $d=\log_2|V(Q^d)|$.

In a separate line of research, Benjamini and Mossel~\cite[Section 3]{BM03} asked whether the asymptotic mixing time of a (lazy) simple random walk on $L_1$ is of order $\Theta_c(d^2)$, as in the case of the giant component of $G(n,c/n)$. 
This question was later reiterated by Pete \cite{Pete08}, by van der Hofstad and Nachmias~\cite[Open Problem (6)]{HN17}, and is advertised on the webpage of the Levin--Peres--Wilmer book~\cite{LPW17} ``Markov Chains and Mixing Times''~\cite[Question 3]{Pereswebsite}.
Progress towards this conjecture was made by Erde, Kang and Krivelevich~\cite{EKK22} who provided an upper bound of $O(d^{11})$, and by Diskin, Erde, Kang and Krivelevich~\cite{DEKK24} who showed an upper bound of $O(d^2(\log d)^2)$.

Our first main result resolves these questions and conjectures.
\begin{theorem}\label{th: main}
Fix $c > 1$ and let $p=p(d)=c/d$. Then \textbf{whp}\footnote{With high probability, that is, with probability tending to $1$ as $d\to \infty$.} the giant component $L_1$ in $Q^d_p$ satisfies each of the following properties.
\begin{enumerate}[label=\upshape(\alph*)]
    \item The diameter of $L_1$ is $\Theta_c(d)$.\label{th: diameter}
    \item The mixing time of a lazy simple random walk on $L_1$ is $\Theta_c(d^2)$.\label{th: mixing time}
\end{enumerate}
\end{theorem}
Note that the lower bounds on each of the typical diameter and the typical mixing time follow from the existence of bare paths (that is, paths containing only vertices of degree 2) of length $\Omega_c(d)$ in $L_1$, which can be verified by a simple second moment computation (see, e.g.,~\cite[Section 5]{EKK22}).

The proof of Theorem \ref{th: main}\ref{th: mixing time} builds upon Theorem \ref{th: large tail deviation} which establishes essentially tight large deviation estimates on the order of $L_1$.
\begin{theorem}\label{th: large tail deviation}
Fix $c>1$, let $p=p(d)=c/d$ and denote by $y=y(c)$ the survival probability of a Galton-Watson process with offspring distribution Poisson$(c)$.
Then, there exists a constant $\eps=\eps(c)>0$ such that, for all $t\ge 2^d/d^{0.1}$,
\begin{align*}
\mathbb{P}(|V(L_1)|\ge y2^d+t)\le \exp\bigg(-\frac{\eps t^2}{2^d (\log(2^d/t))^2}\bigg) \quad \text{and}\quad \mathbb{P}(|V(L_1)|\le y2^d-t)\le \exp\bigg(-\frac{\eps t\log(2^d/t)}{d}\bigg).
\end{align*}
\end{theorem}

The lower tail in the above theorem is tight for $t\in [2^d/d^{0.1}, y2^d]$ (up to the constant in the exponent). To see this, consider the subcube $Q_0(k)\subseteq Q^d$ of dimension $d-k$ composed of the vertices whose first $k$ coordinates are zero. 
Note that, as long as $p(d-k)$ remains bounded away from 1, \whp the giant component $L_1(k)$ of $(Q_0(k))_p$ contains roughly $y2^{d-k}$ vertices and the number of edges between $L_1(k)$ and $Q^d\setminus Q_0(k)$ is $k|V(L_1(k))|$.
Thus, the probability that the number of vertices in $L_1$ deviates from $y2^d$ by approximately $y2^{d-k}$ is at least $(1-p)^{ky2^{d-k}} = \exp(-cyk2^{d-k}/d)$. Note also, that while the choice of $0.1$ in the exponent of $d$ is rather arbitrary, it cannot be made arbitrarily large; in fact, by estimating from below the deviation of the number of vertices in \textit{small} components, one can see that the lower tail above cannot be better than $\exp\left(-\Omega(t^2/2^d)\right)$, which is worse than the above estimate when $t\ll 2^d/d$.

Let us note that the previous papers \cite{DEKK24,EKK22} that achieved partial progress towards the conjectures resolved in Theorem \ref{th: main} essentially relied on a `direct' sprinkling argument and on the product structure of the hypercube. A key missing ingredient in the previous works was a tight expansion estimate, in particular for large connected sets, in the giant component. Naively, one could hope to obtain such an estimate using a `reverse' sprinkling, also known as `thinning', a process of removing edges with a given probability.
Indeed, if a large set in the giant component has a weak expansion, removing each of the edges in its boundary is `fairly cheap' and results in two large components --- a very atypical structure. However, this does not immediately imply the probability bound of Theorem \ref{th: large tail deviation}, as the set of deleted edges depends on the structure of the exposed random subgraph. It turns out that this constitutes the main challenge in resolving the conjectures. In this paper, we perform a quantitative analysis of small perturbations on the edges of the giant component, allowing us to resolve this challenge (see \Cref{s: outline} for more details). Importantly, our arguments do not rely on the hypercube's product structure, and so they could adapt to other sparse high-dimensional percolation models and to settings where conductance of large sets governs mixing behaviour (see, e.g., \cite{AKLP23,AKP21,BBY08,CS24,JKM24,MD18,MMJ24,OCo98}).

As mentioned above, a key ingredient in the proof of Theorem \ref{th: main}\ref{th: mixing time} lies in \Cref{prop: expansion large} and \Cref{cor: new expansion}, which describe the expansion properties of the giant. The following theorem is an abbreviated version of these two results, which resolves \cite[Question 5.1]{EKK22} and might be of independent interest.
\begin{theorem}\label{th: expansion}
Recall $c > 1$, $p=p(d)=c/d$ and $y=y(c)$ from \Cref{th: large tail deviation}. There exists a constant $\eps=\eps(c)>0$ such that \textbf{whp}, for every subset $S\subseteq V(L_1)$ such that $|S|\le |V(L_1)|/2$ and $L_1[S]$ is a connected graph, $L_1$ contains at least $\eps |S|/d$ edges between $S$ and $L_1\setminus S$.
Furthermore, for any constant $\delta\in(0,1)$, there exists a constant $\eta\coloneqq \eta(c,\delta)>0$ such that \textbf{whp} for every subset $S\subseteq V(L_1)$ with $|S|\in [\delta y 2^d,(1-\delta) y 2^d]$, $L_1$ contains at least $\eta |S|/d$ edges between $S$ and $L_1\setminus S$.
\end{theorem}
Since by the above (not necessarily connected) sets of linear (in $2^d$) size typically edge-expand by a factor of $\Omega(1/d)$, using techniques from \cite{K19} (in particular, \cite[Lemma 2.6 and Theorem~2.7]{K19}) implies the existence of an $\Omega(1/d)$-expander in $L_1$ of linear order, resolving \cite[Question 5.1]{EKK22}. Indeed, the latter result is tight (up to a constant factor), as already noted in \cite[Claim 5.2]{EKK22}.

Finally, complementing results on critical percolation~\cite{B-RBN24,BCVSS06,DKLP11,HS04,HN20} by pinning down the mixing time of a lazy simple random walk on the giant component in the barely-supercritical regime $dp=1+o(1)$ is an intriguing and natural next step, and we believe our techniques can provide key insights to it.

\vspace{1em}

\noindent
\textbf{Organisation.} This paper is structured as follows. In \Cref{s: outline}, we provide an outline for the proofs of Theorems~\ref{th: main} and~\ref{th: large tail deviation}. 
In \Cref{s: prelim}, we introduce notation and collect several auxiliary lemmas. 
In \Cref{s: upper tail}, we prove the upper tail estimates in \Cref{th: large tail deviation}. 
Then, in \Cref{s: lower tail}, we prove the lower tail estimates in \Cref{th: large tail deviation}. 
We use this theorem to prove \Cref{th: main}\ref{th: mixing time} in \Cref{s: mixing time}. Finally, \Cref{th: main}\ref{th: diameter} is shown in \Cref{s: diameter}.

\section{Outline of the proofs}\label{s: outline}
We first outline the proof of \Cref{th: large tail deviation}. 
It is divided into two parts. The proof of the upper tail deviation, presented in Section~\ref{s: upper tail}, is simpler. 
It relies on the observation that, if $L_1$ spans significantly more vertices than expected, then the number of vertices in small components falls significantly below its expectation. 
We bound the probability of the latter event via a routine application of the bounded difference inequality (\Cref{azuma}). 

\Cref{s: lower tail} is dedicated to the proof of the lower tail in \Cref{th: large tail deviation}. 
Before delving into the proof's outline, we first recall the broad strategy used in \cite{AKS81, BKL92, K23} to show that \textbf{whp} the number of vertices in the giant component $L_1$ of $Q^d_p$ is at least $(1-o(1))y2^d$ (with $p,y$ as defined in \Cref{th: large tail deviation}). We note that a similar overall strategy was employed in \cite{DEKK24, EKK22} to establish expansion properties of the giant component $L_1$. The argument goes through sprinkling, also known as multi-stage exposure. In the paragraph below, we give a brief sketch of that proof.

Fix suitably small $\delta=\delta(c)>0$ and define $p_1 = (c-\delta)/d$ and $p_2$ such that $(1-p_1)(1-p_2)=1-p$. In particular, $p_2\ge \delta/d$ and we further assume that $p_1d>1$. Let $G_1=Q^d_{p_1}\subseteq G_2=Q^d_{p_1}\cup Q^d_{p_2}$, where the graphs in the union are sampled independently. Note that $Q^d_p$ has the same distribution as $G_2$. 
First, one considers $G_1$ and shows that \textbf{whp} the number of vertices in `big' components of $G_1$ is at least $(1-o(1))y(c-\delta)2^d$.\footnote{Note that $y(c-\delta)$ stands for the survival probability of a Galton-Watson process with offspring distribution Poisson$(c-\delta)$ and not for $y\cdot (c-\delta) = y(c)\cdot (c-\delta)$. To distinguish the two, we always write a dot in the latter case.} 
Then, one shows that \textbf{whp} every vertex in $Q^d$ is within distance two (in $Q^d$) to a vertex in a `big' component of $G_1$. 
The final step of the proof is to argue that after sprinkling, i.e. in $Q^d_{p_1}\cup Q^d_{p_2}$, \textbf{whp} all vertices in `big' components of $G_1$ merge into one giant component in $G_2$ of order at least $(1-o(1))y(c-\delta)2^d$. Since $\delta$ is an arbitrarily small constant and $y$ is a continuous function, this is enough to complete the proof. 

There are several obstacles in adapting the above sprinkling argument to obtain \textit{tight} lower-tail estimates for $|V(L_1)|$. First, the estimates on the order of the giant arise from the number of vertices in `large' components in $G_1$. 
This is insufficient, since $y(c)-y(c-\delta)=\Theta(\delta)$ for a small constant $\delta$ (see Lemma~\ref{lem:yb}). This means that when $\delta=\Theta(1)$, one cannot study deviations of the order $o(n)$ by only taking into account the number of vertices in `large' components in $G_1$, since these are off by $\Theta(\delta)2^d$. Decreasing $\delta$ is also not possible, since the probability of merging these components depends on $\delta$. As noted in the introduction, a natural approach to overcome this is to introduce `reverse sprinkling', also known as thinning; however, this thinning procedure cannot depend on the exposed random graph. Thus, a new outlook, proving that `large' connected sets in the giant have a `good' expansion in the giant is required (and developed) here.

Our starting point is a structural description of `large' components that could emerge outside the giant in $G_2$: vertices that lie in components formed by merging many small $G_1$-components after sprinkling, which we call type-1; and vertices that coalesce with a few `large' $G_1$-components, which we call type-2. In order to handle these vertices, we establish a \textit{stability principle} (\Cref{l: no surprises from outside_gen}), which shows that it is highly unlikely for a `large' set in $G_2$ to disintegrate into mostly small components in $G_1$. We note that this principle is crucial for deriving stronger probability bounds, and it plays a central role both in the lower-tail analysis and in the subsequent expansion arguments. To obtain tight estimates using this principle, we further utilise a weighted decomposition and sparsification arguments (\Cref{claim: no sudden collapse help} and \Cref{l: sparsification}), and develop a quantitative \textit{well-spreadness} argument for `large' $G_1$-components (\Cref{l: large components are well spread}), on which we elaborate in the subsequent paragraph.

Indeed, for type-2 vertices, we show that, with very high probability, almost all the vertices outside `large' components of $G_1$ have many $Q^d$-neighbours in `large' $G_1$-components (recall that in the proof of the existence of a giant component, the property that \textbf{whp} every vertex is within distance two to a `large' component in $G_1$ was utilised). Writing $\mathcal B_\varepsilon$ for the set of vertices in $Q^d$ that see fewer than $\varepsilon d$ neighbours in `large' $G_1$-components, we show that $|\mathcal B_\varepsilon|$ has an exponential tail. While the events that are associated with vertices that are close to each other are highly correlated, in \Cref{cl:associate} we are able to construct algorithmically a large subfamily of $\cB_{\eps}$ where the events are almost independent. 

Type-1 vertices are handled by the weighted decomposition and sparsification steps. Intuitively, if many vertices ended up in type-1 components, then a large family of small $G_1$-components should have had an unusually large number of external incident edges that fail to appear after the sprinkling. \Cref{claim: no sudden collapse help} together with \Cref{l: sparsification} allow to isolate a sufficiently small number of extremely unlikely events that imply this property.

We move to the proof of Theorem \ref{th: main}\ref{th: mixing time}. 
In \Cref{subsec: expansion}, we use the results in \Cref{s: lower tail} to show that `large' connected sets in $L_1\subseteq G_2$ typically expand well (Propositions \ref{prop: expansion} and~\ref{prop: expansion large}). 
This shows \Cref{cor: new expansion}, which also implies \Cref{th: expansion}.
Utilising expansion estimates to obtain bounds on the mixing time was used in previous works on hypercube percolation~\cite{DEKK24, EKK22}: a missing ingredient there was a good quantitative estimate on the expansion of `large' sets, which is a major contribution of this paper. 
The proof of \Cref{prop: expansion} is based on the following argument: assuming that a `large' set $S$ in $G_2$ has a small edge-boundary, the probability that $S$ is disjoint from the largest component of $G_1$ (or itself becomes such) is significant. 
As it turns out, the probability bound on the event that $S$ is disjoint from the largest component in $G_1$ obtained from analysis of the sprinkling between $G_1$ and $G_2$ is tight in general. However, for our needs, we derive a stronger probability bound which can hold \textit{only} for \textit{connected} sets (see Section \ref{subsec: expansion} and the discussion therein). To that end, we utilise the rather general stability principle (\Cref{l: no surprises from outside_gen}), which was also used to prove Theorem \ref{th: large tail deviation}. 
Roughly speaking, \Cref{l: no surprises from outside_gen} shows that it is very unlikely for a `large' connected set in $G_2$ to disintegrate into \textit{mostly} small components in $G_1$ (and, in particular, such a set is far from a type-1 component). Thus, with a significant probability, we must have many vertices in `big' components in $G_1$ outside $L_1$ --- an event whose probability we already estimate in Section \ref{s: lower tail}.
Combining the expansion properties of connected subsets in the giant component with previously known expansion estimates for $Q^d_p$ \cite{DEKK24}, one can derive \Cref{th: main}\ref{th: mixing time} by using a theorem of Fountoulakis and Reed \cite{FR07a} (stated here as \Cref{th: mixing-time-tool}). 
The latter result relates the mixing time of a lazy simple random walk on a graph to the expansion properties of connected sets therein.

Finally, we turn to the proof of Theorem \ref{th: main}\ref{th: diameter}. It utilises two key ideas. 
First, we show that, for any given pair of vertices in $Q^d$, with probability $d^{-O(1)}$, the distance between them in $Q^d_p$ is of order $O(d)$ (\Cref{lem:bootstrap_7.1}). 
Second, using typical expansion properties of the giant (\Cref{cor: new expansion}), we observe that \textbf{whp} every vertex in the giant is connected by short paths to many vertices. 
Then, to find a path of length $\Omega(d)$ between a pair of vertices $u,v$, we construct \textit{many} vertex-disjoint subcubes: each of them has its all-0 vertex in the neighbouring set of $u$ and its all-1 vertex in the neighbouring set of $v$.
This allows us to bootstrap the inverse-polynomial probability bound from \Cref{lem:bootstrap_7.1} to a \textbf{whp} statement, showing that indeed every pair is typically within distance $\Theta(d)$. 
We note that the proof of \Cref{lem:bootstrap_7.1} utilises careful switching and enumeration arguments, focusing on specific types of paths which are easier to analyse (see Sections~\ref{subsec:firstmoment} and \ref{subsec:secondmoment}).

\section{Preliminaries}\label{s: prelim}
\subsection{Notation}
Given a graph $G=(V,E)$, we denote its \textit{order} by $v(G)\coloneqq |V|$. 
For every two subsets $A,B\subseteq V$, we denote by $e_G(A,B)$ the number of edges with one endpoint in $A$ and one endpoint in $B$ (where edges in $A\cap B$ are counted twice), and by $N_G(A)$ the \emph{neighbourhood of $A$ in $G$}, that is, the set of vertices in $V\setminus A$ with at least one neighbour in $A$. 
We further write $G[A]$ for the subgraph of $G$ induced by $A$ with $e_G(A)\coloneqq |E(G[A])|$, and $G[A,B]$ for the subgraph of $G$ induced by the edges between two disjoint sets $A,B\subseteq V$. 
Given $u,v\in V$, we denote by $d_G(u,v)$ the graph distance (in $G$) between $u$ and $v$. 
In general, when the graph $G$ is clear from the context, we sometimes omit the subscript.

For a set $S\subseteq V(Q^d)$ and a spanning subgraph $G$ of $Q^d$, we say that $S$ is \emph{connected in $G$} if $G[S]$ is a connected graph. We say that $S$ is \emph{connected} if it is connected in $Q^d$.
For a family of disjoint sets of vertices $\cC$ in a graph $G$, we set $V(\cC)\coloneqq \bigcup_{C\in \cC}V(C)$, $v(\cC):=|V(\cC)|$ and $e(\cC)=\sum_{C\in \cC}e(C,V(G)\setminus C)$. 
We further denote by $e_{\mathrm{out}}(\cC)$ the number of edges with exactly one endpoint in $V(\cC)$, and by $e_{\mathrm{in}}(\cC)$ the number of edges with endpoints in two distinct sets of $\cC$. 
In particular, $e(\cC)=e_{\mathrm{out}}(\cC)+2e_{\mathrm{in}}(\cC)$. 
We denote by $\textbf{0}$ the all-$0$-vertex in $Q^d$, and by $\textbf{1}$ the all-$1$-vertex in $Q^d$.
For a vertex $v\in V(Q^d)$, we denote by $v(i)$ the $i$-th coordinate of $v$, and by $\mathrm{supp}(v)$ the \emph{support} of $v$, that is, set of indexes $i\in [d]$ such that $v(i)=1$. 

We use standard asymptotic notation. When the implicit constants therein depend on some parameter, we indicate this parameter as a lower right index: for example, $O_c$ or $\Omega_\eps$. Throughout the paper, we systematically ignore rounding signs as long as it does not affect the validity of our arguments.

\subsection{Concentration inequalities}
We start by presenting a version of the well-known Chernoff's bound for binomial random variables (see, for example, \cite[Theorem 2.1]{JLR00}).
\begin{lemma}\label{chernoff}
For a binomial random variable $X$ and any $t\in [0,\mathbb E[X]/2]$, 
\begin{align*}
\mathbb{P}\left(\big|X-\mathbb E[X]\big|\ge t\right)\le \exp\left(-\frac{t^2}{3\mathbb E[X]}\right).
\end{align*}
\end{lemma}

For a constant $C > 0$ and a domain $\Lambda = \Lambda_1\times\ldots\times \Lambda_m\subseteq \mathbb R^m$, a function $f: \Lambda\to \mathbb R$ is said to be \emph{$C$-Lipschitz} if, for every $i\in [m]$, $(z_j)_{j=1}^m\in \Lambda$ and $z_i'\in \Lambda_i$, we have
\[|f(z_1,\ldots, z_{i-1}, z_i, z_{i+1},\ldots, z_m) - f(z_1,\ldots, z_{i-1}, z_i', z_{i+1},\ldots , z_m)| \le C.\]
Next, we state a variant of the bounded difference inequality (see, e.g., Theorem 3.9 in~\cite{M98} and Corollary~6 in~\cite{War16}).

\begin{lemma}\label{azuma}
Fix $p\in [0,1]$ and a vector $X = (X_1,X_2,\ldots, X_m)$ of independent Bernoulli$(p)$ random variables.
Fix $C>0$ and a $C$-Lipschitz function $f:\{0,1\}^m\to \mathbb R$. Then, for every $t\ge 0$,
\begin{align*}
\mathbb{P}\left(\big|f(X)-\mathbb{E}\left[f(X)\right]\big|\ge t\right)\le 2\exp\bigg(-\frac{t^2}{2C^2mp+2Ct/3}\bigg).
\end{align*}
\end{lemma}

We end this section with a switching lemma reminiscent of~\cite[Theorem~2.19]{Wor99}.
Fix integers $k,d\ge 1$ and denote by $\cS_d(k)$ the family of sequences of length $kd$ where every element in $[d]$ appears $k$ times.
A \emph{switching} consists of exchanging the positions of two elements in a sequence. 
For two sequences $\sigma_1,\sigma_2\in \cS_d(k)$, we write $\sigma_1\sim \sigma_2$ if they differ by a single switching.

\begin{lemma}\label{lem:switchings}
Fix $c>0$ and a function $f:\cS_d(k)\to\mathbb R$. Suppose that, for every pair $\sigma_1,\sigma_2\in \cS_d(k)$ with $\sigma_1\sim \sigma_2$, we have $|f(\sigma_1)-f(\sigma_2)|\le c$. 
Let $\sigma$ be chosen uniformly at random from $\cS_d(k)$. Then, for every $t\ge 0$,
\[\mathbb P(|f(\sigma)-\mathbb E[f(\sigma)]|\ge t)\le 2\exp\bigg(-\frac{t^2}{2c^2kd}\bigg).\]
\end{lemma}
\begin{proof}
Consider the map $\psi:\cS_{kd}(1)\to \cS_{d}(k)$ which takes a sequence $\pi'\in \cS_{kd}(1)$ and outputs a sequence $\pi \in \cS_{d}(k)$ obtained by replacing the occurrences of $j+d,\cdots,j+(k-1)d$ in $\pi'$ with $j$ for every $j\in [d]$. 
For every element $\pi\in \cS_{d}(k)$ there exist exactly $(k!)^d$ elements $\cS_{kd}(1)$ which are mapped to $\pi$. 
Thus, defining $\sigma'$ to be a uniformly chosen sequence in $\cS_{kd}(1)$, it suffices to prove that, for every $t\geq 0$, 
\[\mathbb P(|f(\psi(\sigma'))-\mathbb E[f(\psi(\sigma'))]|\ge t)\le 2\exp\bigg(-\frac{t^2}{2c^2kd}\bigg).\]

For every $i\in [0,kd]$, set $X_i=\mathbb E[f(\psi(\sigma'))\mid \sigma'(1),\ldots,\sigma'(i)]$ and note that $(X_i)_{i=0}^{kd}$ is a martingale. We show that, 
\begin{equation}\label{eq:lip}
\text{for every $i\in [kd]$,}\qquad |X_i-X_{i-1}|\le c.
\end{equation}
The lemma follows by combining \eqref{eq:lip} and Azuma's inequality (see e.g.\ \cite[Theorem~2.25]{JLR00}).

Observe that, 
\begin{align*}
    X_{i-1}
    &=\mathbb E[X_i\mid \sigma'(1),...,\sigma'(i-1)]\\
    &=\sum_{s\in [kd]\setminus \{\sigma'(j):1\leq j<i\}}\mathbb E[X_i\mathds{1}_{\sigma'(i)=s}\mid \sigma'(1),...,\sigma'(i-1)]\\
    &=\sum_{s \in [kd]\setminus \{\sigma'(j):1\leq j<i\}} \mathbb P(\sigma'(i)=s|\sigma'(1),...,\sigma'(i-1))\cdot\mathbb E[X_i|\sigma'(1),...,\sigma'(i-1),\sigma'(i)=s]\\
    &=\frac{1}{kd-i+1}\sum_{s \in [kd]\setminus \{\sigma'(j):1\leq j<i\}} \mathbb E[X_i|\sigma'(1),...,\sigma'(i-1),\sigma'(i)=s]
\end{align*}
As a result, it suffices to show that the difference between the minimal and the maximal term in the latter sum is bounded from above by $c$.
To this end, fix any $\sigma'(1),\ldots,\sigma'(i-1)\subseteq [kd]$ as well as two elements $s',s''\in [kd]\setminus \{\sigma'(j):1\leq j<i\}$. Denote by $\cS'$ (resp. $\cS''$) the set of sequences in $\cS_{kd}(1)$ which start with $\sigma(1),\ldots,\sigma(i-1)$ and $\sigma(i)=s'$ (resp. $\sigma(i)=s''$). Then, the map $\phi: \cS'\to \cS''$ which takes a string $\sigma'\in\cS'$ and outputs the string obtained after switching the occurrences of $s'$ and $s''$ is a bijection. 
Further, for every $\sigma'\in \cS'$, we have that $\psi(\sigma')$ and $\psi(\phi(\sigma'))$ differ by a single switching. Thus,  $|f(\psi(\sigma'))-f(\psi(\phi(\sigma'))|\le c$ for every $\sigma'\in \cS'$.  
As a consequence, for every $i\in [kd]$ and every choice of $\sigma'(1),\ldots,\sigma'(i-1),s',s''$ as described, we have that $$|\mathbb E[X_i|\sigma'(1),...,\sigma'(i-1),\sigma'(i)=s']-\mathbb E[X_i|\sigma'(1),...,\sigma'(i-1),\sigma'(i)=s'']|\le c.$$
In turn, this implies~\eqref{eq:lip} and finishes the proof.
\end{proof}

\subsection{Auxiliary graph-theoretic results} The following decomposition lemma appears as Lemma~2.1 in \cite{ADILS25}.
For graphs $H\subseteq G$ and a function $w:V(G)\to \mathbb R$, we write $w(H)=\sum_{v\in V(H)} w(v)$.

\begin{lemma}\label{l: decomps}
Fix $\Delta \ge 2$, $m_0>0$ and set $m \coloneqq(\Delta+1)m_0$. 
Consider a tree $T$ of maximum degree at most $\Delta$ equipped with a weight function $w:V(T)\to (0,m_0]$ satisfying $w(T)\ge m_0$. Then, there exist vertex-disjoint trees $T_1,\ldots, T_k$ such that $V(T)=\bigsqcup_{i\in[k]} V(T_i)$ and, for every $i\in [k]$, $w(T_i)\in [m_0, m]$. 
\end{lemma}

Next, we state an estimate on the number of subtrees of a $d$-regular graph with fixed order and rooted in a particular vertex (see, e.g., \cite[Lemma 2]{BFM98} and \cite[Chapter 7]{K98}).
\begin{lemma}\label{l: trees}
Fix a $d$-regular graph $G$, a vertex $v$ therein and an integer $k\in [d]$.
Then, the number $t(v,k)$ of $k$-vertex subtrees of $G$ rooted in $v$ satisfies
\begin{align*}
\frac{k^{k-2}(d-k)^{k-1}}{(k-1)!}\le t(v,k) \le \frac{k^{k-2} d^{k-1}}{(k-1)!}\le (\e d)^{k-1}.
\end{align*}
\end{lemma}
We utilise the above lemma to estimate the number of forests whose roots lie in a fixed set of vertices. 
\begin{corollary}\label{cor:forests}
Fix an $n$-vertex $d$-regular graph $G$ and a set of vertices $U\subseteq V(G)$ of size $\ell\ge 1$. 
Then, the number of forests $F\subseteq G$ on $k$ edges where $U\subseteq V(F)$ and no connected component of $F$ is disjoint from $U$ is bounded from above by $(\ell+k)^k d^k/k!$\,.
\end{corollary}
\begin{proof}
For every forest $F$ as described, one can assign roots in $U$ to the trees in $F$ as follows: every tree $T$ of $F$ is assigned a single root in $U$ and all vertices in $V(T)\cap U$ except for the root (if any such exist) are considered to be roots of the empty tree on zero vertices.
By combining the latter procedure and \Cref{l: trees}, we obtain that the desired number of forests is at most
\begin{equation}\label{eq:forests_count}
\bigg(\sum_{x_1+\ldots+x_\ell=k;\, x_1,\ldots,x_\ell\ge 0}\;\;
\prod_{i=1}^{\ell} \frac{(x_i+1)^{x_i-1}}{x_i!}\bigg) d^k,
\end{equation}
where we stress that $x_i$ is the number of \textit{edges} in the $i$-th tree.

To analyse the latter expression, consider the power series
\[f(z) = \sum_{j=0}^{\infty} \frac{(j+1)^{j-1}}{j!} z^j.\]
On the one hand,~\eqref{eq:forests_count} is obtained by multiplying $d^k$ and the coefficient of $(f(z))^{\ell}$ in front of $z^k$. 
On the other hand, by~\cite[(2.36)]{CGHJK96},
\[(f(z))^{\ell} = \sum_{j=0}^{\infty} \frac{\ell(\ell+j)^{j-1}}{j!} z^j.\]
Taking the coefficient in front of $z^k$ and using~\eqref{eq:forests_count} finishes the proof.
\end{proof}

Next, we state the celebrated Harper's edge-isoperimetric inequality for hypercubes, see~\cite{H64} and also \cite{B67,H76,L64}.
\begin{lemma}\label{l: Harper}
For every integer $d \ge 1$ and for every set $S\subseteq V(Q^d)$,
\begin{align*}
e(S,V(Q^d)\setminus S)\ge |S|\left(d-\log_2|S|\right).
\end{align*}
\end{lemma}

\subsection{Auxiliary results on percolation}
We start this section by stating the well-known Harris' inequality, also known as the FKG inequality, see \cite[Theorem 6.3.3]{AS16}.
For a finite set $\Lambda$, we say that a family $\cA$ of subsets of $\Lambda$ is \emph{increasing} if, for all subsets $A\subseteq B$ of $\Lambda$ with $A\in \cA$, we also have that $B\in \cA$.
Denote by $\Lambda_p$ a random subset of $\Lambda$ containing every element independently and with probability $p$.

\begin{lemma}[Harris' inequality]\label{lem:Harris}
For every $p\in [0,1]$ and increasing subsets $\cA,\cB$, $$\mathbb P(\Lambda_p\in \cA\cap \cB)\ge \mathbb P(\Lambda_p\in \cA)\mathbb P(\Lambda_p\in \cB).$$
\end{lemma}

Next, we present several lemmas on percolated hypercubes. The first one deals with matchings in $Q^d_p$ and appears as \cite[Lemma 2.9]{EKK22}.
\begin{lemma}\label{l: matchings}
Fix $\delta>0$, let $p=p(d)=\delta/d$ and $t\in [d2^d]$. 
There exists a constant $\alpha\coloneqq\alpha(\delta)>0$ such that, for every set $F\subseteq E(Q^d)$ of size $|F|\ge t$, $F\cap Q^d_p$ contains a matching of size at least $\alpha t/d$ with probability at least $1-\exp(-\alpha t/d)$. 
\end{lemma}

We also make use of some expansion properties of connected subsets in the giant component established more generally in \cite[Theorem~4]{DEKK24}, see also the remark in the beginning of page 747 therein.

\begin{thm}\label{thm: old expansion}
Fix $c>1$, let $p=p(d)=c/d$ and recall the giant component $L_1$ in $Q^d_p$. 
Then, there are constants $\eps_1\in (0,1),\eps_2,K>0$ (all depending only on $c$) such that \textbf{whp}, for all subsets $S\subseteq V(L_1)$ satisfying that $Q^d_p[S]$ is connected:
\begin{enumerate}[label=\upshape(\alph*)]
    \item if $|S| \in [Kd, n^{\eps_1}]$, then
    $$|N_{Q^d_p}(S)|\ge \eps_2|S|;$$
    \item if $|S| \in [n^{\eps_1}, 2v(L_1)/3]$, then
    $$e_{Q^d_p}(S, V(L_1)\setminus S)\ge \frac{\eps_2|S|\log_2(n/|S|)}{d\log d}.$$
\end{enumerate}
\end{thm}
We note that we will extend and improve \Cref{thm: old expansion}(b) in \Cref{s: mixing time}, see \Cref{cor: new expansion}.

We end this section with a `sparsification' lemma which will be key in our subsequent analysis.

\begin{lemma}\label{l: sparsification}
Fix a family $\cC$ of vertex-disjoint sets in $Q^d$ where $v(\cC)\ge d^{50}$ and each set in $\cC$ has size at most $d^{10}$.
Then, for any sufficiently small constant $\eps>0$, there is a family $\cC'\subseteq \cC$ satisfying that
\begin{align*}
    v(\cC')\ge \eps v(\cC)\quad \text{ and }\quad e_{\mathrm{out}}(\cC')\ge (1-2\eps) dv(\cC').
\end{align*}
\end{lemma}
\begin{proof}
Fix $\eps_0 \in (0,1)$ and define $\hat \cC$ to be a random subset of $\cC$ where every set is included independently with probability $\eps_0$. 
Note that $v(\hat\cC)$ is a $d^{10}$-Lipschitz function of the random vector $(\mathds{1}_{C\in \hat\cC})_{C\in \cC}$ and thus, given $\cC$ with $v(\cC)\ge d^{50}$, by the bounded difference inequality (\Cref{azuma}), \whp $v(\hat \cC)=(1+o(1))\eps_0v(\cC)$.
Moreover, by Lemma \ref{l: Harper} and the assumption that all sets in $\cC$ have size at most $d^{10}$, for every $C\in \cC$,
\begin{align*}
    e(C,V(Q^d)\setminus C)\ge (d-\log_2(d^{10}))|C|.
\end{align*}
Thus,
\begin{align}
    e_{\mathrm{out}}(\cC)+2e_{\mathrm{in}}(\cC)=\sum_{C\in\cC}e(C,V(Q^d)\setminus C)\ge (d-10\log_2d)v(\cC).\label{eq: estimae on cc}
\end{align}
Moreover,
\begin{equation}\label{eq:Ee_out}
\begin{split}
\mathbb{E}[e_{\mathrm{out}}(\hat\cC)]=\eps_0\cdot e_{\mathrm{out}}(\cC)+2\eps_0(1-\eps_0)\cdot e_{\mathrm{in}}(\cC)
&\ge \eps_0(1-\eps_0)\left(e_{\mathrm{out}}(\cC)+2e_{\mathrm{in}}(\cC)\right)\\
&\ge \eps_0(1-\eps_0)(d-10\log_2d)v(\cC),
\end{split}
\end{equation}
where the second inequality follows from \eqref{eq: estimae on cc}. 
Further, note that $e_{\mathrm{out}}(\hat\cC)$ is $d^{11}$-Lipschitz function of the random vector $(\mathds{1}_{C\in \hat\cC})_{C\in \cC}$ and thus, by~\eqref{eq:Ee_out} and the bounded difference inequality, \whp $e_{\mathrm{out}}(\hat\cC) \ge (1-o(1))\eps_0(1-\eps_0)dv(\cC)$.
In particular, there exists a set $\cC'\subseteq\cC$ satisfying
\begin{align*}
    v(\cC') = (1-o(1))\eps_0v(\cC)\quad \text{and}\quad e_{\mathrm{out}}(\cC') \ge (1-o(1))(1-\eps_0)\eps_0dv(\cC) = (1-o(1))(1-\eps_0)dv(\cC').
\end{align*}
Choosing $\eps\in (\eps_0/2,\eps_0)$, completes the proof.
\end{proof}

\subsection{Branching processes and the BFS algorithm}\label{BFS description}
We begin with a lemma quantifying the extinction probability of a supercritical Galton-Watson process with large total progeny. 
Recall that, for a probability distribution $\mu$ on non-negative integers, a \emph{Galton-Watson process with offspring distribution $\mu$} is a branching process starting from a root where every vertex produces a random number of children with distribution $\mu$ independently of other vertices.
The following lemma is a simplified version of \cite[Theorem~3.8]{vdH18}.

\begin{lemma}\label{lem:branch}
Fix a Galton-Watson process $T$ with offspring distribution $\mu$ with mean $\mathbb E\mu > 1$.
Then, there is a constant $I = I(\mu) > 0$ such that, for every $k\ge 0$, $\mathbb P(k\le v(T) < \infty)\le \e^{-Ik}/(1-\e^{-I})$.
\end{lemma}

An important parameter describing a Galton-Watson branching process is its \emph{survival probability} given by the probability that the process contains an infinite number of vertices.
Following the introduction, we denote by $y(c)$ the survival probability of a Galton-Watson process with offspring distribution Poisson($c$) and by $b(n,p)$ the survival probability of a Galton-Watson process with offspring distribution $\mathrm{Bin}(n,p)$.
The following lemma compares these survival probabilities.

\begin{lemma}\label{lem:yb}
Fix $c>1$. We have $y(c-\eps)-y(c)=O_c(\eps)$ when $\eps\to 0$, and $y(c)-b(d,c/d) = O_c(1/d)$.
\end{lemma}
\begin{proof}
It is a classic fact (see e.g.\ \cite[Theorem~3.1]{vdH18}) that the extinction probability of a Galton-Watson process with distribution $\mu$ is given by the smallest solution in $[0,1]$ of the equation
\[q=f_\mu(q)\qquad\text{where}\qquad f_\mu:t\in [0,1]\mapsto \sum_{i=0}^{\infty} \mu(i)t^i.\]
In particular, $\bar y(c) := 1-y(c)$ is the smallest solution in $[0,1]$ to $f(c,q)=q$ where $f(c,q):=\e^{c(q-1)}$ while $\bar b(d,c/d) := 1-b(d,c/d)$ is the smallest solution in $[0,1]$ to $f_d(c,q)=q$ where $f_d(c,q):=(1-c/d+cq/d)^d$.
To establish the first equality in the statement of the lemma, we use the implicit function theorem (see e.g.\ \cite[Theorem 2-12]{Spi65}) to show that $\bar y$ is a differentiable function on the interval $(1,\infty)$. To this end, it suffices to verify that
\begin{equation}\label{eq:cross}
\frac{\partial}{\partial q}(f(c,q)-q)=c\e^{c(q-1)}-1 \neq 0\qquad \text{whenever $c>1$ and $q$ satisfy}\qquad \e^{c(q-1)}=q.
\end{equation}
Assuming for contradiction that both equalities hold would mean that $cq=1$, further implying that $c\e^{1-c}=1$. However, immediate analysis shows that $c\e^{1-c}<1$ for every $c>1$. As a result, $\bar y$ (and hence $y$ as well) is a differentiable function of $c$ on the interval $(1,\infty)$. The first equality follows.

Furthermore, it is immediate that $(f_d)_{d\ge 1}$ and $(\tfrac{\partial}{\partial q}f_d)_{d\ge 1}$ approximate uniformly $f$ and $\tfrac{\partial}{\partial q}f$ on the interval $q\in [0,1]$ with $\|f(c,\cdot)-f_d(c,\cdot)\|_{\infty} = O_c(1/d)$ and $\|\tfrac{\partial}{\partial q}f(c,\cdot)-\tfrac{\partial}{\partial q}f_d(c,\cdot)\|_{\infty} = O_c(1/d)$.
Moreover, for every $c>1$,
\begin{align*}
\bar b(d,c/d) - \bar y(c) 
&= f_d(c,\bar b(d,c/d))-f(c,\bar y(c))\\
&= f_d(c,\bar b(d,c/d))-f(c,\bar b(d,c/d))+f(c,\bar b(d,c/d))-f(c,\bar y(c))\\
&= O_c(1/d)+(\bar b(d,c/d) - \bar y(c))\frac{\partial}{\partial q}f(c,\bar y(c)) + o(\bar b(d,c/d) - \bar y(c)).
\end{align*}
Since $\frac{\partial}{\partial q}f(c,\bar y(c))\neq 1$ by~\eqref{eq:cross}, the second equality in the statement of the lemma follows.
\end{proof}

Next, we present a variant of the Breadth First Search (BFS) algorithm, which generates a random subgraph $G_p$ of a graph $G$ while exploring a spanning forest of $G_p$. The algorithm is fed a graph $G=(V,E)$ with an ordering $\sigma$ on its vertices, and a sequence $(X_e)_{e\in E}$ of i.i.d. Bernoulli$(p)$ random variables (with $p\in [0,1]$). 
We maintain three (dynamic) sets of vertices partitioning $V$ at any time: $S$, the set of vertices whose exploration was completed; 
$U$, the set of vertices currently explored and kept in a queue processed following a first-in-first-out discipline; 
and $T$, the set of vertices not yet to be explored by the process. We initialise $S=U=\varnothing$ and $T=V$. The algorithm terminates once $U\cup T=\varnothing$. The algorithm proceeds as follows.
\begin{enumerate}
    \item If $U$ is empty, move the first (according to $\sigma$) vertex $v$ in $T$ to $U$.\label{first}
    \item Otherwise, if $U\neq \varnothing$, let $u$ be the first (according to the queue ordering) vertex in $U$.
    \begin{enumerate}
        \item If $u$ has no neighbours in $T$, move $u$ from $U$ to $S$ and return to \eqref{first}.\label{second}
        \item Otherwise, let $v$ be the first vertex in $T$ according to $\sigma$ such that $e=uv\in E$ and $X_e$ has not been queried.
        \begin{enumerate}
            \item If $X_e=0$ and return to \eqref{second}.
            \item If $X_e=1$, we move $v$ from $T$ to $U$ and return to \eqref{second}.
        \end{enumerate}
    \end{enumerate}
\end{enumerate}

\section{Upper tail}\label{s: upper tail}
In this section, we fix $c>1$, let $p=p(d):=c/d$ and denote by $L_1$ the giant component of $Q^d_p$. 
We further abbreviate $n\coloneqq 2^d$ for the number of vertices in $Q^d$. 

Recall the survival probability $y(c)$. It is a classic fact~\cite{Dwa69,Ott49,Pit97} that a Galton-Watson process with offspring distribution Poisson($c$) has order $k\ge 1$ with probability $\e^{-ck}(ck)^{k-1}/k!$, and is infinite with probability
\begin{equation}\label{eq: ER}
    y(c) := 1-\sum_{k=1}^{\infty}\e^{-ck}\frac{(ck)^{k-1}}{k!}.
\end{equation}
We note that, when $c\in [0,1]$, the weights $\e^{-ck}(ck)^{k-1}/k!$ describe the \emph{Borel probability distribution} over the positive integers and $y(c)=0$.

We are now ready to prove the upper tail estimate promised in \Cref{th: large tail deviation}.

\begin{proof}[Proof of the upper tail estimate in Theorem \ref{th: large tail deviation}]

For every $k\in [d]$, denote by $V_k$ the set of vertices in components of order $k$ in $Q^d_p$.
To estimate $\mathbb E[|V_k|]$, fix a vertex $v$ and list all trees $T_1,\ldots,T_r\subseteq Q^d$ of order $k$ rooted at $v$. Then,
\[\mathbb P(v\in V_k) = \sum_{i=1}^r \mathbb P\bigg(\{T_i\subseteq Q^d_p\}\cap \bigcap_{j=1}^{i-1} \{T_j\not\subseteq Q^d_p\}\bigg)\ge r p^{k-1} (1-p)^{kd},\]
where the last inequality holds since each event requires $k-1$ open edges and less than $kd$ closed edges.
Choosing $k=k(d)\in [1,d^{1/4}]$ and combining the latter bound with \Cref{l: trees}, we obtain that
\begin{align*}
    \mathbb{E}[|V_k|]&\ge n\cdot \frac{k^{k-2}(d-k)^{k-1}}{(k-1)!} p^{k-1}(1-p)^{dk} = (1-O(d^{-1/2})) n\cdot \e^{-ck}\frac{(ck)^{k-1}}{k!},
\end{align*}
where we used that $(d-k)^{k-1}p^{k-1} = (1-O(d^{-1/2}))c^{k-1}$ and $(1-p)^{kd} = (1-O(d^{-1/2}))\e^{-ck}$, with the implicit constant in the $O$-terms being independent of $k$.
Now, for every integer $t\in [n]$, define
\[f(t,n):= \left\lceil \frac{1}{-\log(\e c\cdot \e^{-c})} \log\bigg(\frac{4n}{t(1-\e c\cdot \e^{-c})}\bigg) \right\rceil.\] 
Then, for every $t \in [n]$, 
\begin{align*}
\sum_{k=f(t,n)}^{\infty}\e^{-ck}\frac{(ck)^{k-1}}{k!} &\leq  \sum_{k=f(t,n)}^{\infty} \frac{\e^{-ck}(ck)^{k}}{(k/\e)^k} = \sum_{k=f(t,n)}^{\infty} (\e c\cdot \e^{-c})^k \leq \frac{(\e c\cdot \e^{-c})^{f(t,n)}}{1-\e c\cdot \e^{-c}}
\le \frac{t}{4n},
\end{align*}
where we used that $\e c\cdot \e^{-c}<1$ when $c>1$ for the second inequality. By combining the latter chain of inequalities with~\eqref{eq: ER}, 
we obtain that, for all $t\in [n/d^{0.1},n]$
\begin{align*}
\mathbb{E}\left[\sum_{k=1}^{f(t,n)}|V_k|\right]&\ge (1-O(d^{-1/2}))n\bigg[\sum_{k=1}^{\infty}\e^{-ck}\frac{(ck)^{k-1}}{k!}
- \sum_{k=f(t,n)+1}^{\infty}\e^{-ck}\frac{(ck)^{k-1}}{k!}\bigg]
\\&\geq \left(1-y(c)-O(d^{-1/2})\right)n- \frac{t}{4}.
\end{align*}
Now, fix $t\in [n/d^{0.1},n]$ and set $Z=\sum_{k=1}^{f(t,n)}|V_k|$. Since adding or removing one edge from $Q^d_p$ can change the value of $Z$ by at most $2f(t,n)$, we have that, by the bounded difference inequality (Lemma \ref{azuma}), for every $m\ge 1$,
\begin{align*}
    \mathbb{P}\left(Z\le \mathbb EZ-m\right)\le \exp\bigg(-\frac{m^2}{2 (2f(t,n))^2 cn + 2 (2f(t,n)) m/3}\bigg).
\end{align*}
Therefore, for all $t\in [n/d^{0.1},n]$,
\begin{align*}
\mathbb{P}(v(L_1)\ge y(c)n+t) 
&\le \mathbb{P}(Z\le \mathbb EZ - (t-t/4-O(d^{-1/2}n))\le \mathbb{P}(Z\le \mathbb EZ-t/2)\\
&\le\exp\bigg(-\frac{(t/2)^2}{2 (2f(t,n))^2 cn + 2 (2f(t,n)) (t/2)/3}\bigg)\leq \exp\bigg(-\frac{\eps t^2}{n (\log(n/t))^2}\bigg),
\end{align*}
where $\eps=\eps(c)>0$ is a constant depending only on $c$, as required.
\end{proof}

\section{Lower tail}\label{s: lower tail}
In this section, we set $p=p(d)=c/d$ and fix a suitably small constant $\delta = \delta(c)\in (0,1)$ (such that, in particular, $c-\delta>1$).
We also set $p_1=p_1(d)=(c-\delta)/d$ and $p_2$ such that $(1-p_1)(1-p_2)=1-p$; note that $p_2d\ge \delta$. 
Further set $G_1\coloneqq Q^d_{p_1}$ and $G_2\coloneqq G_1\cup Q^d_{p_2}$, where we stress that $Q^d_{p_2}$ is sampled independently from $G_1$. Note that $G_2$ has the same distribution as $Q^d_p$. We further abbreviate $n:=2^d$.

Before proving the bound on the lower tail in Theorem \ref{th: large tail deviation}, some preparation is necessary. This section is structured as follows. 
First, in Section \ref{subsec: vol}, we establish tail bounds on the number of vertices in components of order at least $d^{10}$ in $Q^d_p$. Then, in Section \ref{subsec: typical}, we establish quantitative bounds on the probability of certain typical structural properties of $G_1$ and describe the interaction between $G_1$ and $G_2$. Finally, in Section \ref{subsec: lower bound proof}, we prove the bound on the lower tail in Theorem \ref{th: large tail deviation}.

\subsection{\texorpdfstring{Vertices in components of order at least $d^{10}$}{Vertices in large components}}\label{subsec: vol}
In this section, for every $\alpha \ge 0$, denote by $\cM_\alpha$ the set of components in $G_2$ of order at least $d^{\alpha}$.
Our first lemma studies $\cM_{1/4}$.
\begin{lemma}\label{l: volume of large}
For every $t\in [n/d^{1/9},n]$, we have
$\mathbb{P}(|v(\cM_{1/4}) - y(c)n|\ge t)\le \exp(-t^2/(40cd^{1/2}n))$.
\end{lemma}
\begin{proof}
Let $v\in V(Q^d)$.
We first estimate the probability that $v$ is in a component in $\cM_{1/4}$. 
To this end, we explore $Q^d_p$ by running the BFS algorithm from \Cref{BFS description} with the following modifications: we start the algorithm with $U=\{v\}$ and terminate the process when either $U$ becomes empty or $|S\cup U|$ contains at least $d^{1/4}$ vertices. 
Note that $v\in V(\cM_{1/4})$ if and only if the BFS process terminates due to the second stopping condition. 
We let $T_1$ be the tree rooted at $v$ which is constructed by the BFS. Note that $T_1$ has at most $d^{1/4}$ vertices.

Since $|S\cup U|\le d^{1/4}$ throughout the execution of the algorithm, at all times, every vertex $u\in U$ satisfies $|N_{Q^d}(u)\setminus (S\cup U)|\ge d-d^{1/4}$. 
As a result, we can (and do) couple the process $T_1$ with Galton-Watson processes $T_2$ and $T_3$ rooted at $v$ with offspring distributions respectively $\mathrm{Bin}(d-d^{1/4},p)$ and $\mathrm{Bin}(d,p)$ so that:
\begin{itemize}
    \item if $v(T_2)=\infty$, then $v(T_1)=d^{1/4}$, and
    \item if $v(T_1)=d^{1/4}$, then $v(T_3)\ge d^{1/4}$.
\end{itemize}

Now, on the one hand, by \Cref{lem:yb}, we have 
\[\mathbb{P}(v\in V(\cM_{1/4}))\ge \mathbb{P}(|v(T_2)|=\infty) \ge y((d-d^{1/4})p)-O_c(1/d)\ge y(c)-O_c(d^{-3/4}).\] 
On the other hand, by \Cref{lem:branch} and \Cref{lem:yb}, there is $I=I(c)>0$ such that 
\[\mathbb{P}(v\in V(\cM_{1/4}))\le \mathbb P(|T_3|\ge d^{1/4})=b(d,c/d)+O_c(\e^{-Id^{1/4}})=y(c)+O_c(1/d).\]
Combining the latter bounds implies that $|\mathbb E[v(\cM_{1/4})]-y(c)n|=O_c(n/d^{3/4})$.
Furthermore, note that adding or removing one edge from $Q^d_p$ can change $v(\cM_{1/4})$ by at most~$2d^{1/4}$.
Thus, by the bounded difference inequality (Lemma \ref{azuma}), for every $t\in [n/d^{1/9},n]$, we have
\begin{align*}
\mathbb{P}(|v(\cM_{1/4})-y(c)|\ge t)&\le \mathbb{P}(|v(\cM_{1/4})- \mathbb{E}[v(\cM_{1/4})]|\ge t/2)\\
&\le 2\exp\bigg(-\frac{t^2/4}{2 (2d^{1/4})^2\cdot (nd/2)\cdot p + 2 (2d^{1/4})\cdot t/6)}\bigg)\le \exp\bigg(-\frac{t^2}{40cd^{1/2}n}\bigg),
\end{align*}
as required.
\end{proof}

Next, we turn our attention to the family $\cM_{10}\setminus \cM_{1/4}$.

\begin{lemma}\label{l: odd size}
There is a constant $\eps=\eps(c)>0$ such that
$\mathbb{P}(v(\cM_{10}\setminus \cM_{1/4})\ge n/d^{1/9})\le \exp(-\eps n/d^{1/9})$.
\end{lemma}
\begin{proof}
Given a family $\cC$ of vertex-disjoint connected sets of order at most $d^{10}$ with $v(\cC)\ge n/d^{1/9}$,  \Cref{l: sparsification} implies that, for any suitably small $\eps_1>0$, there is a subfamily $\cC'\subseteq \cC$ such that $v(\cC')\ge \eps_1 n/d^{1/9}$ and $e_{\mathrm{out}}(\cC')\ge (1-2\eps_1)v(\cC')$. 
It thus suffices to estimate the probability that there exists such a family $\cC'$ of components of $G_2$ where each component in the family has order between $d^{1/4}$ and $d^{10}$.

To this end, fix $k\ge 1$ and integers $x_1,\ldots, x_k\in [d^{1/4},d^{10}]$ of sum $x\coloneqq\sum_{i=1}^{k}x_i\in[\eps_1n/d^{1/9},n]$. 
By Lemma \ref{l: trees}, the number of forests in $Q^d$ containing $k$ trees of orders $x_1,\ldots, x_k$ is bounded from above by $\binom{n}{k} \prod_{i=1}^{k}(\e d)^{x_i-1} = \binom{n}{k}(\e d)^{x-k}$. 
At the same time, the probability that each tree in such a forest $F$ spans a connected component of $G_2$ is bounded from above by $p^{x-k}(1-p)^{e_{\mathrm{out}}(F)}\le p^{x-k}(1-p)^{(1-2\eps_1)xd}$. 
Thus, by the union bound, the probability that a family $\cC'$ with the described properties exists is at most

\begin{align}  
\sum_{x=\eps_1n/d^{1/9}}^n&\sum_{\substack{x_1+\cdots+x_k=x\\x_1,\ldots,x_k\in[d^{1/4},d^{10}]}}\binom{n}{k}(\e d)^{x-k}p^{x-k}(1-p)^{(1-2\eps_1)xd}\nonumber\\
    &\le \sum_{x=\eps_1n/d^{1/9}}^n\sum_{\substack{x_1+\cdots+x_k=x\\x_1,\ldots,x_k\in[d^{1/4},d^{10}]}}\binom{n}{d^{-1/4}x}\exp((\log\left(\e c\right)-(1-2\eps_1)c)x)\nonumber\\
    &\le \sum_{x=\eps_1n/d^{1/9}}^n\sum_{\substack{x_1+\cdots+x_k=x\\x_1,\ldots,x_k\in[d^{1/4},d^{10}]}}\left(\frac{\e n}{d^{-1/4}x}\right)^{d^{-1/4}x}\exp((\log\left(\e c\right)-(1-2\eps_1)c)x),\label{eq:largeUB}
\end{align}
where the first inequality uses that $k\le d^{-1/4}x$. To simplify~\eqref{eq:largeUB}, note that there are at most $(d^{10})^{d^{-1/4}x}$ ways to choose $x_1,\ldots, x_k\in[d^{1/4},d^{10}]$ with $\sum_{i=1}^{k}x_i=x$.
Furthermore, by choosing $\eps_1$ suitably small with respect to $c$ and using that $c>1$, we have that $c_0 := (1-2\eps_1)c - \log(\e c) > 0$.
Thus,~\eqref{eq:largeUB} is at most
\begin{align*}
\sum_{x=\eps_1 n/d^{1/9}}^n\left(\frac{\e nd^{10+1/4}}{x}\right)^{d^{-1/4}x}\e^{-c_0x}\le \sum_{\eps_1n/d^{1/9}}^n\exp\bigg(x\left(\frac{1}{d^{1/4}}\log\left(\frac{\e nd^{41/4}}{x}\right)-c_0\right)\bigg).
\end{align*}
Now, since $x\ge\eps_1n/d^{1/9}$, we have $\log(\e nd^{10}/x)=O(\log d)$. Altogether, we obtain
\begin{align*}
\mathbb P(v(\cM_{10}\setminus \cM_{1/4})\ge n/d^{1/9})\le  \sum_{\eps_1n/d^{1/9}}^n\exp(-xc_0/2)\le n\exp(-\eps_1c_0n/(2d^{1/9})).
\end{align*}
Thus, setting $\eps = \eps_1 c_0/3$ finishes the proof.
\end{proof}

With the last two lemmas at hand, we conclude this subsection with the following deviation bound for $\cM_{10}$.

\begin{corollary}\label{l: W(p)}
For every $t\in [n/d^{1/9},n]$, $\mathbb{P}(|v(\cM_{10})-y(c)n|\ge 2t)\le 2\exp(-t^2/(40cd^{1/2}n))$.
\end{corollary}
\begin{proof}
Since $v(\cM_{10}) = v(\cM_{1/4})-v(\cM_{10}\setminus \cM_{1/4})$, Lemmas~\ref{l: volume of large} and~\ref{l: odd size} yield
\begin{align*}
\mathbb{P}(|v(\cM_{10})-y(c)n|\ge 2t)&\le \mathbb{P}(\{|v(\cM_{1/4})-y(c)n|\ge t\}\cup\{v(\cM_{10}\setminus \cM_{1/4})\ge t\})\\
&\le \mathbb{P}(|v(\cM_{1/4})-y(c)n|\ge t)+\mathbb{P}(v(\cM_{10}\setminus \cM_{1/4})\ge n/d^{1/9})\\
&\le \exp(-t^2/(40cd^{1/2}n))+\exp(-\eps n/d^{1/9})\le 2\exp(-t^2/(40cd^{1/2}n)).\qedhere
\end{align*}
\end{proof}

\subsection{\texorpdfstring{Typical properties of $G_1$ and $G_2$}{Typical properties of G1 and G2}}\label{subsec: typical}
The idea that `the giant component is well-spread' in the percolated hypercube has been used (sometimes implicitly) in many previous works \cite{AKS81,BKL92,DEKK24,EKK22}. The first lemma in this subsection gives a quantitative characterisation of this spreadness.

Recall the graphs $G_1$ and $G_2$ defined in the beginning of \Cref{s: lower tail}. 
Also, recall the families $(\cM_{\alpha})_{\alpha\ge 0}$ from \Cref{subsec: vol} and define $\cM_{\alpha}^-$ to be the set of components of order at least $d^{\alpha}$ in $G_1$.
For $\eps > 0$, we define the set of \emph{$\eps$-bad vertices} $\vbad = \vbad(\eps)$ to be the set of vertices in $Q^d$ with less than $\eps d$ neighbours (in $Q^d$) in components in $\cM_2^-$.
Further, we set $\cM_{[0,2)}^-:=\cM_0^-\setminus \cM_2^-$, that is, the family of components of order less than $d^2$ in $G_1$.

\begin{lemma}\label{l: large components are well spread}
For all suitably small constants $\eps = \eps(c,\delta)>0$, $\gamma = \gamma(\eps)>0$ and $\nu=\nu(c,\delta)>0$, for every $t\in [n^{1-\gamma},n]$, $\mathbb{P}(|\vbad(\eps)|\ge t)\le \e^{-\nu t}$.
\end{lemma}
\begin{proof}
Fix suitably small $\eps= \eps(c),\gamma=\gamma(\eps)>0$ to be determined in the sequel and $t\in [n^{1-\gamma},n]$.

\begin{claim}\label{cl:associate}
Suppose that $|\vbad(\eps)|\ge t$. Then, there exist a set $U\subseteq \vbad(\eps)$, pairwise disjoint families of components $(\cC_u)_{u\in U}$ in $G_1$, and a set of edges $E'\subseteq E(Q^d)$ satisfying each of the following properties:
\begin{itemize}
    \item for every $u\in U$, $V(\cC_u)\subseteq V(\cM_{[0,2)}^-)$,
    \item for every $u\in U$, each component in $\cC_u$ contains at least one vertex $v$ such that $uv\in E'$,
    \item for every $u\in U$, the components in $\cC_u$ together contain exactly $(1-2\eps)d$ vertices $v$ such that $uv\in E'$,
    \item the components in $\cC := \bigcup_{u\in U} \cC_u$ jointly cover at least $\eps t/2$ vertices.
\end{itemize}
\end{claim}
\begin{proof}[Proof of \Cref{cl:associate}]
For some $m\ge 1$, we construct a sequence of sets $U_0=\varnothing\subseteq U_1\subseteq \ldots \subseteq U_m = U$ with $|U_i|=i$ for every $i\in [m]$, a sequence of sets $W_0=\vbad(\eps)\supseteq W_1\supseteq \ldots \supseteq W_m=\varnothing$, and a sequence of pairwise disjoint families of components $\cC_0=\varnothing\subseteq \cC_1\subseteq \ldots \subseteq \cC_m = \cC$ with $v(\cC)\ge \eps t/2$ as follows. 

First, we underline that, throughout the process, vertices in $W_i$ have at least $(1-2\eps)d$ neighbours in $V(\cM_{[0,2)}^-\setminus \cC_i)$. Note that this holds by definition for $i=0$. We now continue assuming $i\ge 1$. At the $i$-th step, select an arbitrary vertex $u$ in $W_{i-1}$. Then, set $U_i=U_{i-1}\cup \{u\}$ and let $\cC_u\subseteq \cM_{[0,2)}^-$ be a family of components disjoint from $\cC_{i-1}$ where the following conditions jointly hold: 
\begin{itemize}
    \item each component contains at least one neighbour (in $Q^d$) of $u$, and
    \item there exists a set of edges $E'_u$ between $u$ and $\cC_u$ of size exactly $(1-2\eps)d$ such that each component $C\in \cC_u$ is incident to an edge from $E'_u$.
\end{itemize}
\noindent
Note that such a family exists by our assumption on $W_{i-1}$. Then, we set $\cC_i=\cC_{i-1}\cup \cC_u$ and define $W_i$ as the subset of vertices in $W_{i-1}$ with at most $\eps d$ neighbours in $V(\cC_i)$ (and, in extension, at least $(1-2\eps)d$ neighbours in $V(\cM_{[0,2)}^-\setminus \cC_i)$, as required). We stress here that, by construction, every vertex in $\vbad(\eps)\setminus W_i$ has at least $\eps d$ neighbours in $V(\mathcal{C}_i)$.
The procedure ends at the first step $m$ when $W_m=\varnothing$.

We set $E'=\bigcup_{u\in U} E'_u$. Now, since every vertex in $U_m$ has at least $(1-2\eps)d$ neighbours in $V(\cC_m)$, we have that $d v(\cC_m)\geq  (1-2\eps)d |U_m|$, implying that $v(\cC_m)\geq  (1-2\eps) |U_m|$. In addition, since every vertex in $\vbad(\eps)\setminus U_m$ has at least $\eps d$ neighbours in $V(\mathcal{C}_m)$, we have that 
$\eps d |\vbad(\eps)\setminus U_m|\leq d v(\cC_m)$, implying that $|\vbad(\eps)\setminus U_m|\leq  v(\cC_m)/\eps$. As a result,
$$
t-\frac{v(\cC_m)}{1-2\varepsilon}\leq t- |U_m| \leq |\vbad(\varepsilon)\setminus U_m|\leq\frac{v(\cC_m)}{\varepsilon }.
$$ 
Assuming that $\varepsilon < 0.1$, the above inequality implies that $v(\cC_m)\geq \eps t/2$, as required.
\end{proof}

Next, conditionally on the event $|\vbad(\eps)|\ge t$ and on the graph $G_1$, fix a set $U$, a family $\cC$ and a set $E'$ as in \Cref{cl:associate}. 
Recall that, for each $u\in U$, every $C\in\mathcal{C}_u$ is adjacent to $u$, and $v(C)\le d^2$. Therefore $v(\mathcal{C}_u)\leq d^3$. 
Thus, we can apply \Cref{l: sparsification} to the family of vertex disjoint sets $(V(\cC_u))_{u\in U}$, and obtain that, for any sufficiently small $\eta > 0$, there is a set $U'\subseteq U$ such that $\cC' =\bigcup_{u\in U'} \cC_u$ satisfies
\begin{align}\label{eq:C'}
v(\cC')\ge \eta v(\cC)\ge \eta \eps t/2\quad \text{ and }\quad e_{\mathrm{out}}(\cC')\ge (1-2\eta) dv(\cC').
\end{align}
Fix such sets $U'\subseteq U$ and $\cC'\subseteq \cC$ for a suitably small $\eta=\eta(c)>0$ to be determined in the sequel.

Set $d':=(1-2\eps)d$ and $m\coloneqq |U'|$. For every $u\in U'$, denote by $R_u$ the set of $E'$-neighbours of $u$ in $V(\cC_u)$ and set $R=\bigcup_{u\in U'}R_u$. 
Observe that $G_1[V(\cC')]$ contains a spanning forest $F'$ where every tree in $F'$ intersects the set $R$.
Furthermore, by starting with $F'$ and recursively removing one edge on a path connecting two vertices in $R$ in the current forest until this is no longer possible, we obtain a forest $F$ with $|R|$ trees where each tree contains one vertex in $R$ (called its \emph{root}) and $V(\cC')=V(F)$. In particular, we have that $v(\cC')=d'm+|E(F)|$.
To complete the proof, we bound the probability that such a forest $F$ exists in $G_1$. We note that the argument from the proof of \Cref{l: odd size} would be suboptimal here, and we utilise \Cref{cor:forests} instead. 

First, there are $\binom{n}{m}$ ways to choose the set $U'$. 
Then, for every $u\in U'$, we choose the sets of roots in the $Q^d$-neighbourhood of $u$ in at most $\binom{d}{d'}$ ways. Letting $k$ be the number of edges in the forest $F$, by \Cref{cor:forests}, there are at most $(d'm+k)^kd^k/k!$ forests on $k$ edges in $Q^d$ whose roots lie in $R$. 
Due to the bounds~\eqref{eq:C'} on $v(\cC')=d'm+k$ and on the number of edges in $Q^d$ with exactly one endpoint in $V(\mathcal{C}')$, we conclude that the probability that a forest $F$ with the required properties exists is at most
\begin{equation}\label{eq:main-5.5}
\sum_{k=\max\{0,\eta \eps t/2-d'm\}}^{n-1} \binom{n}{m}\binom{d}{d'}^m\cdot \frac{(d'm+k)^kd^k}{k!}\cdot p^k (1-p_1)^{(1-2\eta)d(d'm+k)}.
\end{equation}
Recall that each component in $\cC$ has order at most $d^2$. Thus, thanks to~\eqref{eq:C'}, $d^2d'm\geq \eta\varepsilon t/2$, implying that $m\geq \eta \eps t/(2d^3)\ge \e n^{1-2\gamma}$.
Then, the product of the first two terms in~\eqref{eq:main-5.5} is bounded from above by 
\begin{align}
\bigg(\frac{\e n}{m}\bigg)^m \bigg(\frac{1}{(2\eps)^{2\eps}(1-2\eps)^{1-2\eps}}\bigg)^{dm}&=
\exp\left(2\gamma dm+dm\log\bigg(\frac{1}{(2\eps)^{2\eps}(1-2\eps)^{1-2\eps}}\bigg)\right)\notag \nonumber\\
&\leq 
\exp\left(2d'm\log\bigg(\frac{1}{(2\eps)^{2\eps}(1-2\eps)^{1-2\eps}}\bigg)\right)
\label{eq:1-2 terms}
\end{align}
for $\eps$ and $\gamma=\gamma(\eps)$ small enough. In turn, the product of the remaining terms in~\eqref{eq:main-5.5} is bounded from above by
\begin{align}
\label{eq:L5.4-auxiliary-bound}
\bigg(1+\frac{d'm}{k}\bigg)^k (\e dp)^k \e^{-(1-2\eta)(c-\delta)(d'm+k)}&\le \e^{d'm} (\e c)^k \e^{-(1-2\eta)(c-\delta)(d'm+k)}\notag\\&=(\e^{1-(1-2\eta)(c-\delta)})^{d'm}(c \e^{1-(1-2\eta)(c-\delta)})^k.
\end{align}

At this point, we choose $\eta=\eta(c,\delta)$ suitably small so that 
\begin{equation}
(1-2\eta)(c-\delta)\geq 1+2\eta \qquad \text{and}\qquad  c\cdot \e^{1-(1-2\eta)(c-\delta)}\le \e^{-\eta},
\label{eq:choice-eta}
\end{equation}
where we recall that $\delta=\delta(c)$ is sufficiently small.
In particular, \eqref{eq:choice-eta} implies that the expression in~\eqref{eq:L5.4-auxiliary-bound} is at most $\e^{-2\eta d'm-\eta k}$. Finally, we choose $\eps=\eps(\eta,c)$ suitably small ensuring that~\eqref{eq:1-2 terms} is bounded from above by $\e^{\eta d'm}$. 
Recalling that $d'm+k=v(\mathcal{C}')\geq\eta\varepsilon t/2$ due to~\eqref{eq:C'}, by combining the latter estimates with~\eqref{eq:main-5.5} and the union bound over all $m\in [\eta \eps t/2d^3,n]$, the probability that the desired forest $F$ exists is at most
$$n\max_{m\in [\eta \eps t/2d^3,n]}\bigg\{\sum_{k=\max\{0,\eta \eps t/2-d'm\}}^{n-1}\e^{-\eta d'm-\eta k}\bigg\}\le n^2\cdot\e^{-\eta^2 \eps t/2}\le\e^{-\nu t}$$
where $\nu = \eta^2 \eps/3$, completing the proof.
\end{proof}

Our next lemma shows that it is unlikely that many vertices in `small' components in $G_1$ merge into a `large' component in $G_2$. 
It is formulated in wider generality than needed here, as we apply the lemma slightly differently in \Cref{s: lower tail} (where \Cref{l: no surprises from outside} is mostly relevant) and in \Cref{s: mixing time}.
Recall $\delta, p_1, p_2, G_1, G_2$ from the beginning of \Cref{s: lower tail}.

\begin{lemma}\label{l: no surprises from outside_gen}
For every suitably small $\eta = \eta(c) > 0$, there is $\eps=\eps(c,\delta,\eta)>0$ with the following property: 
the probability that there is a vertex set $S$ simultaneously satisfying 
\begin{enumerate}[label=\emph{(\alph*)}]
    \item $|S| = k\in [d^{51},n]$,
    \item each connected component in $G_2[S]$ has order at least $d^{10}$,
    \item $G_1$ contains no edge between $S$ and $V(Q^d)\setminus S$,
    \item $|V(\cM_{[0,2)}^-)\cap S|\ge (1-\eta)k$
\end{enumerate}
is at most $\exp(-\eps k)$.
\end{lemma}
\begin{proof}
Fix $k\in [d^{51},n]$ and suppose that a set $S$ satisfies each of the properties listed above. 
Set $S'=V(\cM_{[0,2)}^-)\cap S$.
The following claim is a key part of the proof.

\begin{claim}\label{claim: no sudden collapse help}
There are $m\ge 1$ and vertex-disjoint trees $T_1,\ldots, T_m\subseteq G_2[S]$ such that, for every $i\in [m]$, $T_i$ satisfies each of the following properties:
\begin{enumerate}[label=\emph{(\arabic*)}]
    \item $v(T_i)\in [d^2,2d^5]$,
    \item $|S'\cap V(T_i)|\ge (1-2\eta)v(T_i)$,\label{i: intersection in collapse}
    \item $|\bigcup_{i\in [m]} V(T_i)|\ge k/2$,
    \item for every component $C$ in $G_1[S']$ and every $i\in [m]$, either $V(C)\subseteq V(T_i)$ or $V(C)\cap V(T_i)=\varnothing$. \label{i: components in collapse}
\end{enumerate}
\end{claim}
\begin{proof}
By assumption (c), the vertex set of every connected component in $G_1$ is either contained in $S$ or disjoint from $S$.
Based on this observation, we construct an auxiliary graph $G_{\mathrm{aux}}$ as follows: starting from $G_2[S]$, contract every component in $\cM_{[0,2)}^-$ contained in $S$ to a single vertex, delete loops, and identify multiple edges.
To every vertex in $G_{\mathrm{aux}}$, associate a weight equal to the number of vertices it corresponds to before the contraction of $G_2[S]$. 
Hence, $G_{\mathrm{aux}}$ has vertex-weights in the interval $[1,d^2]$, maximal degree at most $d^2\cdot d = d^3$ and, by (b), every connected component has total weight at least $d^{10}$.

Fix a spanning forest $F_{\mathrm{aux}}$ where every tree spans a different component of $G_{\mathrm{aux}}$.
Then, by applying \Cref{l: decomps} for each tree in $F_{\mathrm{aux}}$, we can find a spanning subforest $F_{\mathrm{aux}}'\subseteq F_{\mathrm{aux}}$ where every tree has weight in the interval $[d^2,(d^3+1)d^2]\subseteq [d^2,2d^5]$. 
Note that the forest $F_{\mathrm{aux}}'$ in $G_{\mathrm{aux}}$ corresponds to a spanning forest $F$ in $G_2[S]$ where every tree contains between $d^2$ and $2d^5$ vertices, and, by construction, for every component $C$ in $G_1[S']$, every tree in $F$ either contains $C$ or does not intersect with $C$.

Finally, suppose towards contradiction that more than $k/2$ vertices belong to trees $T\subseteq F$ with the property $|S'\cap V(T)| < (1-2\eta)v(T)$.
Then, we must have $|S'| < (1-2\eta)(k/2) + (k/2) = (1-\eta)k$, contradicting (d).
Hence, denoting by $T_1,\ldots,T_m$ the trees $T\subseteq F$ with $|S'\cap V(T)| \ge (1-2\eta)v(T)$ provides a collection satisfying each of the properties (1)--(4).
\end{proof}

We now consider trees $T_1, \ldots, T_m$ as given by \Cref{claim: no sudden collapse help}. 
Since those trees are vertex-disjoint, each of them has order at most $2d^5$ and $|\bigcup_{i\in [m]} V(T_i)|\ge k/2>d^{50}$, \Cref{l: sparsification} applies. 
Thus, for any sufficiently small constant $\eps'>0$, there is a family $\cT\coloneqq \{T_1',\ldots, T_{\ell}'\}\subseteq \{T_1,\ldots, T_m\}$ with $v(\cT)\ge \eps' k/2$ and $e_{\mathrm{out}}(\cT)\ge (1-2\eps')dv(\cT)$. 
We denote this event by $\cE = \cE(\eps')$ and estimate its probability in the remainder of the proof.

Fix integers $x_1,\ldots, x_{\ell}\in [d^2,2d^5]$ with $x\coloneqq \sum_{i=1}^{\ell} x_i\in [\eps' k/2, n]$. By Lemma~\ref{l: trees}, the number of forests in $Q^d$ with $\ell$ trees of orders $x_1,\ldots, x_{\ell}$ is bounded from above by $\binom{n}{\ell}(\e d)^{x-\ell}$.
Assume the trees $T_1',\ldots, T_{\ell}'$ have been chosen.
By Claim \ref{claim: no sudden collapse help}\ref{i: intersection in collapse}, for the tree $T_i'$ of order $x_i$, the number of ways to specify a subset of $(1-2\eta x_i)$ vertices in $V(T_i')$ that belong to $V(\cM_2^-)$ is bounded from above by $\tbinom{x_i}{2\eta x_i}\le (\e/2\eta)^{2\eta x_i}$.
Note that by \Cref{claim: no sudden collapse help}\ref{i: components in collapse}, every edge between $S'\cap V(T_i')$ and $V(Q^d)\setminus (S'\cap V(T_i'))$ is not present in $G_1$ for all distinct $i,j\in [\ell]$. Thus, by \Cref{claim: no sudden collapse help}\ref{i: intersection in collapse}, there are at least $e_{\mathrm{out}}(\cT)-2\eta v(\cT)d=e_{\mathrm{out}}(\cT)-2\eta xd$ closed edges in $E(V(\cT),V(Q^d)\setminus V(\cT))$.
Hence, the probability that each edge in such a forest is in $G_2$ but no edge adjacent to $S'$ on its boundary is in $G_1$ is at most $p^{x-\ell}(1-p_1)^{(1-2\eps')xd-2\eta xd}$. 
Thus, by the union bound,
\begin{align}
    \mathbb{P}\left(\cE\right)&\le \sum_{x=\eps' k/2}^{n} \sum_{\substack{x_1+\cdots+x_{\ell}=x\\x_1,\ldots,x_{\ell}\in[d^{2},2d^{5}]}}\binom{n}{\ell} \bigg(\prod_{i=1}^{\ell} \bigg(\frac{\e}{2\eta}\bigg)^{2\eta x_i}\bigg) (\e dp)^{x-\ell}(1-p_1)^{(1-2\eps'-2\eta)xd}\nonumber\\
    &\le \sum_{x=\eps' k/2}^{n}\sum_{\substack{x_1+\cdots+x_{\ell}=x\\x_1,\ldots,x_{\ell}\in[d^{2},2d^{5}]}}\left(\frac{\e n}{x/d^2}\right)^{x/d^2}\exp(x\left(\log(\e c)+2\eta\log(\e/2\eta)-(1-2\eps'-2\eta)(c-\delta)\right)),\label{eq:largeUB2}
\end{align}
where we used that $\ell\le x/d^2$. 
To simplify~\eqref{eq:largeUB2}, observe that there exist at most $(2d^5)^{x/d^2}$ ways to choose $x_1,\ldots, x_{\ell}\in[d^2,2d^5]$ with $\sum_{i=1}^{\ell}x_i=x$.
Furthermore, by choosing $\delta,\eps',\eta$ suitably small with respect to $c$ and using that $c>1$, we have that $c_0 := (1-2\eps'-2\eta)(c-\delta) - 2\eta\log(\e/2\eta) - \log(\e c) > 0$.
Thus,~\eqref{eq:largeUB2} is at most
\begin{align*}
\sum_{x=\eps' k/2}^{n}\left(\frac{2\e nd^7}{x}\right)^{x/d^2}\e^{-c_0x}\le \sum_{x=\eps' k/2}^{n}\exp\bigg(x\left(\frac{1}{d^2}\log\left(\frac{2\e nd^7}{x}\right)-c_0\right)\bigg).
\end{align*}
Using that $\log(2\e nd^7/x)=o(d^2)$, we obtain
\begin{align*}
\mathbb P(\cE)\le \sum_{x=\eps' k/2}^n\exp(-xc_0/2)\le n \exp(-\eps'c_0k/4).
\end{align*}
Thus, setting $\eps = \eps' c_0/5$ finishes the proof.
\end{proof}

Next, we state a useful (and immediate) corollary of \Cref{l: no surprises from outside_gen}.
Define $\cS$ to be the subset of $\cM_{10}$ which does not intersect with any component in $\cM_2^-$.

\begin{corollary}\label{l: no surprises from outside}
Given $\eps = \eps(c,\delta)$ as in \Cref{l: no surprises from outside_gen}, we have $\mathbb{P}(v(\cS)\ge d^{-1/2} n)\le \exp(-\eps d^{-1/2} n)$.
\end{corollary}
\begin{proof}
By definition, we have that each component in $G_2[V(\cS)]$ has order at least $d^{10}$ and $G_1$ has no edge between $V(\cS)$ and $V(Q^d)\setminus V(\cS)$. Further, $|V(\cM_{[0,2)})\cap V(\cS)|=v(\cS)$. Therefore, we may apply Lemma \ref{l: no surprises from outside_gen} to the set $\cS$ and obtain the required bound.
\end{proof}

Before moving further, define $G_1^+\coloneqq G_1\cup G_2[V(\cM_{[0,2)}^-)]\cup G_2[V(\cM_{[0,2)}^-),V(\cM_2^-)]$, that is, the graph $G_1^+$ contains all edges of $G_1$ and all edges of $G_2$ incident to components of order less than $d^2$ in $G_1$. 
For a vertex $v\in V(\cM_2^-)$, we denote by $S_v\subseteq V(\cM_{[0,2)}^-)$ the set of vertices in components $C$ in $G_2[V(\cM_{[0,2)}^-)]$ where $C$ is adjacent to $v$ in $G_1^+$. 
Given a set $\cC\subseteq \cM_2^-$, we define $S_{\cC} = S_{V(\cC)} \coloneqq \bigcup_{v\in V(\cC)} S_v$.

Note that, when moving from $G_1$ to $G_2$, the expected order of the giant component increases (in fact, by $\Theta_c(\delta n)$). 
The next lemma says no family of components $\cC\subseteq \cM_2^-$ connects via $G_2$-paths to too many vertices in $G_2\setminus V(\cM_2^-\setminus \cC)$ which themselves may connect to $\cM_2^-\setminus \cC$ only through $\cC$. 
In practice, this means that the giant component in $G_1$ is incremented by well-spread connected pieces to form the giant component of $G_2$.
The following (rather crude) quantification of this statement will suffice for us. 
\begin{lemma}\label{l: no explosions}
There are $K=K(c,\delta), \eps = \eps(c,\delta)>0$ such that the probability that there is a family $\cC\subseteq \cM_2^-$ with $|S_{\cC}|\ge K\max\{v(\cC),d^{-1/2}n\}$ and $E(Q^d_{p_2}[S_{\cC},V(\cM_2^-\setminus \cC)])=\varnothing$ is at most $\exp(-\eps d^{-1/2} n)$.
\end{lemma}
\begin{proof}
By Lemma \ref{l: large components are well spread}, there are constants $\eps_1 = \eps_1(c-\delta,\delta)>0$ and $\nu=\nu(c-\delta,\delta)\in (0,1)$ such that, with probability at least $1-\e^{-\nu d^{-1/2} n}$, there are at most $d^{-1/2}n$ vertices in $Q^d$ which have less than $\eps_1 d$ neighbours in $V(\cM_2^-)$. We expose the graph $G_1$ and assume that the said event holds. 

Now, fix a family $\cC\subseteq \cM_2^-$ and denote $s:=K\max\{v(\cC),d^{-1/2}n\}$ with $K$ suitably large. Then, we reveal $Q^d_{p_2}[V(\cM_{[0,2)}^-)]$ and $Q^d_{p_2}[V(\cC),V(\cM_{[0,2)}^-)]$, thus discovering $S_{\cC}$. 
Suppose that $|S_{\cC}|\ge s$. By assumption, there are at least $\eps_1 d(s-d^{-1/2}n)$ edges of $Q^d$ between $S_{\cC}$ and $V(\cM_2^-)$, and therefore
\begin{align*}
    e_{Q^d}(S_{\cC}, V(\cM_2^-\setminus\cC)) = e_{Q^d}(S_{\cC}, V(\cM_2^-))-e_{Q^d}(S_{\cC}, V(\cC))\ge \eps_1 d(s-d^{-1/2}n)-dv(\cC)>\eps_1 ds/2.
\end{align*}
Using that $dp_2\ge \delta$, the probability that none of the edges with one endpoint in $S_{\cC}$ and the other endpoint in $V(\cM_2^-\setminus\cC)$ appears in $Q^d_{p_2}$ is at most $(1-p_2)^{\eps_1 ds/2}\le \exp(-\eps_1 \delta s/2)$. 
However, the number of ways to choose a family~$\cC\subseteq \cM_2^-$ with $k=v(\cC)$ is at most
\[\sum_{i=0}^{k/d^2} \binom{n/d^2}{i}\le n \bigg(\frac{n}{d^2}\bigg)^{k/d^2}\le n \e^{k/d},\]
where we used that there are at most $n/d^2$ components in $\cM_2^-$ and there are at most $k/d^2$ components in $\mathcal{C}$.
Thus, the probability that a family $\cC$ violating the statement of the lemma exists is at most 
\begin{align*}
    \e^{-\nu d^{-1/2}n}+\sum_{k=1}^{n}n\exp(k/d-K \eps_1 \delta \max\{k,d^{-1/2}n\}/2).
\end{align*}
Choosing $K \ge 5/(\eps_1 \delta)$ and setting $\eps=\min\{\nu,K\eps_1\delta/2\}/2$ completes the proof. 
\end{proof}

\subsection{Proof of the lower tail bound of Theorem \ref{th: large tail deviation}}\label{subsec: lower bound proof}
Fix $K$ as in the statement of Lemma~\ref{l: no explosions} and $\eps_0 := (10K)^{-1}$.
Furthermore, fix $\eps_1>0$, $\gamma=\gamma(\eps_1)\in (0,1]$, and $t=t(d)\in [n^{1-\gamma}/\eps_0,n]$ satisfying the assumptions of \Cref{l: large components are well spread}. We note that we begin with considering such values of $t$ as it will be useful for us in Section \ref{s: mixing time}, however, later in the proof we will restrict our attention to the interval $[n/d^{0.1},n]$.
Given $t_1\in [\varepsilon_0 t,n/2]$, denote by $\cE_1(t_1)$ the event that there are more than $r=r(t_1):=t_1 \log_2(n/t_1)/(2d)$ $\eps_1$-bad vertices in $Q^d$, that is, vertices with less than $\eps_1 d$ neighbours in $V(\cM_2^-)$.
By Lemma~\ref{l: large components are well spread}, for $\eps_1$ suitably small, 
\begin{equation}
\label{eq:E_1_t_1}
\mathbb P(\cE_1(t_1))\le \e^{-\nu r}
\end{equation}
for some $\nu=\nu(c-\delta)>0$.

Given $t_1\in [\varepsilon_0 t,n/2]$, denote by $\cE_2(t_1)$ the event that 
there exists a component-respecting partition of $V(\cM_2^-)$ into two parts $A,B$ with $t_1=|A|\le |B|$ and such that, 
after sprinkling with probability $p_2$, there are no paths between $A$ and $B$ in $G_2$. Define $\cE_2:=\bigcup_{t_1\in [\eps_0 t,n/2]} \cE_2(t_1)$.

An important part of the proof consists in estimating the probability of the event $\cE_2$.
To this end, we fix a graph $H$ satisfying the event $\mathcal{E}_1(t_1)^c$ for some $t_1\in [\eps_0t,n/2]$ and assume $G_1=H$. 
Then, we consider a partition $A,B$ as above with $t_1=|A|\leq |B|$. We bound from above the probability of the event $\mathcal{E}_2(A,B)$ that there are no paths between $A$ and $B$ in $G_2$. 
We then use the latter probability to bound $\mathbb{P}(\mathcal{E}_2)$ via the union bound:
\begin{align}
\label{eq:E_2-union_bound}
 \mathbb{P}(\mathcal{E}_2)&\leq\sum_{t_1\in [\eps_0t,n/2]}\mathbb{P}(\mathcal{E}_2(t_1))\leq
 \sum_{t_1\in [\eps_0t,n/2]}\biggl(\mathbb{P}(\mathcal{E}_2(t_1)\mid\cE_1(t_1)^c)+\mathbb{P}(\cE_1(t_1))\biggr)\notag\\ 
 &\leq\sum_{t_1\in [\eps_0t,n/2]}\left(\mathbb{P}(\cE_1(t_1))+\max_{H\in\cE_1(t_1)^c}\sum_{A\sqcup B=\cM_2^-}\mathbb{P}(\mathcal{E}_2(A,B)\mid G_1=H)\right).
\end{align}

\begin{lemma}\label{lem:estimate_E2}
There is a constant $\eps_2 = \eps_2(c,\delta)>0$ such that $\mathbb P(\cE_2)\le \exp(-\eps_2 t\log(n/t)/d)$.
\end{lemma}
\begin{proof}
Fix $A,B$ and $|A|=t_1$ as above. Denote by $A_1$ the set of vertices in $Q^d$ outside $V(\cM_2^-)$ with at least $\eps_1d/2$ neighbours in $A$, and denote by $B_1$ the set of vertices in $Q^d$ outside $V(\cM_2^-)\cup A_1$ with at least $\eps_1d/2$ neighbours in $B$. 
We further set $A'=A\cup A_1$ and $B'=B\cup B_1$. Now, by Lemma \ref{l: Harper},
$$
e_{Q^d}(A',V(Q^d)\setminus A')\ge|A'|(d-\log_2|A'|)\ge t_1\log_2(n/t_1),
$$
where we note that $|A'|\ge t_1$ and $|V(Q^d)\setminus A'|\ge |B'|\ge |A|=t_1$.

Note that the second endpoint of every edge from $A'$ to $V(Q^d\setminus A')$ is either $\eps_1$-bad or in $B'$.
Therefore,
\begin{align}\label{eq:bipartite}
e_{Q^d}(A',B')\ge t_1\log_2(n/t_1)-d\cdot r \ge t_1\log_2(n/t_1)/2 \ge \varepsilon_0 t\log_2(n/(\varepsilon_0 t))/2,
\end{align}
where the last inequality follows by simple analysis of the function $x\mapsto x\log_2(n/x)$ and the fact that $\eps_0\le 0.1$.

Let $F_0\coloneqq E_{Q^d}(A',B')$. Then, by Lemma~\ref{l: matchings}, there is a constant $\alpha\coloneqq \alpha(\delta,\eps_0)>0$ such that the probability that the size of a maximum matching in $(F_0)_{p_2}$ is less than $\alpha t_1\log_2(n/t_1)/2d$ is at most $\exp(-\alpha t_1\log_2(n/t_1)/2d)$.
Now, set $m=m(t_1):=\alpha t_1\log_2(n/t_1)/8d$ and define the events 
\begin{align*}
\cE_3 &:= \text{$Q^d_{p_2}$ contains a matching of size $m$ between $A_1$ and $B_1$,}\\
\cE_4 &:= \text{$Q^d_{p_2}$ contains a matching of size $m$ between $A_1$ and $B$,}\\
\cE_5 &:= \text{$Q^d_{p_2}$ contains a matching of size $m$ between $A$ and $B_1$,}\\
\cE_6 &:= \text{$Q^d_{p_2}$ contains a matching of size $m$ between $A$ and $B$.}
\end{align*}
In particular, $\mathbb P(\cE_3\cup\cE_4\cup\cE_5\cup\cE_6\mid G_1=H)\ge 1-\e^{-4m}$.
We separately estimate the probability of $\cE_2(A,B)$ conditionally on $\{G_1=H\}\cap \cE_i$ for each $i\in [3,6]$.

\begin{claim}\label{cl:E3-E6}
For every $i\in [3,6]$, $\mathbb P(\cE_2(A,B)\mid \{G_1=H\}\cap \cE_i)\le (1-\delta^2\eps^2_1/5)^m$.
\end{claim}
\begin{proof}[Proof of \Cref{cl:E3-E6}]
We start with the case $i=3$. Expose $Q^d_{p_2}[A_1,B_1]$ and denote by $M$ a matching of size $m$ therein.
Then, by definition of $A_1,B_1$ and using that the edges between $A$ and $A_1$, and between $B$ and $B_1$ have not yet been revealed in $Q^d_{p_2}$, for every edge $uv\in M$, the probability that $u$ is adjacent to $A$ in $Q^d_{p_2}$ and $v$ is adjacent to $B$ in $Q^d_{p_2}$ is bounded from below by 
$$
(1-(1-p_2)^{\eps_1 d/2})^2 \ge (1-\e^{-\eps_1 p_2d/2})^2 \ge (1-\e^{-\delta \eps_1/2})^2\ge \delta^2 \eps^2_1/5,
$$
where we used that $p_2d\ge \delta$ and $\delta = \delta(c)$ is suitably small.
The independence of the latter events shows the bound for $i=3$.

We turn to the case $i=4$. Expose $Q^d_{p_2}[A_1,B]$ and denote by $M$ a matching of size $m$ therein.
Then, by definition of $A_1$ and using that the edges between $A$ and $A_1$ in $Q^d_{p_2}$ have not yet been revealed, for every edge $uv\in M$ with $v\in A_1$, the probability $v$ is adjacent to $A$ in $Q^d_{p_2}$ is bounded from below by 
$$
1-(1-p_2)^{\eps_1 d/2}\ge 1-\e^{-\delta \eps_1/2}\ge \delta \eps_1/3 \ge \delta^2 \eps^2_1/5.
$$
Again, the independence of the latter events shows the bound for $i=4$.

The inequality $\mathbb P(\cE_2(A,B)\mid \{G_1=H\}\cap \cE_5)\le (1-\delta^2 \eps^2_1/5)^m$ follows as in the case $i=4$. 
Finally, we trivially have $\mathbb P(\cE_2(A,B)\mid \{G_1=H\}\cap \cE_6)=0$ from the definitions of $\cE_2(A,B)$ and $\cE_6$.
\end{proof}

Note that
\begin{align}\label{eq: helping}
\mathbb P(\cE_2(A,B)\mid G_1=H) & \le \mathbb P((\cE_3\cup \cE_4\cup \cE_5\cup \cE_6)^c\mid G_1=H) + \sum_{i=3}^6 \mathbb P(\cE_2(A,B)\mid \cE_i\cap  \{G_1=H\})\notag\\
&\le \e^{-4m(t_1)}+4(1-\delta^2\eps^2_1/5)^{m(t_1)}\le \exp\left(-\delta^2\eps^2_1m(t_1)/6\right),
\end{align}
where the penultimate inequality follows from Claim \ref{cl:E3-E6}. 

We now estimate the maximum in \eqref{eq:E_2-union_bound}. The number of choices of $A$ (and $B$) given $t_1$ is at most 
\begin{align*}
    \sum_{j=0}^{t_1/d^2}\binom{n/d^2}{j}\le n\binom{n/d^2}{t_1/d^2}\le n\left(\frac{\e n}{t_1}\right)^{t_1/d^2}\le \exp\bigg(\frac{2t_1}{d^2}\log\bigg(\frac{\e n}{t_1}\bigg)\bigg),
\end{align*}
where the first inequality uses that $t_1=|A|\le |B|\le n/2$. Thus, by the above together with \eqref{eq:E_1_t_1}, \eqref{eq:E_2-union_bound} and \eqref{eq: helping}, we obtain that
\begin{align}\label{eq: bounding ce2}
    \mathbb{P}\left(\cE_2\right)&\le  \sum_{t_1\in [\eps_0t,n/2]}\left(\mathbb P(\cE_1(t_1))+\exp\left(\frac{2t_1}{d^2}\log\bigg(\frac{\e n}{t_1}\bigg)-\frac{\delta^2\eps^2_1m(t_1)}{6}\right)\right)\notag\\
    &\le n\exp\biggl(-\frac{\nu\varepsilon_0 t \log_2(n/(\varepsilon_0 t))}{2d}\biggr)+\exp\biggl(-\frac{1}{2}\cdot \frac{\delta^2\eps^2_1}{6}\cdot \frac{\alpha \eps_0t\log_2(n/(\eps_0t))}{8d}\biggr)\nonumber\\
    &\le \exp\biggl(-\frac{\eps_2t\log(n/t)}{d}\biggr),
\end{align}
where we made use of the relation $m(t_1)=\omega(t_1\log(\e n/t_1)/d^2)$ for every $t_1\in [\eps_0t,n/2]$, and assumed that $\eps_2\coloneqq \eps_2(\eps_0,\eps_1,\delta,\nu,\alpha)>0$ is a sufficiently small constant.
\end{proof}

We now restrict ourselves to $t\in [n/d^{0.1},n]$. Observe that the event $\cE_2^c$ implies that there is a component $L_1'$ in $G_2$ whose order is at least $v(\cM_2^-)-\eps_0t$. 
Denote the family of components in $\cM_2^-$ which are not in $L_1'$ by $\cC$. 
Moreover, recall the set $S_{\cC}$ from the notation paragraph before \Cref{l: no explosions} and the subset $\cS$ of $\cM_{10}$ containing the components disjoint from $V(\cM_2^-)$, see Figure \ref{f: big proof}. Further, note that the inequalities
\begin{align}
    \mathbb{P}(v(\cM_{10})\le y(c)n-2n/d^{1/9})\le \exp(-n/d^{7/9})\qquad \text{and}\qquad \mathbb{P}(v(\cS)\ge n/d^{1/2})\le \exp(-n/d^{2/3})\label{eq: W(p)}
\end{align}
follow from \Cref{l: W(p)} and \Cref{l: no surprises from outside}, respectively. 
Note that $L_1'$ is the unique component in $\cM_{10}\setminus \cS$ intersecting $\cM_2^-\setminus \cC$ (see Figure \ref{f: big proof}), implying that
\begin{align*}
|V(L_1')|\ge v(\cM_{10}\setminus\cS)-v(\cC)-|S_{\cC}| = v(\cM_{10})-v(\cS)-v(\cC)-|S_{\cC}|.
\end{align*}
Thus, $\mathbb{P}\left(|V(L_1)|\le y(c)-t\right)$ is bounded from above by
\begin{equation}\label{eq:lower_tail_final_expansion}
\begin{split}
\mathbb{P}\left(|V(L_1')|\le y(c)-t\right)
\le \mathbb{P}(\cE_2)
&+ \mathbb{P}(v(\cM_{10})\le y(c)-2n/d^{1/9})
+\mathbb{P}(v(\cS)\ge n/d^{1/2})\\
&+\mathbb{P}(\{v(\cC)+|S_{\cC}|\ge t/2\}\cap \{v(\cC)\le \eps_0t\}),
\end{split}
\end{equation}
where the last term uses that $v(\cC)\le \eps_0 t$ on the event $\cE_2^c$.
Thus, by combining~\eqref{eq:E_1_t_1}, \eqref{eq: bounding ce2}, and \eqref{eq: W(p)}, the sum of the first three terms on the right hand side of~\eqref{eq:lower_tail_final_expansion} is at most 
\[
\e^{-\eps_2t\log(n/t)/d}+\e^{-n/d^{7/9}}+\e^{-n/d^{2/3}}\le 2\e^{-\eps_2t\log(n/t)/d}.
\]
For the last term, the two events therein jointly yield 
$$
|S_{\cC}|\ge (1/2-\eps_0)t\ge K\cdot (\eps_0t)\ge K\max\{d^{-1/2}n,v(\cC)\}.
$$
Note that $S_{\cC}$ is not adjacent to the components $\cM_2^-\setminus \cC$ in $G_2$: indeed, otherwise, if $v\in \cC$ is adjacent to a component $C\subseteq S_v$ which is, in turn, adjacent to a component in $\cM_2^-\setminus \cC$, then $v$ and $C$ belong to $L_1'$, contradicting the definition of $\cC$. Therefore, the assumptions of \Cref{l: no explosions} are satisfied, implying 
\[\mathbb{P}(\{v(\cC)+|S_{\cC}|\ge t/2\}\cap \{v(\cC)\le \eps_0t\})= o(\e^{-\eps_2t\log(n/t)/d}).\]
Choosing $\eps = \eps_2/2$ completes the proof of the lower tail estimate in \Cref{th: large tail deviation}.
\qed{}
\begin{figure}[H]
\centering
\includegraphics[width=0.8\textwidth]{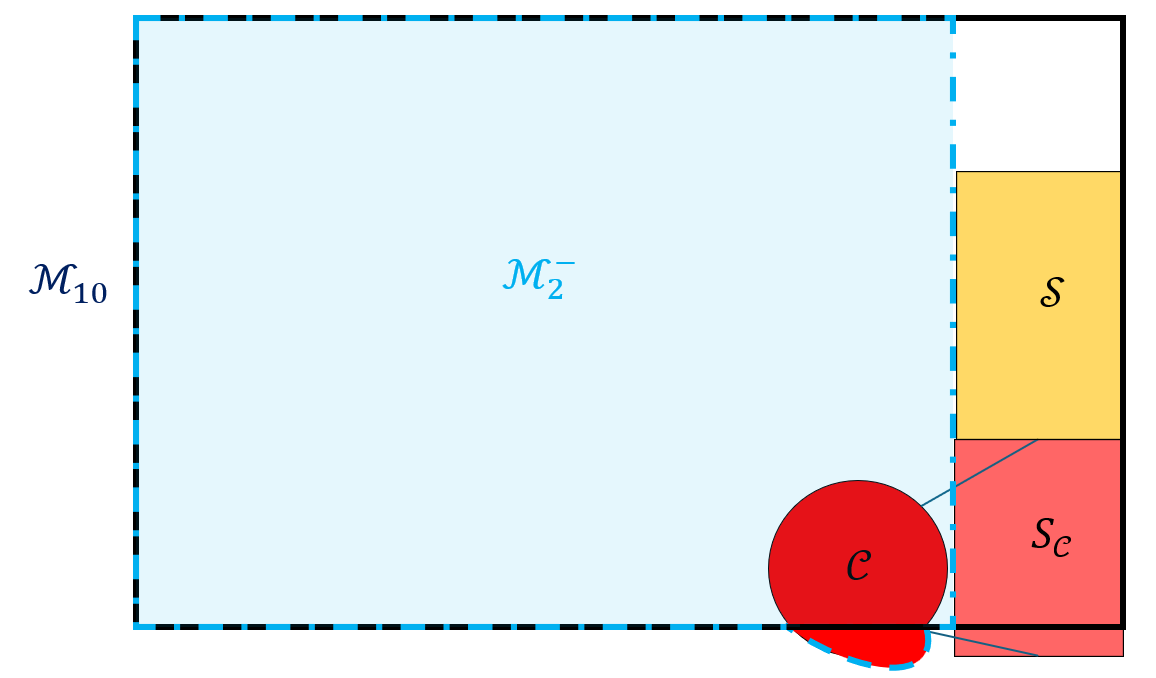}
\caption{Illustration of some of the sets from the last proof. The family $\cM_{10}$ is represented by a rectangle with a thick black border, and the family $\cM_2^-$ is represented by a light-blue shape with a dashed border. 
The family $\cC$ of components in $\cM_2^-$ which are not part of $L_1'$ and the set of vertices in $S_{\cC}$ attached to it through edges of $Q^d_{p_2}$ are both represented in red. 
The family $\cS$ of components which are in $\cM_{10}$ but do not intersect $\cM_2^-$ is represented in yellow. $L_1'$ consists of the light-blue and the white region.}
\label{f: big proof}
\end{figure}

\section{Mixing time}\label{s: mixing time}
In this section, we fix a constant $c>1$, set $p=p(d)=c/d$ and $n=2^d$. Recall $p_1,p_2,G_1$ and $G_2$ from Section \ref{s: lower tail} and that $L_1$ stands for the giant component in $G_2\cong Q^d_p$. 
The section is structured as follows: in \Cref{subsec: expansion}, we utilise results in \Cref{s: lower tail} to deduce some expansion properties of $L_1$, and use them to derive \Cref{th: main}\ref{th: mixing time} in \Cref{subsec: mixing}.

\subsection{Expansion properties of the giant}\label{subsec: expansion}
First, we show that connected sets in the giant $L_1$ expand sufficiently well.

\begin{proposition}\label{prop: expansion}
For every sufficiently small $\delta = \delta(c) > 0$, there are $\gamma = \gamma(c,\delta),\eps = \eps(c,\delta)>0$ such that \textbf{whp}, for every set $S\subseteq V(L_1)$ with size $|S|\eqqcolon  k\in [n^{1-\gamma},y(c-2\delta)n/4]$ connected in $L_1$, we have that $e_{G_2}(S,V(L_1)\setminus S)\ge \eps k\log_2(n/k)/d$.
\end{proposition}
\begin{proof}
Fix sufficiently small constants $\delta = \delta(c) > 0$ and $\gamma = \gamma(c,\delta) > 0$ (in particular, satisfying $c-2\delta>1$) and let $k = k(d)\in [n^{1-\gamma},y(c-2\delta)n/4]$. 
We define $p_0\coloneqq(c-2\delta)/d$, $G_0\coloneqq Q^d_{p_0}$ (equipped with the natural coupling ensuring $G_0\subseteq G_1\subseteq G_2$), $\rho = p_1/p=1-\delta/c$ and $\rho_0=p_0/p_1 = 1-\delta/(c-\delta)$. 
Define the events
\begin{align*}
    \cA_k(\eps)&\coloneqq\{\exists S\subseteq V(L_1), |S|=k, L_1[S] \text{ is connected}: e_{G_2}(S,V(L_1)\setminus S)<\eps k\log_2(n/k)/d\},\\
    \cB_k &:= \{\exists S\subseteq V(L_1), |S|=k, L_1[S] \text{ is connected}: e_{G_1}(S,V(L_1)\setminus S)=0\}.
\end{align*}
Note that
\begin{align*}
    \mathbb{P}(\cB_k\mid \cA_k(\eps))\ge (1-\rho)^{\eps k\log_2(n/k)/d}=\exp\left(-\log(c/\delta)\eps k\log_2(n/k)/d\right).
\end{align*}
Combining the latter with the inequality $\mathbb{P}(\cA_k(\eps)) \mathbb{P}(\cB_k\mid \cA_k(\eps))\le \mathbb{P}(\cB_k)$ then yields
\begin{align}
\mathbb{P}(\cA_k(\eps)) \le \mathbb{P}(\cB_k)\cdot \exp\left(\log(c/\delta)\eps k\log_2(n/k)/d\right).\label{eq: estimate on Akeps}
\end{align}

Next, we estimate $\mathbb{P}(\cB_k)$ from above. Denote by $\cM_2^{(0)}$ (resp.\ $\cM_{[0,2)}^{(0)}$) the family of components of order at least (resp.\ less than) $d^2$ in $G_0$.
Fix a small constant $\eta = \eta(c)>0$ and denote by $\cD_k(\eta)$ the event that some set $S\subseteq V(L_1)$ of size $k$ connected in $G_2$ satisfies the assumptions of \Cref{l: no surprises from outside_gen} for $G_0$ instead of $G_1$: 
namely $e_{G_0}(S,V(L_1)\setminus S)=0$ and the number of vertices in $V(\cM_{[0,2)}^{(0)})\cap S$ is at least $(1-\eta)k$.
By \Cref{l: no surprises from outside_gen}, there exists a constant $\eps_1 = \eps_1(c,2\delta,\eta)>0$ such that 
\begin{align}
\mathbb{P}(\cD_k(\eta))\le \exp(-\eps_1k)\label{eq: cd_k_eta}    
\end{align}
Further, let $\cM_{10}^{(0)}$ be the family of components of order at least $d^{10}$ in $G_0$. Let $\cF$ be the event that $v(\cM_{10}^{(0)})\le \frac{2y(c-2\delta)n}{3}$. By Corollary \ref{l: W(p)}, 
\begin{align}
\mathbb{P}\left(\cF\right)\le \exp\left(-n/d^{3/4}\right)\label{eq: cf 3/4}    
\end{align}
Note that $\cF^c$ implies that $v(\cM_{2}^{(0)})\ge v(\cM_{10}^{(0)})>\frac{2y(c-2\delta)n}{3}$.

Now, denote by $\cE_k(\eta)$ the event that there is a component-respecting partition $A,B$ of $V(\cM_2^{(0)})$ such that $\eta k \le |A|\le |B|$ and there are no paths between $A$ and $B$ in $G_1$.
Note that $\cB_k\cap \cD_k(\eta)^c\cap \cF^c\subseteq \cE_k(\eta)$: indeed, let $S$ be the set whose existence is claimed by $\cB_k$. We can set $A$ to be the set of vertices in $G_0[S]$ in components in $\cM_2^{(0)}$. By $\cD_k(\eta)^c$ there are at least $\eta k$ such vertices and since $|A|\le |S|\le k\le y(c-\delta)n/4$ we have by $\cF^c$ that $|B|=v(\cM_2^{(0)}\setminus A)\ge v(\cM_2^{0})-|A|\ge |A|$, and by $\cB_k$ there are no edges in $G_1$ between $S$ and $V(L_1)\setminus S$ (in fact, between $S$ and $V\setminus S$).
However, by \Cref{lem:estimate_E2} (applied with $G_0,G_1$ instead of $G_1,G_2$), we have that there is $\eps_2 = \eps_2(c-\delta,\delta)>0$ such that $\mathbb P(\cE_k(\eta))\le \exp(-\eps_2 \eta k\log(n/\eta k)/d)$.
By combining the latter observation with~\eqref{eq: estimate on Akeps}, \eqref{eq: cd_k_eta}, and \eqref{eq: cf 3/4}, we have
\[\begin{split}
\mathbb P(\cA_k(\eps))
&\le (\mathbb P(\cD_k(\eta))+\mathbb{P}(\cF)+\mathbb P(\cB_k\cap \cD_k(\eta)^c\cap \cF^c))\exp\bigg(\frac{\log(c/\delta)\eps k\log_2(n/k)}{d}\bigg)\\
&\le (\mathbb P(\cD_k(\eta))+\mathbb P(\cE_k(\eta)))\exp\bigg(\frac{\log(c/\delta)\eps k\log_2(n/k)}{d}\bigg)\\
&\le \bigg(2\exp(-\eps_1 k)+\exp\bigg(-\frac{\eps_2\eta k\log(n/\eta k)}{d}\bigg)\bigg)\exp\bigg(\frac{\log(c/\delta)\eps k\log_2(n/k)}{d}\bigg).
\end{split}\]
Thus, by choosing $\eps = \eps(c,\delta,\eps_1,\eps_2,\eta)$ suitably small shows that $\mathbb P(\cA_k(\eps))=o(1/n)$. The union bound over the less than $n$ values of $k$ completes the proof.
\end{proof}

Our approach also requires some control on the expansion of sets in $L_1$ which are larger than $v(L_1)/2$. This is ensured by the following proposition.

\begin{proposition}\label{prop: expansion large}
For every $\nu\in (0,1)$, there exists a constant $\eps = \eps(c,\nu)>0$ such that \textbf{whp} every set $S\subseteq V(L_1)$ of size $|S|\eqqcolon k\in [\nu v(L_1),(1-\nu)v(L_1)]$ satisfies that $e_{L_1}(S,V(L_1)\setminus S)\ge \eps k/d$.
\end{proposition}
\begin{proof}
Fix $k\in [\nu v(L_1),(1-\nu)v(L_1)]$ and recall from \Cref{lem:yb} that the survival probability $y(c)$ defined in \Cref{th: large tail deviation} satisfies that, for every $c>1$ and $\delta=\delta(c)>0$ small, $y(c)-y(c-\delta)\le \eta \delta$ for some fixed constant $\eta = \eta(c)>0$.

Recalling $\cM_{10}$ from \Cref{subsec: vol}, fix $\beta = \beta(c,\nu)>0$ sufficiently large and define
\begin{align*}
    \mathcal{A}_k(\eps)&\coloneqq\{\exists S\subseteq V(L_1), |S|=k\colon e_{L_1}(S,V(L_1)\setminus S)<\eps k/d\},\\
    \mathcal{B}_k(\beta)&\coloneqq\left\{v(\cM_{10})\ge y(c)n+k/\beta\right\}.
\end{align*}
By \Cref{l: W(p)},
\begin{align}
    \mathbb{P}\left(\mathcal{B}_k(\beta)\right)\le 2\exp(-k^2/(4\beta^2\cdot 40cd^{1/2}n)).\label{eq: bk}
\end{align}

Fix $\delta = \delta(c,\nu)>0$ small, recall $p_1=(c-\delta)/d$ and set $\rho=p_1/p=1-\delta/c$. 
Denote by $\mathcal{C}_k(\beta)$ the event that the largest component $L_1^-$ of $G_1$ has order at most $y(c-\delta)n-k/\beta$. Then, by Theorem \ref{th: large tail deviation}, there is a constant $\eps_1=\eps_1(c-\delta)>0$ such that
\begin{align}\label{eq:boundcC}
\mathbb{P}\left(\mathcal{C}_k(\beta)\right)\le \exp(-\eps_1 (k/\beta)\log(\beta n/k)/d).
\end{align}
Denote by $\cD_k$ the event that some set $S\subseteq V(L_1)$ of size $k$ connected in $L_1$ satisfies $e_{G_1}(S,V(L_1)\setminus S)=0$.
On the event $\cD_k$, the largest component $L_1^-$ of $G_1$ has order at most $\max\{k,v(\cM_{10})-k\}$. Then, by assuming $\cB_k(\beta)^c\cap \cC_k(\beta)^c$, we have that $v(L_1)\le v(\cM_{10})< y(c)n+k/\beta$ and $v(L_1)\ge v(L_1^-)>y(c-\delta)n-k/\beta$. Thus, assuming $\cB_k(\beta)^c\cap \cC_k(\beta)^c$ we have that
\begin{align*}
\max\{k,v(\cM_{10})-k\}&\le \max\{(1-\nu)v(L_1), v(\cM_{10})-\nu v(L_1)\}\\
&\le \max\{(1-\nu)v(\cM_{10}), v(\cM_{10})-\nu v(L_1^-)\}\\
&=v(\cM_{10})-\nu v(L_1^-)\\
&< y(c)n+k/\beta -\nu\left(y(c-\delta)n-k/\beta\right)\\
&\le (1-\nu)y(c-\delta)n+(1+\nu)k/\beta+\eta\delta n.
\end{align*}
Now, for $\delta=\delta(c,\nu)$ sufficiently small, $\eta\delta n\le \nu y(c-\delta)n/2$. Further, for $\beta=\beta(c,\nu)$ sufficiently large, $(2+\nu)k/\beta \le 3v(L_1)/\beta \le \nu y(c-\delta)n/2$. Thus, assuming $\cB_k(\beta)^c\cap \cC_k(\beta)^c$,
\begin{align*}
    \max\{k,v(\cM_{10})-k\}<y(c-\delta)-k/\beta.
\end{align*}
We thus have that $\cD_k\cap \cB_k(\beta)^c\cap \cC_k(\beta)^c$ contradicts $\cC_k(\beta)^c$, and therefore
\[\cD_k\subseteq \cB_k(\beta)\cup \cC_k(\beta),\]
which implies that
\[\mathbb P(\cD_k\mid\cA_k(\eps))\mathbb P(\cA_k(\eps))\le \mathbb P(\cD_k)\le \mathbb P(\cB_k(\beta))+\mathbb P(\cC_k(\beta))\qquad\Longrightarrow\qquad \mathbb P(\cA_k(\eps))\le \frac{\mathbb P(\cB_k(\beta))+\mathbb P(\cC_k(\beta))}{\mathbb P(\cD_k\mid\cA_k(\eps))}.\]
At the same time, we also have that
\begin{equation}\label{eq:D_k}
\mathbb P(\cD_k\mid \cA_k(\eps))\ge (1-\rho)^{\eps k/d} = \exp(-\log(c/\delta)\eps k/d).
\end{equation}
To complete the proof note that, first, by \eqref{eq: bk} and \eqref{eq:D_k}, we have that $\mathbb P(\cB_k(\beta)) = o(\mathbb P(\cD_k\mid\cA_k(\eps))/n)$.
Second, by choosing $\eps = \eps(c,\beta,\delta)$ suitably small and using~\eqref{eq:boundcC} and~\eqref{eq:D_k}, we have that $\mathbb P(\cC_k(\beta)) = o(\mathbb P(\cD_k\mid\cA_k(\eps))/n)$. We then have that $P(\cA_k(\eps)=o(1/n)$. The union bound over the less than $n$ possible values of $k$ completes the proof.
\end{proof}

We conclude this section with the following corollary restating \Cref{thm: old expansion}, \Cref{prop: expansion} and \Cref{prop: expansion large} in a common framework for the purposes of \Cref{subsec: mixing}.

\begin{corollary}\label{cor: new expansion}
There are constants $\eps_1 = \eps_1(c)\in (0,1)$, $K_1 = K_1(c)>0$, and $\gamma=\gamma(c)>0$ such that, for every $\nu>0$, there exists $\eps_2 = \eps_2(\nu,c)>0$ such that each of the following holds \textbf{whp}:
\begin{enumerate}[label=\upshape(\alph*)]
    \item For every $k \in [K_1d,n^{\eps_1}]$ and every set $S\subseteq V(L_1)$ of size $k$ connected in $L_1$,
    \begin{align*}
        |N_{G_2}(S)|\ge \eps_2 k.
    \end{align*} \label{item: small}
    \item For every $k \in [n^{\eps_1},n^{1-\gamma}]$ and every set $S\subseteq V(L_1)$ of size $k$ connected in $L_1$
    \begin{align*}
        e_{G_2}(S, V(L_1)\setminus S)\ge \frac{\eps_2k\log_2(n/k)}{d\log d}.
    \end{align*} \label{item: medium}
    \item For every $k \in [n^{1-\gamma},(1-\nu)v(L_1)]$ and every set $S\subseteq V(L_1)$ of size $k$ connected in $L_1$
    \begin{align*}
        e_{G_2}(S, V(L_1)\setminus S)\ge \frac{\eps_2k\log_2(n/k)}{d}.
    \end{align*} \label{item: large}
\end{enumerate}
\end{corollary}

\subsection{Proof of Theorem \ref{th: main}\ref{th: mixing time}}\label{subsec: mixing}
First, we introduce some terminology related to Markov chains (for a comprehensive introduction, see~\cite{LPW17}). 
For a graph $G$ with vertex set $V$ and a vertex $u\in V$, the \emph{lazy (simple) random walk on $G$ starting at $u$} is a Markov chain with initial state $u$ and such that, 
at every step, the walk stays at its current state $v$ with probability $1/2$ and moves to a uniformly chosen neighbour of $v$ in $G$ with probability $1/(2\deg(v))$. 
If $G$ is a connected graph, the lazy random walk on $G$ is an irreducible, reversible and ergodic Markov chain with \emph{stationary distribution} $\pi$ on $V$ where $\pi(v) = \deg(v)/(2e(G))$.
Our goal is to determine the speed of convergence of the lazy random walk to its stationary distribution.
To this end, we define the \emph{total variation distance} $\dtv$ between two distributions $\pi_1$ and $\pi_2$ on $V$ by setting 
\[
\dtv(\pi_1,\pi_2):=\max_{A\subseteq V}|\pi_1(A)-\pi_2(A)|.
\]
For $t\ge 0$ and $u\in V$, we denote by $P^t(u,\cdot)$ the distribution of the $t$-th state of the lazy random walk on $G$ starting at $u$.
Then, the \textit{mixing time} of the lazy random walk is defined as
\[\tmix:=\min\{t\ge 0:d(t)\le 1/4\}\qquad\text{where, for every $t\ge 0$,}\qquad d(t):= \max_{u\in V}\dtv(P^t(u,\cdot),\pi).\]
For every set $S\subseteq V$, observe that
\begin{align*}
\pi(S)=\sum_{v\in S}\pi(v)=\frac{2e(S)+e(S,V\setminus S)}{2e(G)} \quad \text{and set} \quad 
Q(S):=\sum_{v\in S, u\in V\setminus S}\pi(v)P^1(v,u)=\frac{e(S,V\setminus S)}{4e(G)}.
\end{align*}
The \textit{conductance} $\Phi(S)$ of $S$ is then given by
\begin{align*}
    \Phi(S):=\frac{Q(S)}{\pi(S)\pi(V\setminus S)}.
\end{align*}
Note that, since $Q(S)=Q(V\setminus S)$ for every set $S\subseteq V$, we also have $\Phi(S)=\Phi(V\setminus S)$. 
Further,~we~denote $\pi_{\min}=\min_{v\in V} \pi(v)$ and, for all $\rho\in [\pi_{\min},1]$, we define
\begin{align*}
    \Phi(\rho):=\min\left\{\Phi(S): S\subseteq V, \rho/2\le \pi(S)\le \rho, \text{$S$ is connected in } G\right\},
\end{align*}
and if no set $S$ with the required properties exists for some $\rho\in [\pi_{\min},1]$, we set $\Phi(\rho)=1$. The following theorem of Fountoulakis and Reed~\cite{FR07a} bounds the mixing time of the lazy random walk on $G$ through the conductance of connected sets.

\begin{thm}[Theorem 1 of \cite{FR07a}] \label{th: mixing-time-tool}
\textit{There exists an absolute constant $K>0$ such that
\begin{align*}
    \tmix\le K\sum_{j=1}^{\lceil \log_2(\pi_{\min}^{-1})\rceil} \frac{1}{\Phi^2(2^{-j})}.
\end{align*}}
\end{thm}

Throughout the rest of this section, we study the mixing time of the lazy random walk on the giant component $L_1$ in $G_2$ and write $e(S)$ and $e(S,R)$ for $e_{L_1}(S)$ and $e_{L_1}(S,R)$, respectively. 

We now aim to bound $\Phi(\rho)$. We require the following estimates on the stationary distribution.
\begin{claim}\label{l: pi(S) and |S|} 
For every $K_1>1$, there is $K_2=K_2(K_1)>0$ such that \whp the following holds for any set $S\subseteq V(L_1)$ connected in $L_1$:
\begin{enumerate}[label=\upshape(\alph*)]
    \item If $\pi(S)\in [K_2d/n,1]$, then $|S|\ge K_1\pi(S)n/K_2$.
    \item If $\pi(S)\le 1/2$, then $|S|\le (1-K_2^{-1})v(L_1)$.
\end{enumerate}
\end{claim}
\begin{proof}
First of all, by \Cref{th: large tail deviation}, \whp we have $v(L_1)\ge 0.9yn$ with $y=y(c)$.
We assume this event in the sequel.

We prove (a) by using a first moment argument. 
Set $\xi = \xi(d)\in [d,n/K_2]$ and note that every set $S$ with $\pi(S)\ge K_2\xi/n$ satisfies
\[e(S)+e(S,V(L_1)\setminus S)\ge \frac{2e(S)+e(S,V(L_1)\setminus S)}{2} = \pi(S) e(L_1)\ge \pi(S) v(L_1)/2\ge K_2y \xi/3.\]
Hence, it is enough to show that, for $K_2$ suitably large, \textbf{whp} no set $S\subseteq L_1$ connected in $G_2$ satisfies simultaneously $e(S)+e(S,V(L_1)\setminus S)\ge K_2y\xi/3$ and $|S|\eqqcolon k< K_1 \xi$.
Indeed, by using \Cref{l: trees} and choosing $K_2$ large, the expected number of such sets is bounded from above by
\begin{equation}\label{eq:fsm_mom}
\begin{split}
\sum_{k=1}^{K_1\xi} n (\e d)^{k-1} p^{k-1} \cdot \binom{kd}{K_2y \xi/3-(k-1)} &p^{K_2y \xi/3-(k-1)}
\le \sum_{k=1}^{K_1\xi} n (\e c)^{k-1} \cdot \bigg(\frac{\e k c}{K_2y \xi/6}\bigg)^{K_2y \xi/6}\\
&\le K_1\xi\cdot n (\e c)^{K_1\xi} \cdot \bigg(\frac{6\e c K_1}{K_2 y}\bigg)^{K_2y \xi/6}=o(1),
\end{split}
\end{equation}
where the first inequality uses the fact that $\tbinom{a}{b} p^b\le (\e a p/b)^b$ for all integers $a\ge b\ge 1$ and the fact that $k\le K_1\xi\le  K_2y\xi/6$.
Part (a) now follows from Markov's inequality.

We move to part (b). 
It is enough to show that \whp no set $S\subseteq V(L_1)$ satisfies $\pi(V(L_1)\setminus S) \geq 1/2$ and $|V(L_1\setminus S)| \le K_2^{-1}v(L_1)$.
By~\cite[Theorem~3.1]{DKL25}, $G_2$ {\it converges locally in probability} to a Galton-Watson process with offspring distribution Poisson($c$); see the definition of the local convergence in probability in~\cite{DKL25}. 
Here, we will need the following simple corollary from \cite[Theorem~3.1]{DKL25}: for every integer $i\ge 0$, \textbf{whp} the number of vertices in $G_2$ with degree $i$ is at most $n\exp(-\Theta(\sqrt{\log d}))$-far from $n\cdot\mathbb{P}(\mathrm{Po}(c)=i)$.

Now, denote by $\ell = \ell(K_2)$ the unique integer satisfying
\[
\mathbb P(\mathrm{Po}(c)\ge \ell) > K_2^{-1} \ge \mathbb P(\mathrm{Po}(c)\ge \ell+1).
\]
Let $D$ be the set of all vertices in $Q^d$ which have degree at least $\ell$ in $G_2$. Since $K_2^{-1}<\mathbb{P}(\mathrm{Po}(c)\geq\ell)$, \textbf{whp} $K_2^{-1}n<|D|$.
Take any $U\subseteq V(L_1)$ of size at most $K_2^{-1}n<|D|$. Every vertex in $U\setminus D$ has degree strictly less than $\ell$. Therefore, \textbf{whp} for every such set $U$,
$$
 \sum_{u\in U} \deg_{L_1}(u)\le
 \sum_{u\in D} \deg (u)\le
 (1+o(1))\sum_{i=\ell}^{\infty}i\cdot \mathbb{P}(\mathrm{Po}(c)=i)\cdot n
=(1+o(1))n 
 \sum_{i=\ell}^{\infty} \e^{-c} \frac{c^{i}}{i!} \cdot i.
$$

Recall that $\mathbb{E}[\eta\1_{\eta\geq \ell}]\to 0$ as $\ell\to\infty$ when $\mathbb{E}|\eta|<\infty$.
Then, \textbf{whp}, for every set $U\subseteq V(L_1)$ such that $|U|\le K_2^{-1}v(L_1)\le K_2^{-1}n$ and large enough $K_2$, we have that 
\[\pi(U) = \frac{1}{2e(L_1)}\sum_{u\in U} \deg_{L_1}(u)\le \frac{1}{0.9y}\bigg((1+o(1))\sum_{i=\ell}^{\infty} \e^{-c} \frac{c^{i}}{i!} \cdot i\bigg) < \frac{1}{2},\]
which finishes the proof of part (b).
\end{proof}

\begin{remark}\label{rem:sparse}
We note that, by the first moment argument presented in~\eqref{eq:fsm_mom} with $K_2$ suitably large, it also follows that \textbf{whp}, for every set $S$ of size $|S|\ge K_1d$ connected in $L_1$, we have that $e(S,V(L_1))=2e(S)+e(S,V(L_1)\setminus S)\le K_2|S|$. 
\end{remark}

Next, we use Corollary \ref{cor: new expansion} to translate the said bounds on $\pi(S)$ to bounds on $\Phi(S)$.

\begin{lemma}\label{l: bounding Phi(S)}
There are constants $\gamma,\eps,K>0$ such that \textbf{whp}, for every set $S\subseteq V(L_1)$ connected in $L_1$, each of the following holds.
\begin{enumerate}[label=\upshape(\alph*)]
    \item If $\pi(S)\in [Kd/n,K n^{-\gamma}]$, then $\Phi(S)\ge \frac{\eps}{d\log d}\cdot\log_2\left(\frac{1}{3c\pi(S)}\right)$. \label{phi small and medium}
    \item If $\pi(S)\in [K n^{-\gamma},1/2]$, then $\Phi(S)\ge \frac{\eps}{d}\cdot\log_2\left(\frac{1}{3c\pi(S)}\right)$. \label{phi large}
\end{enumerate}
\end{lemma}
\begin{proof}
Fix $K_1>1$ and $K_2=K_2(K_1)>0$ suitably large so that \Cref{l: pi(S) and |S|} and \Cref{rem:sparse} jointly hold. Further, fix $\gamma,\nu=1/K_2$ and $\eps_2(K_2)$ as in Corollary \ref{cor: new expansion}.
Then, \whp every set $S\subseteq V(L_1)$ which is connected in $L_1$ and has $\pi(S)\in [K_2d/n,1/2]$ satisfies $|S|\in [K_1d, (1-K_2^{-1})v(L_1)]$, by Claim \ref{l: pi(S) and |S|}. 
By Corollary \ref{cor: new expansion}, \textbf{whp}, for every set $S$ connected~in~$L_1$ with $|S|\in [K_1d,K_2n^{1-\gamma}]$, $e(S,V(L_1)\setminus S)\ge \frac{\eps_2 |S|\log_2(n/|S|)}{d\log d}$. 
By a combination of the last bound and \Cref{rem:sparse}, \textbf{whp}, for all such sets~$S$,
\begin{align}
    \Phi(S)=\frac{Q(S)}{\pi(S)\pi(V(L_1)\setminus S)}=\frac{e(S,V(L_1)\setminus S)}{2\left(2e(S)+e(S,V(L_1)\setminus S)\right)\pi(V(L_1)\setminus S)}
    &\ge \frac{\eps_2 |S|\log_2(n/|S|)/(d\log d)}{2K_2 |S|\pi(V(L_1)\setminus S)}\notag\\
    &\ge \frac{\eps_2\log_2(n/|S|)}{2K_2d\log d}. \label{eq: PHI(S)}
\end{align}
Now, since $S$ is connected in $L_1$, $2e(S)+e(S,V(L_1)\setminus S)\ge |S|-1$. Further, by standard bounds on the tails of the Binomial distribution, \textbf{whp} $e(L_1)\le e(Q^d_p)\le 2cn$. Thus, $\pi(S)\ge \frac{|S|-1}{2cn}\ge \frac{|S|}{3cn}$, and thus
\begin{align*}
    \Phi(S)\ge\frac{\eps_2}{2K_2d\log d}\cdot \log_2\left(\frac{1}{3c\pi(S)}\right).
\end{align*}

Next, fix a set $S$ with $\pi(S)\in [K_2 n^{-\gamma}/K_1,1/2]$. 
Then, \Cref{l: pi(S) and |S|} implies that \textbf{whp} for every such set $S$,
\[|S|\in [K_1 \pi(S)n/K_2, (1-K_2^{-1})v(L_1)]\subseteq [n^{1-\gamma},(1-K_2^{-1})v(L_1)].\] 
In turn, \Cref{cor: new expansion}(c) ensures that $e(S,V(L_1)\setminus S)\ge \frac{\eps_2 |S|\log_2(n/|S|)}{d}$.
By combining the latter observations with \Cref{rem:sparse}, we obtain (in a similar way to \eqref{eq: PHI(S)}) that \whp every set $S$ connected in $L_1$ with $\pi(S)\in[K_2 n^{-\gamma}/K_1,1/2]$ satisfies
\begin{align*}
    \Phi(S)\ge \frac{\eps_2 |S|\log_2(n/|S|)/d}{2K_2|S|\pi(V(L_1)\setminus S)}\ge  \frac{\eps_2}{2K_2d}\cdot\log_2\left(\frac{1}{3c\pi(S)}\right),
\end{align*}
where the last inequality used once again that $\pi(S)\ge \frac{|S|}{3cn}$
Choosing $\eps=\eps_2/(2K_2)$ and $K=\max\{K_1,K_2/K_1\}$ finishes the proof.
\end{proof}

We are now ready to prove Theorem \ref{th: main}\ref{th: mixing time}.
\begin{proof}[Proof of Theorem \ref{th: main}\ref{th: mixing time}]
As explained in the introduction, it suffices to show the upper bound.
By Theorem \ref{th: mixing-time-tool}, there is an absolute constant $K'>0$ such that
\begin{align}
     \tmix\le K'\sum_{j=1}^{\lceil \log_2(\pi_{\min}^{-1})\rceil}\Phi^{-2}\left(2^{-j}\right). \label{eq: A}
\end{align}
Fix $\gamma,\eps,K$ as in Lemma \ref{l: bounding Phi(S)}. 
For every integer $j$ such that $2^{-j}\in [2Kd/n,Kn^{-\gamma}]$ (corresponding to $j\in [j_1+1,j_2] := [\gamma d-\log_2 K,d-\log_2(2Kd)]$), we have that
\[\Phi^{-2}(2^{-j})\le \left(\frac{d\log d}{\eps \log_2(2^{j}/(3c))}\right)^2\le \frac{\log_2(3c)d^2 (\log d)^2}{\eps^2 j^2}.\]
Similarly, for all $j$ such that $2^{-j}\in (Kn^{-\gamma},1/2]$ (corresponding to $j\in [j_1]:=[1,\gamma d-\log_2 K]$), we have that
\[\Phi^{-2}(2^{-j})\le \frac{\log_2(3c)d^2}{\eps^2 j^2}.\]
Thus, \textbf{whp}
\begin{align}
    \sum_{j=1}^{\lceil \log_2(\pi_{\min}^{-1})\rceil}\Phi^{-2}\left(2^{-j}\right)=\sum_{j=1}^{j_1}\frac{\log_2(3c)d^2}{\eps^2 j^2}+\sum_{j=j_1+1}^{j_2}\frac{\log_2(3c)d^2 (\log d)^2}{\eps^2 j^2}+\sum_{j=j_2+1}^{\lceil \log_2(\pi_{\min}^{-1})\rceil}\Phi^{-2}\left(2^{-j}\right). \label{eq: B}
\end{align}
The first two terms are of order $O(d^2)$. We now estimate the third sum. Since $L_1$ is connected and \textbf{whp} $e(L_1) < 2cn$, \textbf{whp} for every $S\subseteq V(L_1)$ we have (similarly
\[
\Phi(S)=\Phi(V(L_1)\setminus V(S))= \frac{e(S,V(L_1)\setminus V(S))}{2(2e(V(L_1)\setminus V(S))+e(S,V(L_1)\setminus V(S)))\pi(S)}\ge \frac{1}{4e(L_1)\pi(S)} \geq \frac{1}{8cn \pi(S)} .\]
Hence, \textbf{whp} for every $S$ with $\pi(S) \leq 2^{-j}$, we have $\Phi(S) \ge \frac{2^j}{8cn}$ and so $\Phi\left(2^{-j}\right) \geq \frac{2^j}{8cn}$. Therefore, \textbf{whp}
\begin{align}
\sum_{j=j_2+1}^{\lceil \log_2(\pi_{\min}^{-1})\rceil}\Phi^{-2}\left(2^{-j}\right)\le  
    2 \left(\frac{8c n}{2^{j_2}}\right)^2 =O(d^2). \label{eq: D}
\end{align}
In total, the right hand side of~\eqref{eq: A} is of order $O(d^2)$ \whp and thus $\tmix = O(d^2)$ \whp, as desired.
\end{proof}

\section{Diameter}\label{s: diameter}
Recall that $\textbf{1}$ denotes the all-1 vertex and $\textbf{0}$ denotes the all-0 vertex. 
For a vertex $u$ and $i\in [d]$, we write $u(i)$ to denote the $i$-th coordinate of $u$.
For a pair of vertices $u,v\in Q^d$, we denote by $Q(u,v)$ the smallest subcube in $Q^d$ containing $u$ and $v$ (where we often treat $u$ as the all-0 vertex of $Q(u,v)$, and $v$ as the all-1 vertex of $Q(u,v)$), and write $\Id(u,v)$ for the set of coordinates where $u$ and $v$ coincide.

In \cite{ADEK23}, it was shown that if $p=\frac{c}{d}$ with $c>\e$, then with probability bounded away from zero, $\mathbf{0}$ is connected to $\mathbf{1}$ in a 'monotone' path of length $d$. It was further shown that if $p=\frac{c}{d}$ with $c<e$, then \textbf{whp} this does not hold. The next result, which is the key technical lemma in this section, shows that even when $1<c\le e$, there is a non-negligible probability of connecting $\mathbf{0}$ and $\mathbf{1}$ with a path of length $O(d)$. Its proof is delayed to Subsection \ref{sec:proofs_aux_lemmas}. 
\begin{lemma}\label{lem:distance_andipodas}
Fix $c>1$ and let $p=p(d)=c/d$. There are constants $K_1=K_1(c)>0$ and $K_2=K_2(c)>0$ such that $\mathbf{0}$ is connected to $\mathbf{1}$ in $Q_p^d$ by a path of length at most $K_1d$ with probability at least $d^{-K_2}$.
\end{lemma}

Next, we bootstrap the conclusion of \Cref{lem:distance_andipodas} to any pair of vertices in the cube $Q^d$.

\begin{lemma}\label{lem:bootstrap_7.1}
Fix $c>1$ and let $p=p(d)=c/d$. There are constants $K_3=K_3(c)>0$ and $K_4=K_4(c)>0$ such that, for every pair of vertices $u,v\in V(Q^d)$, with probability at least $d^{-K_4}$ there is a path of length at most $K_3d$ between $u$ and $v$ in $Q^d_p$. 
\end{lemma}
\begin{proof}
Fix a suitably small constant $\eps=\eps(c)$ (such that, in particular, $(1-\eps/2)c>1$). 
\begin{claim}\label{cl:chain}
For every pair of vertices $u,v\in V(Q^d)$, there is a sequence of vertices $u=v_1,v_2,...,v_\ell=v$ with $\ell\leq 1+2/\eps$ such that, for every $i\in [\ell-1]$, $v_i$ and $v_{i+1}$ differ in at least $(1-\eps/2)d$ coordinates. 
\end{claim}
\begin{proof}[Proof of \Cref{cl:chain}]
First, consider a pair of vertices $u,v$ for which the set of coordinates $J\subseteq [d]$ where they differ has size $|J|\le \eps d$. 
Then $u$ and $v$ satisfy the statement with $\ell=3$: indeed, there is a vertex $w$ which differs from both $u$ and $v$ in each coordinate in $[d]\setminus J$ and in $|J|/2$ of the coordinates in $J$.  
Moreover, for every pair of vertices $u,v$, there exists a sequence of vertices $u=v_1',v_2',...,v_{\ell'}'=v$ with $\ell'\le 1+1/\eps$ where every pair of consecutive vertices differ in at most $\eps d$ coordinates. 
Thus, for every pair of vertices $u,v$, we can extend the sequence $u=v_1',\ldots, v_{\ell'}'=v$ to a sequence $u=v_1', v_{1,2}', v_2', \ldots,v_{\ell'-1}', v_{\ell'-1,\ell'}'  v_{\ell'}'$ with $2\ell'-1\le 1+2/\eps$ vertices where $v_{i,i+1}'$ is a vertex which differs from both $v_i'$ and $v_{i+1}'$ by at least $(1-\eps/2)d$ coordinates. This completes the proof of the claim.
\end{proof}

Fix a pair of vertices $u,v$ and a sequence $v_1,\ldots,v_\ell$ as in \Cref{cl:chain}.
Then, for every $i\in [\ell-1]$, $Q(v_i,v_{i+1})$ has dimension at least $(1-\eps/2)d$.
Set $c'\coloneqq p\cdot (1-\eps/2)d =(1-\eps/2)c$ and note that $c'>1$ by the choice of $\eps$. 
By \Cref{lem:distance_andipodas} applied to $Q(v_i,v_{i+1})$ (whose dimension is $D\in [(1-\eps/2)d,d]$ and satisfies $pD\geq c'$), there are constants $K_1'=K_1'(c')>0$ and $K_2'=K_2'(c)>0$ such that $v_i$ and $v_{i+1}$ are connected by a path of length at most $K_1'D\leq K_1'd$ in $Q(v_i,v_{i+1})_p$ with probability at least $D^{-K_2'}\geq d^{-K_2'}$. 
As each of the latter events is increasing, Harris' inequality (\Cref{lem:Harris}) implies that the probability of having a path of length at most $\ell K_1' d$ between $u$ and $v$ is at least $d^{-\ell K_2'}$.
Setting $K_3= (1+2/\eps)K_1'\ge \ell K_1'$ and $K_4= (1+2/\eps)K_2'\ge \ell K_2'$ finishes the proof.
\end{proof}

Denote by $\cI$ the set of vertex pairs $(u,v)$ with $\Id(u,v)\ge d/3$.
Next, using \Cref{lem:bootstrap_7.1}, we prove that typically, for all pairs of vertices $u,v$ in $\cI$ where both $u$ and $v$ can reach many vertices via short paths in $Q^d_p$, 
the distance between $u$ and $v$ in $Q^d_p$ is $O(d)$.

\begin{lemma}\label{lem:diameter_for a pair}
Fix $c>1$ and let $p=p(d)=c/d$. There is a constant $K_5=K_5(c)>0$ such that the following holds \textbf{whp}: for every pair $(u,v)\in \cI$, if for each $w\in \{u,v\}$ the number of vertices within distance $(\log d)^5$ from $w$ in $Q_p^d$ is at least $\exp((\log d)^4)$, then there is a path of length at most $K_5d$ between $u$ and $v$ in $Q^d_p$.
\end{lemma}
\begin{proof}
Fix a pair of vertices $(u,v)\in \cI$. For a vertex $w$ and a subset $I\subseteq [d]$, we define $Q(w,I)$ to be the cube of dimension $d-|I|$ containing all vertices which agree with $w$ in each coordinate in $I$ (and where the other coordinates vary). 
In addition, denote by $\mathcal{E}_I=\mathcal{E}_I(u,v)$ the event that, for each $w\in \{u,v\}$, the number of vertices within distance $(\log d)^5$ from $w$ in $Q(w,I)_p$ is at least $\exp((\log d)^4)/\max\{|I|,1\}$. Observe that $\cE_{\varnothing}$ is the event that for each $w\in \{u,v\}$, the number of vertices within distance $(\log d)^5$ from $w$ in $Q^d_p$ is at least $\exp((\log d)^4)$.

Set $s=d/(\log d)^2$ and consider a family of disjoint subsets $I_1,I_2,\ldots,I_s$ of $\Id(u,v)$, each of size $d/(3s)=(\log d)^2/3$.
 
\begin{claim}\label{claim:smallerevents}
$\cE_\varnothing\subseteq \bigcup_{j=1}^s \cE_{I_j}$.
\end{claim}
\begin{proof}
Assume that the event $\mathcal{E}_\varnothing$ holds. 
Then, for each $w\in \{u,v\}$, there is a tree rooted at $w$ in $Q^d_p$, denoted by $T_w$, which has depth at most $(\log d)^5$ and contains $\exp((\log d)^4)$ vertices.
For each vertex $x\in V(T_w)$, we denote by $P_x$ the unique path from $x$ to the root $w$ in $T_w$. 
We also denote by $L_w(x)$ the set of coordinates $i$ such that some vertex $x'$ on $P_x$ satisfies $x'(i)\neq w(i)$.
Note that $|L_w(x)|\le e(P_x)\le (\log d)^5$. Furthermore, if $L_w(x)$ does not intersect $I_j$, then $P_x\subseteq Q(w,I_j)_p$.
In particular, for every $x\in V(T_w)$, $P_x\not\subseteq Q(w,I_j)_p$ for at most $(\log d)^5$ indices $j\in [s]$.

For each $w\in \{u,v\}$, denote by $B_w$ the set of indices $j\in [s]$ for which the event $\cE_{I_j}$ fails for $w$. The above implies that, for each $w\in \{u,v\}$,
\[ (\log d)^5 \cdot |V(T_w)|  \ge \sum_{j\in B_w} \left(1-\frac{1}{|I_j|}\right)|V(T_w)|\ge \frac{|B_w||V(T_w)|}{2},\]
where the first inequality is obtained by counting the number of pairs $(I_j,x)$ such that $P_x\not\subseteq Q(w,I_j)_p$ in two different ways: for the upper bound in the left hand side, we fix $x$ and bound the number of choices of $I_j$ by $(\log d)^5$, and for the lower bound in the right hand side, we fix $j\in B_w$ and use the definition of $B_w$ to bound from below the number of vertices in $T_w$ which are not reachable in $Q(w,I_j)_p$ by $|V(T_w)|-|V(T_w)|/|I_j|$.
Hence, $|B_w|\leq 2(\log d)^5$ for each $w\in \{u,v\}$ and therefore, the event $\mathcal{E}_{I_j}$ holds for at least $s-4(\log d)^5\ge 1$ indices $j\in [s]$.    
\end{proof}

Next, we fix $j\in [s]$ and let $L_j\coloneqq L_j(u,v)\subseteq I_j$ be an arbitrary subset of size $|I_j|/2=d/6s=(\log d)^2/6$.
We further define $r=r(d):= \binom{d/6s}{d/12s}\ge 2^{d/12s}$ and denote by $L_{1,j},L_{2,j},...,L_{r,j}$ the subsets of $L_j$ of size $|L_j|/2=d/12s$. 
Also, for each $t\in [r]$, denote by $Q_{t,j}$ the subcube of dimension $d-|I_j|$ with vertex set 
\[\{z\in V(Q^d): z(i)=u(i) \text{ for every }i\in I_j\setminus L_{t,j} \text{ and } z(i)\neq u(i) \text{ for every } i\in L_{t,j}\},\]
noting that the coordinates outside $I_j$ vary and that for every $i\in I_j$, $u(i)=v(i)$. Further observe that the cubes $Q_{1,j},\ldots,Q_{r,j}$ are pairwise vertex-disjoint.

\begin{claim}\label{claim:join_vertices_tocubes}
For every $w\in \{u,v\}$ and $t\in [r]$, the probability that $\mathcal{E}_{I_j}$ occurs and there does not exists a path $P(w,t,j)$ of length at most $d$ from $w$ to $Q_{t,j}$ in $Q^d_p$ which is edge-disjoint from all of the hypercubes $Q_{1,j},Q_{2,j},\ldots Q_{r,j}$ is $o(n^{-4})$.
\end{claim}
\begin{proof}
Fix $w\in \{u,v\}$ and $t\in [r]$.
Reveal the edges in $Q(w,I_j)_p$ and assume that the set of vertices $S_w$ within distance $(\log d)^5$ from $w$ in $Q(w,I_j)_p$ has size at least $\exp((\log d)^4)/|I_j|$ (an event which is implied by $\mathcal{E}_{I_j}$). 
For $z\in S_w$, we denote by $P_{t,z}$ a shortest path from $z$ to $Q_{t,j}$ in $Q^d$. 
By construction, these $|S_w|$ paths are of length $|L_{t,j}|=d/12s$, and edge-disjoint from each of the cubes $(Q_{l,j})_{l\neq t}$. Furthermore, for all distinct $z,z'\in S_w$, if we let $i\in [d]$ be such that $z(i)\neq z'(i)$, then $i\notin \Id(z,z')\supset I_j$. It is clear that all vertices in $P_{t,z}$ agree on coordinates outside of $I_j$. Thus, every vertex on $P_{t,z}$ differs from every vertex 
on $P_{t,z'}$ on their $i$-th coordinate (at least), that is, the $|S_w|$ paths are pairwise vertex-disjoint. Hence, since all paths $P_{t,z}$, $z\in S_w$, are edge-disjoint from $Q(w,I_j)$, and thus their appearance in $Q_p^d$ is independent of $Q(w,I_j)_p$, the probability that none of these $|S_w|$ vertex-disjoint paths belongs to $Q_p^d$ is at most
\begin{equation*}
    (1-p^{d/12s})^{\exp((\log d)^4/6)}\leq \exp\left(-(c/d)^{(\log d)^2/12}\exp((\log d)^4/6)\right) = o(n^{-4}).\qedhere
\end{equation*}
\end{proof}

Now, let us finish the proof of Lemma~\ref{lem:diameter_for a pair}. For $(u,v)\in \cI$, $j\in [s]$, $t\in [r]$, and a constant $K'>0$, denote by $\cF(u,v,j,t,K')$ the event that there exists $u',v'\in Q_{t,j}$ such that each of the following holds:
\begin{itemize}
    \item there are paths of length at most $d$ from $u$ to $u'$ and from $v$ to $v'$ in $Q^d_p$ which do not intersect any of the hypercubes $Q_{1,j},Q_{2,j},\ldots Q_{r,j}$, and
    \item there is a path of length at most $K'd$ between $u'$ and $v'$ in $(Q_{t,j})_p$.
\end{itemize}
Observe that $\cF(u,v,j,t,K')$ implies that there is a path between $u$ and $v$ in $Q^d_p$ of length at most $(K'+2)d$. Thus, for every constant $K'>0$,
the statement of the lemma fails with probability at most 
\begin{align}\label{eq:001}
    \sum_{(u,v)\in \cI}  \mathbb{P}\bigg(\cE_{\varnothing}(u,v)\cap \bigcap_{j\in [s]} \bigcap_{t\in[r]}\cF(u,v,j,t,K')^c\bigg)
\end{align}
Thereafter, by \Cref{claim:smallerevents}, we have that
\begin{align*}
\cE_{\varnothing}(u,v)\cap \bigcap_{j\in [s]} \bigcap_{t\in[r]}\cF(u,v,j,t,K')^c
&\subseteq \bigg(\bigcup_{j'\in [s]}\cE_{I_{j'}}(u,v)\bigg)\cap \bigcap_{j\in [s]} \bigcap_{t\in[r]}\cF(u,v,j,t,K')^c\\
&\subseteq \bigcup_{j'\in [s]}\bigg(\cE_{I_{j'}}(u,v)\cap  \bigcap_{t\in[r]}\cF(u,v,j',t,K')^c  \bigg).
\end{align*}
Hence, the expression in \eqref{eq:001} is bounded above by 
\begin{align}\label{eq:002}
    \sum_{(u,v)\in \cI}  \sum_{j\in[s]}\mathbb{P}\left(\cE_{I_j}\cap \bigcap_{t\in[r]}\cF(u,v,j,t,K')^c\right).
\end{align}

To that end, we estimate the probability of $\cE_{I_j}\cap \cF(u,v,j,t,K')^c$. 
We reveal all the edges of $Q_p^d$ lying outside the cubes $Q_{1,j},Q_{2,j},\ldots Q_{r,j}$. 
By \Cref{claim:join_vertices_tocubes}, for all $t\in [r]$,
with probability $1-o(n^{-4})$, the event $\cE_{I_j}$ holds and there exist paths from $u$ to $Q_{t,j}$ and from $v$ to $Q_{t,j}$ of lengths at most $d$ spanned by the set of revealed edges.
For every $t\in [r]$, denote by $u_t$ and $v_t$ the endpoints of these two paths in $Q_{t,j}$.
Then, \Cref{lem:bootstrap_7.1} implies that there exist constants $K_3',K_4'$ such that we can further connect $u_t$ and $v_t$ in $(Q_{t,j})_p$ by a path of length at most $K_3'd$ with probability at least $d^{-K_4'}$; 
note that these events are independent for different indices $t\in [r]$ since the cubes $Q_{1,j},Q_{2,j},\ldots Q_{r,j}$ are vertex-disjoint. 
Hence,  
\begin{align*}
\mathbb{P}\left(\cE_{I_j}\cap \bigcap_{t\in[r]}\cF(u,v,j,t,K')^c\right)&\le o(rn^{-4})+\left(1-d^{-K_4'}\right)^{r}\\
&\leq  o(n^{-3})+\exp(-d^{-K_4'}r)\\&\le o(n^{-3})+\exp\left(-d^{-K_4'}2^{(\log d)^2/12}\right)=o(n^{-3}),
\end{align*}
where the error term is independent of $u,v,j$ and $t$.

Therefore, by~\eqref{eq:002}, the statement of the lemma fails with probability at most
\[\begin{split}
&\sum_{(u,v)\in \cI} \sum_{j\in [s]} \mathbb P\bigg(\cE_{I_j}(u,v)\cap \bigcap_{t\in [r]} \cF(u,v,j,t,K_3')^c\bigg)\\
&\le \sum_{(u,v)\in \cI} \sum_{j\in [s]} \mathbb P\bigg( \bigcap_{t\in [r]}\cE_{I_j}(u,v)\cap \cF(u,v,j,t,K_3')^c\bigg)=o(|\cI|sn^{-3})=o(1),
\end{split}\]
as required.
\end{proof}

Before proving Theorem \ref{th: main}\ref{th: diameter}, we require one last lemma.

\begin{lemma}\label{l: distance two}
\textbf{Whp}, every vertex $v\in V(Q^d)$ is within distance two (in $Q^d$) from a vertex $u$ in $L_1$ with the following property: $u$ has at least $\exp((\log d)^4)$ vertices within distance $(\log d)^5$ in $Q^d_p$. 
\end{lemma}
\begin{proof}
Since $pd=c>1$, \textbf{whp} the second-largest component in $Q^d_p$ has order $O(d)$, see e.g.\ \cite[Theorem~13.1]{FK16}. We assume this event in the sequel.
Fix $v=\bf{0}$ and a sufficiently small $\delta>0$ satisfying $p(1-6\delta)d>1+\delta$.
Let $J=\{1,\cdots,\delta d\}$ and let $W$ be a set consisting of $\delta^2d^2/4$ vertices with support of size 2 and included in $J$. In particular, all vertices in $W$ are within distance two from $v$ in $Q^d$.
Further, for all $u,w\in W$, $Q(u,J)$ and $Q(w,J)$ are vertex-disjoint, and each has dimension $(1-\delta)d$. 

Fix $u\in W$ and consider a modified BFS process of $Q(u,J)_p$ starting from $u$ (at layer 0) and exploring one layer at a time. 
The process is defined as follows: start with $i=0$ and $Z_0=\{u\}$. For every integer $i\in [0,\delta d-1]$, 
given a set $Z_i$ of vertices reached at level $i$, define $z_i = \min\{|Z_i|,\delta d\}$ and consider the set $Z_i'$ consisting of some $z_i$ vertices in $Z_i$.
Then, explore all edges from $Z_i'$ to layer $i+1$ in $Q(u,J)_p$ and define $Z_{i+1}$ as the set of neighbours of $Z_i'$ on this next layer. 
We denote by $\cG_u$ the event that $u$ connects to at least $K_1d$ vertices on the first $(\log d)^3$ layers in this process where $K_1=K_1(\delta)$ is a constant allowing us to apply \Cref{cor: new expansion}(a) in $Q(u,J)_p$.

First, we bound from below $\mathbb P(\cG_u)$.
To this end, note that every vertex $v\in Z_i'$ with $i\le \delta d-1$ is adjacent to at least $d-|J|-i-z_i\ge (1-3\delta)d$ vertices on layer $i+1$ in $Q(u,J)$ which have no neighbour in $Z_i'\setminus \{v\}$.
Hence, as long as the BFS exploration remains below layer $\delta d$, the process stochastically dominates a Galton-Watson branching process with offspring distribution $\mathrm{Bin}((1-3\delta) d,p)$.

Classic martingale arguments for branching processes (see e.g.\ \cite[Chapter 4.3.4]{D19}) show that the BFS process reaches a layer $i\le (\log d)^2$ with $z_i = \delta d$ with probability $\eta=\eta(\delta)>0$ which does not depend on $d$; denote this event by $\cG_u'$.
Moreover, for every $i\le \delta d-1$ with $z_i\ge \delta d$, Chernoff's bound implies that
\[\mathbb P(z_{i+1}<\delta d \mid Z_i') = \mathbb P(\mathrm{Bin}((1-3\delta)d\cdot \delta d, p) < \delta d)=o(1).\]
Hence, by iterating the last observation $K_1/\delta$ times, we obtain that
\[\mathbb P(\cG_u)\ge \mathbb P(\cG_u')\cdot\mathbb P(z_i=\delta d\text{ for all }i\in [(\log d)^2, (\log d)^2+K_1/\delta]\mid \cG_u')\ge \eta-o(1).\]

Finally, conditionally on $\cG_u$, \Cref{cor: new expansion}(a) ensures that, for some $\eps_2=\eps_2(\delta)>0$, \textbf{whp} there are at least $(1+\eps_2)^{(\log d)^5-(\log d)^3}\ge \exp((\log d)^4)$ vertices at distance at most $(\log d)^5$ from $u$ in $Q(u,J)_p$ (which are in $L_1$ by assumption).
As a result, the event from the lemma fails for $v=\textbf{0}$ with probability at most 
\[\prod_{u\in W} (1-(1-o(1))\cdot\mathbb P(\cG_u)) = \prod_{u\in W} (1-\eta-o(1)) = o(n^{-1}).\]
A union bound over the $n$ vertices in $Q^d$ completes the proof.
\end{proof}

With \Cref{lem:diameter_for a pair,l: distance two} at hand, we are ready to prove Theorem \ref{th: main}\ref{th: diameter}.
\begin{proof}[Proof of Theorem \ref{th: main}\ref{th: diameter}]
As noted in the introduction, it suffices to show the upper bound.
To this end, condition on the events in \Cref{cor: new expansion} and \Cref{lem:diameter_for a pair,l: distance two}. We show that the conclusion of \Cref{th: main}\ref{th: diameter} holds deterministically under the conditioning. 

First of all, observe that, by Lemma \ref{l: distance two}, there are vertices $v_0,v_1,v_{\mathrm{mid}}\in L_1$ within distance 2 in $Q^d$ from $\textbf{0},\textbf{1}$ and from the middle layer of $Q^d$, respectively, such that each of $v_0,v_1,v_{\mathrm{mid}}$ has at least $\exp((\log d)^4)$ vertices within distance $(\log d)^5$ in $L_1$. Note that each of the vertex pairs $v_0,v_{\mathrm{mid}}$ and $v_{\mathrm{mid}},v_1$ agrees in $d/2-O(1)\ge d/3$ coordinates.
Thus, by Lemma~\ref{lem:diameter_for a pair}, there are paths of length at most $K_5d$ between $v_0$ and $v_{\mathrm{mid}}$ and between $v_{\mathrm{mid}}$ and $v_1$ in $Q^d_p$. 

Now, fix $K_1$ as in \Cref{cor: new expansion}. Note that if $|V(L_1)|<K_1d$, then $\mathrm{diam}(L_1)<K_1d$. Thus we only need to consider the case $|V(L_1)|\geq K_1d$, which we henceforward assume. Fix a vertex $v\in V(L_1)$ and denote by $\ell_v$ the smallest integer ensuring that the set $S_v$ of vertices within distance $\ell_v$ from $v$ in $L_1$ satisfies $|S_v|\ge K_1d$. Such an integer exists due the assumption  $|V(L_1)|\geq K_1d$. Observe that $K_1d\le |S_v|\le K_1d^2$ and $\ell_v\le K_1d$. 
Then, by \Cref{cor: new expansion}(a), there are at least $(1+\eps_2)^{(\log d)^5}|S_v|$ vertices in $L_1$ within distance at most $(\log d)^5$ in $L_1$ from some vertex in $S_v$, and therefore there is some vertex $u$ within distance $O(d)$ to $v$ which has at least $\exp((\log d)^4)$ vertices within distance $(\log d)^5$ in $L_1$ to it. 
Since at least one of $v_0,v_1$ is at distance less than $2d/3$ from $u$ in $Q^d$ (and hence agrees with $u$ in at least $d/3$ coordinates) \Cref{lem:diameter_for a pair} implies that there is a path between $u$ and $\{v_0,v_1\}$ in $L_1$ of length at most $K_5d$.
As a result, the vertex $v$ is at distance at most $K_1d+K_5d$ from the pair of vertices $\{v_0,v_1\}$ in $L_1$.
Since $v_0$ and $v_1$ are within distance at most $2K_5d$ in $L_1$ themselves, the diameter of $L_1$ is at most $2(K_1d+K_5d)+2K_5d=O(d)$, as desired.
\end{proof}

\subsection{Proof of Lemma \ref{lem:distance_andipodas} }\label{sec:proofs_aux_lemmas}
Throughout this section, we fix $c=1+\eps$ for small $\eps$ and $p=p(d)=c/d$. 
Note that, by monotonicity of the statement, proving \Cref{lem:distance_andipodas} for such a value of $c$ implies it for all larger values. 

For odd $k\geq 1$, we denote by $\mathcal{P}_d(k)$ the set of paths $P$ from $\textbf{0}$ to $\textbf{1}$ such that, for every $i\in [d]$, there are exactly $k$ edges $uv\in E(P)$ with $u(i)\neq v(i)$. In particular, every path in $\cP_d(k)$ has length $kd$ (and $kd+1$ vertices). 
We also define $M_d(k)$ to be the multiset containing all elements of $[d]$ and such that each element $i\in [d]$ appears exactly $k$ times in $M_d(k)$. 
We denote by $\mathcal{S}_d(k)$ the set of sequences of length $kd$ containing every element of $[d]$ exactly $k$ times, that is, the set of permutations of $M_d(k)$.
For a path $P$ in $Q^d$, we denote by $\phi(P)$ the sequence of length $e(P)$ with elements in $[d]$ where the $i$-th element in $\phi(P)$ indicates the coordinate in which the $i$-th and the $(i+1)$-st vertices of $P$ differ. 
Thus, $\phi$ defines an injective map from  $\mathcal{P}_d(k)$ to $\mathcal{S}_d(k)$.

\begin{definition}\label{def:tame}
    A path $P\in \mathcal{P}_d(k)$ is called \emph{tame} if each of the following conditions is satisfied:
    \begin{itemize}
        \item[\textbf{P1}] for every pair of vertices $u,v$ on $P$, if $u$ and $v$ differ in a single coordinate, then $u$ and $v$ are consecutive vertices on $P$, and
        \item[\textbf{P2}] for every $\ell\in [d/k^3]$ and every subsequence $\sigma$ of $\ell$ consecutive elements of $\phi(P)$, there are at least $\ell/2$ elements which appear exactly once in $\sigma$.
    \end{itemize}  
    We denote the family of tame paths by $\cP'_d(k)$.
\end{definition}
   
We prove Lemma \ref{lem:distance_andipodas} by a second moment computation involving the tame paths of length $kd$ from $\textbf{0}$ to $\textbf{1}$. 
As we will see, the property of being tame imposes additional structure which reduces the number of possible intersections, thus simplifying our arguments.

Our proof is based on the next two lemmas which are shown in Sections \ref{subsec:firstmoment} and \ref{subsec:secondmoment}, respectively.

\begin{lemma}\label{lem:firstmoment}
For every odd positive integer $k$, 
\[
     \e^{-21k} \frac{(kd)!}{(k!)^d} \leq  |\cP'_d(k)| \leq  \frac{(kd)!}{(k!)^d}.
\]
\end{lemma}

\begin{lemma}\label{lem:secondmoment}
There exist $\eta=\eta(\eps) \in (0,\eps/2)$ and an odd integer $k=k(\eps)$ such that the following holds. For every $r\in [0,kd]$ and $P\in \cP'_d(k)$, the number of paths in $\cP'_d(k)$ intersecting $P$ in exactly $r$ edges is bounded from above by $(\tfrac{1+\eta}{d})^r d^{k^5} (kd)!/(k!)^d$.
\end{lemma}

Next, we use \Cref{lem:firstmoment,lem:secondmoment} to prove Lemma \ref{lem:distance_andipodas}.

\begin{proof}[Proof of Lemma \ref{lem:distance_andipodas}]
Let  $\eta=\eta(\eps) \in (0,\eps/2)$ and $k$ be as in \Cref{lem:secondmoment}. Denote by $X$ the number of paths in $\cP'_d(k)$ contained in $Q^d_p$. Note that $\mathbb{E}[X]=|\cP'_d(k)|p^{kd}$. Furthermore, we have that
\begin{align*}
\mathbb{E}[X^2]
&= \sum_{P,P'\in \cP'_d(k)} \mathbb P(E(P)\cup E(P')\subseteq E(Q^d_p)) = \sum_{P\in \cP'_d(k)} \sum_{r=0}^{kd} \sum_{\substack{P' \in \cP'_d(k) :\\ |E(P)\cap E(P')|=r}} \mathbb P(E(P)\cup E(P')\subseteq E(Q^d_p)) \nonumber\\
&\le\sum_{P\in \cP'_d(k)} \sum_{r=0}^{kd} \bigg(\left(\frac{1+\eta}{d}\right)^r d^{k^5} \frac{(kd)!}{(k!)^d}\bigg) p^{2kd-r} = |\cP'_d(k)| p^{2kd} \sum_{r=0}^{kd} \bigg(\left(\frac{1+\eta}{1+\eps}\right)^r d^{k^5} \frac{(kd)!}{(k!)^d}\bigg)\\
&=\frac{(\mathbb{E}[X])^2}{|\cP'_d(k)|}\sum_{r=0}^{kd} \bigg(\left(\frac{1+\eta}{1+\eps}\right)^r d^{k^5} \frac{(kd)!}{(k!)^d}\bigg)\le \frac{(\mathbb{E}[X])^2}{|\cP'_d(k)|}d^{k^5}\frac{(kd)!}{(k!)^d}\cdot \frac{1+\eps}{\eps-\eta} \leq \frac{(\mathbb{E}[X])^2}{|\cP'_d(k)|}d^{k^5}\frac{(kd)!}{(k!)^d}\cdot \frac{4}{\eps},
\end{align*}
where the first inequality follows from \Cref{lem:secondmoment}, the second inequality uses that $\eta<\eps$ and the last one uses that $\eta<\eps/2\leq 1$.

Now, combining the Cauchy-Schwarz inequality with the lower bound in \Cref{lem:firstmoment} implies that
\begin{align*}
    \mathbb{P}(X\ge 1)\ge \frac{(\mathbb{E}[X])^2}{\mathbb{E}[X^2]}\ge |\cP'_d(k)|\frac{\eps}{4}d^{-k^5}\frac{(k!)^d}{(kd)!}\ge \frac{\eps}{4} d^{-k^5}\e^{-21k}\ge d^{-k^5-1}.
\end{align*}
As $k=k(\eps)$ depends only on $\eps$, this concludes the proof.
\end{proof}

\subsubsection{Proof of Lemma \ref{lem:firstmoment}}\label{subsec:firstmoment} 
First of all, note that $|\mathcal{P}_d(k)|\leq |\mathcal{S}_d(k)|\le (kd)!/(k!)^d$. Thus, we focus on the lower bound. The result holds trivially for $k=1$, so we may assume that $k\ge 3$.

For every $\ell\in[d/k^3]$, we say that a sequence $\sigma\in \mathcal{S}_d(k)$ is \emph{$\ell$-tame} if every subsequence of $\sigma$ of exactly $\ell$ consecutive elements has at least $\ell/2$ elements that appear exactly once in $\sigma$. 
Moreover, for every $\ell\in\{d/k^3+1,\ldots,(k-1)d+1\}$, we say that a sequence $\sigma \in \mathcal{S}_d(k)$ is \emph{$\ell$-tame} if every subsequence of $\sigma$ of $\ell$ consecutive elements contains at least two elements which appear an odd number of times (note that every subsequence of $\sigma$ of length  $\ell\ge (k-1)d+2$ contains at least $\ell-(k-1)d\geq 2$ elements appearing $k$ times and $k$ is odd).
Observe that the image of $\phi$ contains the subfamily $\cS'_d(k)\subseteq \mathcal{S}_d(k)$ of sequences which are $\ell$-tame for every $\ell \in [(k-1)d+1]$.

In the sequel, we consider a uniformly random sequence $\sigma\in \cS_d(k)$ and show that it belongs to $\cS'_d(k)$ with probability at least $\e^{-21k}$.
We analyse the events that $\sigma$ is $\ell$-tame for different values of $\ell$ separately.

\vspace{0.5em}
\noindent
\textbf{Case 1: $\ell\in [10]$.} 
We construct $\sigma$ in $kd$ steps divided into $k$ stages. During every stage, we consecutively insert the symbols $1,\ldots,d$ uniformly at random in the currently constructed sequence. 
Then, the probability that every symbol is inserted at distance at least 11 from all symbols of the same type during each of the $k$ stages (guaranteeing that $\sigma$ is $\ell$-tame for every $\ell\in [10]$) is at least
\[\prod_{i=2}^k \bigg(1-\frac{20(i-1)}{d(i-1)}\bigg)^d = (1+o(1)) \e^{-20k}\ge 2\e^{-21k}.\]

\vspace{0.5em}
\noindent 
\textbf{Case 2: $\ell\in [11,(\log d)^{10}]$.} If $\sigma$ is not $\ell$-tame for some $\ell$ in the fixed interval, then there exists a set $S\subseteq [d]$ of size $\lfloor \ell/4\rfloor$ and a subsequence $\sigma'$ of $\sigma$ of $\ell$ consecutive elements such that $\sigma'$ contains at least $2\lfloor \ell/4\rfloor$ copies of elements of $S$. 
This occurs for some $\ell\in [11,(\log d)^{10}]$ with probability at most
\begin{align*}
\sum_{\ell=11}^{(\log d)^{10}} kd \binom{d}{\lfloor \ell/4\rfloor} \binom{\ell}{2\lfloor \ell/4\rfloor} \left(\frac{k \lfloor \ell/4\rfloor}{kd-\ell} \right)^{2\lfloor \ell/4\rfloor}
& \leq \sum_{\ell=11}^{(\log d)^{10}} kd \cdot d^{\lfloor \ell/4\rfloor}\cdot 2^\ell \left(\frac{k \lfloor \ell/4\rfloor}{kd-\ell} \right)^{2\lfloor \ell/4\rfloor}
\\&\leq \sum_{\ell=11}^{(\log d)^{10}} 8kd  \left(\frac{(16k \lfloor \ell/4\rfloor)^2}{kd-\ell} \right)^{\lfloor \ell/4\rfloor}=o(1),
\end{align*}
where the expression in the first sum used that there are less than $kd$ ways to fix an interval of $\ell$ consecutive positions, $\binom{d}{\lfloor \ell/4\rfloor}$ ways to choose the symbols in $S$ and $\binom{\ell}{2\lfloor \ell/4\rfloor}$ ways to choose the exact positions of some of these symbols in the said interval.

\vspace{0.5em}
\noindent \textbf{Case 3: $\ell\in [(\log d)^{10}+1,d/k^3]$.} Again, given $\ell$, there are less than $kd$ ways to choose an interval $\sigma'$ of $\ell$ consecutive positions (seen as a subsequence of $\sigma$).
Once $\sigma'$ is fixed, we reveal its elements one by one. 
In this procedure, the probability that the $i$-th revealed element matches an element that has already been revealed is at most $k\ell/(kd-\ell)\leq 2/k^3$.
Thus, the probability that we observe more than $\ell/4$ elements that match an earlier element in $\sigma'$
(and hence fewer than $\ell/2$ elements that appear exactly once in $\sigma'$) is at most
\[
\mathbb P\left(\mathrm{Bin}\left(\ell,\frac{2}{k^3}\right)\geq \frac{\ell}{4}\right)
\leq \binom{\ell}{\ell/4} \left(\frac{2}{k^3}\right)^{\ell/4}
\leq \left(\frac{\e\ell}{\ell/4}\cdot \frac{2}{k^3}  \right)^{\ell/4}=o(d^{-2}).\]
A union bound over $O(d)$ values for $\ell$ and $O(d)$ choices for the position of $\sigma'$ shows that the probability that the sequence $\sigma$ is not $\ell$-tame for some $\ell\in [(\log d)^{10}+1,d/k^3]$ is $o(1)$.

\vspace{0.5em}
\noindent \textbf{Case 4: $\ell \in [d/k^3,(k-1)d+1]$.} 
Fix $\ell$ as required and an interval $\sigma'$ as above (again seen as a subsequence of $\sigma$ on $\ell$ elements).
Then, the expected number of elements appearing exactly once (so an odd number of times) in $\sigma'$ is $\Omega_k(d)$.
Moreover, by exchanging a pair of elements in $\sigma$, the number of elements which are met exactly once in $\sigma'$ can change by at most 2.
By the switching \Cref{lem:switchings}, $\sigma'$ contains at most $2$ elements appearing exactly once with probability $o(d^{-2})$, and the union bound over $O(d)$ choices for $\ell$ and $O(d)$ choices for $\sigma'$ shows that $\sigma$ is not $\ell$-tame for some $\ell\in [d/k^3,(k-1)d+1]$ with probability $o(1)$.

\vspace{0.5em}
\noindent 
Altogether, $\sigma$ is tame with probabiltity at least $2\e^{-21k}-o(1)\ge \e^{-21k}$, as required.\qed

\subsubsection{Proof of Lemma \ref{lem:secondmoment}}\label{subsec:secondmoment}
Fix $r\in [kd]$, a path $P \in \cP'_d(k)$ and define $\cY_r(P)$ to be the set of paths in $\cP'_d(k)$ intersecting $P$ in exactly $r$ edges. 
For $P'\in \cY_r(P)$, we denote by $\mathcal{I}(P,P')$ the set of maximal paths (with respect to inclusion) induced by $E(P)\cap E(P')$; 
in particular, $|\mathcal{I}(P,P')|\leq r$. 
For $\ell \in [0,r+1]$, we denote by $\cY_{r,\ell}(P)$ the set of paths $P'\in \cY_r(P)$ such that $E(P')\setminus E(P)$ spans $\ell$ maximal paths (with respect to inclusion). We prove that, for every $\ell\in [r+1]$,
\begin{equation}\label{main:eq}
    |\cY_{r,\ell}(P)|\leq \left(\frac{1+\eps}{d}\right)^r d^{k^5-2} \frac{(kd)!}{(k!)^d},
\end{equation}
which is enough to conclude.
In the rest, we fix $\ell\in [r+1]$ and, for ease of notation, write $\cY$ for $\cY_{r,\ell}(P)$.

Fix $P'\in \cY$. First, we describe an encoding of $P'$ using $P$, which we then use to count the number of paths in $\cY$. 
Denote by $\textbf{0}=v_1,...,v_{\ell'}=\textbf{1}$ be the sequence of vertices that either are incident to exactly one edge in $E(P')\setminus E(P)$ (or, equivalently, in $E(P)\setminus E(P')$) or belong to $\{\textbf{0},\textbf{1}\}$, in the order in which we meet them as we traverse $P'$ starting from $\textbf{0}$. 
For every $i\in [\ell'-1]$, we write $W_i$ for the oriented subpath of $P'$ from $v_i$ to $v_{i+1}$. 
Thus, $\cI(P,P')$ is the set of paths with either odd or even indices among $W_1,W_2,...,W_{\ell'-1}$. Note also that $W_1$ starts at $\textbf{0}$ and $W_{\ell'-1}$ ends at $\textbf{1}$.
We denote by $\ell_{odd}$ and $\ell_{even}$ the largest odd and even integers, respectively, which are smaller than $\ell'$. 
Recalling the map $\phi$ defined before \Cref{def:tame}, we will encode $P'$ by using the sequence $\phi(W_1),\phi(W_2),...,\phi(W_{\ell'-1})$ as follows. 

\vspace{0.5em}
\noindent 
\textbf{Step 1:} Given $P$, we start by specifying $\cI(P,P')$. To this end, we choose $2\ell$ vertices on $P$; this is done in $\tbinom{kd}{2\ell}$ ways and produces the $\ell$ paths in $E(P)\setminus E(P')$, meanwhile determining $\cI(P,P')$ as well.
Then, we specify the order and the orientation in which the paths in $\mathcal{I}(P,P')$ appear on $P'$; this is done in at most $|\mathcal{I}(P,P')|! 2^{|\mathcal{I}(P,P')|}$ ways. 
Note that, if the first chosen vertex on $P$ is not $\textbf{0}$, then this determines $\{W_1,W_3,\ldots,W_{\ell_{odd}}\}=\mathcal{I}(P,P')$ and, in extension, $\phi(W_1),\ldots,\phi(W_{\ell_{odd}})$. 
If it is $\textbf{0}$, our choice determines $\{W_2,W_4,\ldots,W_{\ell_{even}}\}=\mathcal{I}(P,P')$ and, in extension, $\phi(W_2),\ldots,\phi(W_{\ell_{even}})$. 
Since $|\cI(P,P')|\leq \ell+1$, there are at most 
\begin{equation}\label{eq:step1}
    \binom{kd}{2\ell}(\ell+1)!2^{\ell+1}
\end{equation}
choices for the maximal paths in $E(P)\cap E(P')$ and for their order and orientation on $P'$. 

\vspace{3mm}

Before moving to the second step, we collect some notation.
First, let $R_1,\ldots,R_\ell$ be the subsequence of $W_1,\ldots,W_{\ell'-1}$ of paths outside $\cI(P,P')$. 
For $i\in [\ell]$, we denote by $R_i^-$ and $R_i^+$ the starting and the ending point of $R_i$. Note that once the paths in $\cI(P,P')$ are specified, ordered and oriented, the endpoints $R_i^-$ and $R_i^+$ are uniquely determined for $i\in [\ell]$. Given this set of endpoints and the set $\cI(P,P')$, it remains to encode the sequence $\phi(R_1),\ldots,\phi(R_{\ell})$.

Further, for every $i\in [\ell]$, we partition the elements of the sequence $\phi(R_i)$ into two multisets: $S_i^{\mathrm{odd}}$ contains a single copy of each element which appears an odd number of times in $\phi(R_i)$ and $S_i^{\mathrm{rep}}:=\phi(R_i)\setminus S_i^{\mathrm{odd}}$ (where $\phi(R_i)$ is seen as a multiset here by a slight abuse of notation).
Note that, given the vertices $R_i^-$ and $R_i^+$, $S_i^{\mathrm{odd}}$ is the set of coordinates where $R_i^-$ and $R_i^+$ differ. Let $x_i:=|S_i^{\mathrm{odd}}|$, $y_i:=|S_i^{\mathrm{rep}}|$.

Order the indices $i\in [\ell]$ so that $x_i+y_i$ is decreasing (we break ties by selecting the paths with smaller indices) and let $B^+$ be the set of the first $\min\{k^4,\ell\}$ indices $i$. We let $B^-:=[\ell]\setminus B^+$. Thus, $B^-$ indexes short paths, if any exist. Let 
\[x_s=\sum_{i\in B^-}x_i, \hspace{5mm} x_b=\sum_{i\in B^+}x_i,\hspace{5mm} y_s=\sum_{i\in B^-}y_i\hspace{5mm} \text{ and }\hspace{5mm}y_b=\sum_{i\in B^+}y_i.\]

\vspace{0.5em}
\noindent 
\textbf{Step 2:}  As $B^+$ contains the $\min\{k^4,\ell\}$ longest paths, recalling that $|E(P)\cap E(P')|=r$,~we~have 
\begin{align}\label{eq:smallsizes}
kd-r-x_s-y_s= x_b+y_b\geq k^4(x_i+y_i)\text{ for every }i\in B^-.
\end{align}
Note that the inequality in~\eqref{eq:smallsizes} is a non-trivial statement only in the case when $|B^+|=\ell\ge k^4$ as otherwise $B^-$ is empty. Moreover, the number of possibilities for the set $B^+$ is at most
\begin{equation}\label{eq:step2}
\max\bigg\{1,\binom{\ell}{k^4}\bigg\}\leq \ell^{k^4}\leq d^{k^4+1}.    
\end{equation}

In addition, since $P'$ is a tame path and every subpath with index in $B^-$ has length at most $d/k^3$, we have that $x_i\ge (x_i+y_i)/2$ or equivalently $x_i\ge y_i$ for every $i\in B^-$. In particular, $x_s\ge y_s$. 
We also note that $x_i\geq 2$ for every $i\in [\ell]$. 
Indeed, assuming for contradiction that $x_i=1$ for some $i$ would imply that the endpoints of $R_i$ are adjacent in $Q^d$ and, in extension, by \textbf{P1} and the fact that both $P,P'$ are tame, that the edge $R^-_iR^+_i$ belongs to both $P,P'$. In this case, $R_i=R^-_iR^+_i\in E(P)$, which contradicts the fact that $R_i$ is spanned by edges in $E(P')\setminus E(P)$. 
Summarising, given the endpoints $R_i^-,R_i^+$ for all $i\in [\ell]$, the sets $S_i^{\mathrm{odd}}$ are determined and have sizes $x_i\geq 2$ implying that $x_s+x_b\ge 2\ell$ and $y_s\leq \min\{x_s, kd-r-2\ell\}$.

\vspace{0.5em}
\noindent 
\textbf{Step 3:} Given the endpoints $R_i^-,R_i^+$ for all $i\in B^-$ (determining the sets $S_i^{\mathrm{odd}}$ with sizes $x_i$), we encode the sequence $\{\phi(R_i)\}_{i\in B^-}$ as follows. Let $y^*:= \min\{x_s, kd-r-2\ell\}$.
First, we specify even $y_s \le 2\lfloor y^*/2\rfloor$ and then the values of $y_i$ for $i\in B^-$. 
Given $y_i$ and $x_i$, we specify the elements in $S_i^{\mathrm{rep}}$ with multiplicities and their positions in the sequence $\phi(R_i)$; this is done in at most $\binom{x_i+y_i}{y_i} d^{y_i}$ ways by first specifying the $y_i$ positions in $\phi(R_i)$ and then the element that appears in each of these positions. Finally, we specify the order in which the elements in $S_i^{\mathrm{odd}}$ appear in $\phi(R_i)$. 
By denoting for convenience by $\mathrm{Ev}(y^*)$ the set of even integers in $[0,y^*]$ and by $\mathrm{EvPar}(t)$ the set of ordered integer partitions of an even integer $t$ into $|B^-|$ non-negative even integers $(y_i)$ with $y_i\le x_i$ for every $i\in B^-$, the number of choices the sequence $\{\phi(R_i)\}_{i\in B^-}$ is bounded from above by
\begin{align*}
\sum_{y_s\in \mathrm{Ev}(y^*)} \sum_{(y_i)\in \mathrm{EvPar}(y_s)} \prod_{i\in B^-} \binom{x_i+y_i}{y_i} d^{y_i}  (x_i)! 
&\leq  \sum_{y_s\in \mathrm{Ev}(y^*)} \sum_{(y_i)\in \mathrm{EvPar}(y_s)} \prod_{i\in B^-} 2^{x_i+y_i} d^{y_i} (x_i)!\\
&\leq  \sum_{y_s\in \mathrm{Ev}(y^*)} \binom{y_s/2+|B^-|}{|B^-|} 2^{x_s+ y_s} d^{y_s}  \prod_{i\in B^-}  (x_i)!,  
\end{align*}
where in the second inequality we used an upper bound on $|\mathrm{EvPar}(y_s)|$ coming from the number of integer partitions of $y_s/2$ into $|B^-|$ parts. 
Note that the bound above takes into account some partitions in $\mathrm{EvPar}(y_s)$ potentially generating paths which are among the $k^4$ longest ones outside $\cI(P,P')$. However, since we are only interested in an upper bound, working with the restrictions on $x_i,y_i$ with $i\in [\ell]$ and $x_s,x_b,y_s,y_b$ established in the previous step is enough for our purposes.
Further, using that $y_s/2+|B^-|\le y_s/2+x_s/2\le x_s$, the last sum is further bounded from above by
\begin{align}\label{eq:step3}
\sum_{y_s\in \mathrm{Ev}(y^*)} 2^{x_s} 2^{x_s+ y_s} d^{y_s} \prod_{i\in B^-}(x_i)!\leq \sum_{y_s\in \mathrm{Ev}(y^*)} 2^{3x_s} d^{y_s} \prod_{i\in B^-} (x_i)!. 
\end{align}

\vspace{0.5em}
\noindent 
\textbf{Step 4:} It remains to deal with the long paths outside $\cI(P,P')$.
For every $j\in [d]$ and $t\in [k]$, if $j$ appears exactly $t$ times in the sequences $\{\phi(Z)\}_{Z\in \cI(P,P')}$ and $\{\phi(R_i)\}_{i\in B^-}$, then it appears exactly $k-t$ times in the sequences $\{\phi(R_i)\}_{i\in B^+}$. 
Thus, given the sequences $\{\phi(Z)\}_{Z\in \cI(P,P')}$ and $\{\phi(R_i)\}_{i\in B^-}$, the elements in $\{\phi(R_i)\}_{i\in B^+}$ are determined: observe that their number is $kd-r-x_s-y_s$. 
Thus, we can determine the sequences $\{\phi(R_i)\}_{i\in B^+}$ by specifying a sequence on $kd-r-x_s-y_s$ underlying elements and then partitioning it into $|B^+|\leq k^4$ consecutive subsequences. 
Again, since we are only aiming for an upper bound, the fact that each of the considered paths belongs to $B^+$ is not taken into account, and only the previously derived restrictions on $x_s,x_b,y_s,y_b$ are used.

For integers $a\in [0,k]$ and $b\in [0,d-1]$, we set $\mult(ad+b):=(a!)^{d}(a+1)^b$: note that $\mult(t)$ is the minimal number of permutations of the elements of a sequence on $[d]$ which leaves this sequence unchanged. 
For example, $\mult(d+1)=2$ since every sequence $\sigma$ on $d+1$ elements has two repeating elements and exchanging their positions would not change $\sigma$.
By the above reasoning, the number of ways to determine the paths indexed by $B^+$ is bounded from above by
\begin{equation}\label{eq:step4}
\frac{(kd-r-x_s-y_s)!}{\mult(kd-r-x_s-y_s)}\binom{kd}{k^4}\leq \frac{d^{k^4+1}(kd-r-x_s-y_s)!}{\mult(kd-r-x_s-y_s)}.
\end{equation}
As the procedure described in Steps 1--4 for encoding the path $P'$ is indeed injective, we have that $|\cY|$ is bounded from above by the product of the last expressions in \eqref{eq:step1}, \eqref{eq:step2}, \eqref{eq:step3} and \eqref{eq:step4}. 

Next, we bound from above the product of \eqref{eq:step3} and \eqref{eq:step4}.
To this end, we note that, for all integers $a,b\ge 1$ with $k^4a\leq b$, we have that $a!b!\leq (a-1)!(b+1)!/k^4$.
Combining this observation with the fact that $x_i\ge 2$ for every $i\in [\ell]$ and that $kd-r-x_s-y_s\ge x_b+y_b\ge k^4 x_i$ for every $i\in B^-$ when $B^-\neq \varnothing$, we obtain that
\begin{equation}\label{eq:factorials}
\bigg(\prod_{i\in B^-} x_i!\bigg)\cdot (kd-r-x_s-y_s)!\le 2^{|B^-|}\cdot \frac{(kd-r-2|B^-|-y_s)!}{(k^4)^{x_s-2|B^-|}}.
\end{equation}
Note also that, for every $i\in [kd]$, $\mult(i)/k\leq \mult(i-1)$. Recalling that $y^*=  \min\{x_s, kd-r-2\ell\} $ (so, in particular, $y^*\le x_s$), $|B^-|=\max\{0,\ell-k^4\}$ (so, in particular, $2\ell-2|B^-|\leq 2k^4$) and using \eqref{eq:smallsizes} and \eqref{eq:factorials}, the product of \eqref{eq:step3} and \eqref{eq:step4} is bounded from above~by 
\begin{align}\label{eq:smallpieces4}
  &\sum_{y_s=0}^{y^*} 2^{3x_s} d^{y_s} 2^{|B^-|}\cdot \frac{d^{k^4+1}(kd-r-2|B^-|-y_s)!/(k^4)^{x_s-2|B^-|}}{\mult(kd-r-y_s)/k^{x_s}} \nonumber\\
    & \leq \sum_{y_s=0}^{y^*} d^{y_s} \bigg(\frac{2}{k}\bigg)^{3x_s} d^{k^4+1}(2k^8)^{|B^-|} \frac{(kd-r-2|B^-|-y_s)!}{\mult(kd-r-y_s)} \nonumber\\
    &\leq \sum_{y_s=0}^{kd-r-2\ell} \bigg(\frac{8d}{k^3}\bigg)^{y_s} d^{k^4+1} k^{10\ell} \frac{(kd-r-2\ell-y_s)! (kd)^{2\ell-2|B^-|}}{\mult(kd-r-y_s)}\nonumber\\
    &\leq \sum_{y_s=0}^{kd-r-2\ell} \bigg(\frac{8d}{k^3}\bigg)^{y_s} d^{3k^4+1} k^{12\ell} \frac{(kd-r-2\ell-y_s)!}{\mult(kd-r-y_s)}.
\end{align}
Using \eqref{eq:smallpieces4} to bound the product of \eqref{eq:step3} and \eqref{eq:step4}, the product of \eqref{eq:step1}, \eqref{eq:step2}, \eqref{eq:step3} and \eqref{eq:step4} is at most
\begin{align*}
|\cY|&\ \leq  d^{4k^4+2} \binom{kd}{2\ell} 2^{\ell+1} (\ell+1)! k^{12\ell} 
\sum_{y_s=0}^{kd-r-2\ell} \bigg(\frac{8d}{k^3}\bigg)^{y_s} \frac{(kd-r-2\ell-y_s)!}{\mult(kd-r-y_s)}.
\end{align*}
Finally, the proof is completed by the following claim applied with $z=r+y_s\ge |\cI(P,P')|\ge \ell-1$, summing over at most $kd$ choices for $r$ and $\ell$, and choosing $k$ large so that, in particular, $4k^4+6 < k^5$ and $16(1+\eta) < k^3$ ensuring that $\sum_{i\ge 0} (8(1+\eta)/k^3)^i = O(1)$.

\begin{claim}
Fix suitably small $\eta>0$ and an integer $k = k(\eta)\in [\eta^{-4},\eta^{-5}]$. Given $z,\ell\in [0,kd]$ with $z+2\ell \leq kd$ and $\ell\leq z+1$, we have that
\begin{equation}\label{eq:mult}
 \binom{kd}{2\ell} 2^{\ell+1} (\ell+1)! k^{12\ell} 
  \frac{(kd-2\ell-z)!}{\mult(kd-z)}
\leq d \left(\frac{1+\eta}{d} \right)^z \frac{(kd)!}{\mult(kd)} = d  \left(\frac{1+\eta}{d} \right)^z \frac{(kd)!}{(k!)^d}.
\end{equation}
\end{claim}
\begin{proof}
First of all, note that $\ell-1\le z\le kd-2\ell$ implies that $\ell\le (kd+1)/3$.
We consider 3 cases based on the values of $\ell$ and $z$.

\vspace{0.5em}
\noindent 
\textbf{Case 1:  $\ell\geq 100k^2d/\log d$.} Then, using that $kd-2\ell\ge kd/4$, we have
\begin{align*}
     \prod_{i=0}^{z-1} \frac{1}{kd-2\ell-i} \leq \frac{(kd)^{kd-2\ell-z}}{(kd-2\ell)!}\leq \frac{(kd)^{kd-2\ell-z}}{((kd-2\ell)/\e)^{kd-2\ell}} =\bigg(\frac{\e kd}{kd-2\ell}\bigg)^{kd-2\ell}\frac{1}{d^z} \leq  \frac{k^{kd}}{d^{z}}.
\end{align*}

Therefore, if $\ell\geq 100k^2d/\log d$, then
\begin{align*}
\binom{kd}{2\ell} 2^{\ell+1} (\ell+1)! k^{12\ell} \frac{(kd-2\ell-z)!}{\mult(kd-z)}
& \le \frac{(kd)!}{(2\ell)!(kd-2\ell)!} 2^{\ell+1} (\ell+1)! k^{12\ell} (kd-2\ell)!\prod_{i=0}^{z-1} \frac{1}{kd-2\ell-i}
\\ &\leq (kd)! \frac{kd (2k)^{12\ell} (\ell/\e)^{\ell}}{(2\ell/\e)^{2\ell}}  \frac{k^{kd}}{d^z}
\\ &\leq  \frac{(kd)! k^{kd}}{\ell^{0.9\ell} d^z} \leq  \frac{(kd)!}{d^z} \e^{-50k^2 d} \leq \frac{1}{d^z} \frac{(kd)!}{(k!)^d}, 
\end{align*}
where the first inequality used that $\mult(kd-z)\ge 1$, the second inequality used Stirling's formula for $\ell!$ and $(2\ell)!$ and the bound $2^{\ell+1} k^{12\ell} (\ell+1)!\le kd (2k)^{12\ell} \ell!$, the third inequality used the relation $\ell^{0.1\ell}\ge kd (2k)^{12\ell} (\e/4)^{\ell}$ and the last inequality comes from Stirling's formula for $k!$.

\vspace{3mm}
\noindent \textbf{Case 2:  $\ell < 100k^2d/\log d$ and $kd-2\ell-z\geq 2d$.} Fix integers $a\in [0,k]$ and $b\in [0,d-1]$ such that $z=ad+b$. Also, set $z^*:= kd -\eta^{-2}d$ .

\vspace{3mm}
\noindent 
If $z\leq z^*$, then $k-a\geq \eta^{-2}$ and we have
\begin{align}\label{eq:1}
  \prod_{i=0}^{z-1} \frac{1}{kd -2\ell-i}& \leq
\prod_{i=2}^{a+1} \left( \frac{1}{(k-i)d}  \right)^{d}
\cdot \left( \frac{1}{(k-a-2)d} \right)^{b} \nonumber
\\& \leq \prod_{i=0}^{a-1} \left( \frac{1}{(k-i)d} \frac{k-i}{k-i-2}  \right)^{d}
\cdot \left( \frac{1}{(k-a)d}\frac{k-a}{k-a-2} \right)^{b} \nonumber
\\&\leq d^{-z} \left(\frac{\eta^{-2}}{\eta^{-2}-2}\right)^z \prod_{i=0}^{a-1} \left(\frac{1}{k-i}\right)^d
\cdot \left( \frac{1}{k-a} \right)^{b} \nonumber
\\& \leq d^{-z} \cdot \left(1+3\eta^2\right)^z \frac{\mult(kd-z)}{\mult(kd)} \leq \bigg(\frac{1+\eta}{d}\bigg)^z \frac{\mult(kd-z)}{\mult(kd)},
\end{align}
where the third inequality used that $\tfrac{k-i}{k-i-2}$ is bounded from above by $\frac{k-a}{k-a-2}\le \frac{\eta^{-2}}{\eta^{-2}-2}$ for every $i\in [0,a]$.
Moreover, if $z > z^*$, by \eqref{eq:1} and the fact that $z-z^*\le\eta^{-2}d\leq 2\eta^2z^* $ (using that $k\ge \eta^{-4}$), we have
\begin{align}\label{eq:2}
  \prod_{i=0}^{z-1} \frac{1}{kd -2\ell-i}& =\prod_{i=0}^{z^*-1} \frac{1}{kd -2\ell-i} \prod_{i=z^*}^{z-1} \frac{1}{kd -2\ell-i}
  \nonumber
  \\&   \leq  d^{-z^*} \cdot \left(1+3\eta^2\right)^{z^*} \frac{\mult(kd-z^*)}{\mult(kd)} \cdot d^{-z+z^*} \nonumber 
  \\& \leq  d^{-z} \cdot \left(1+3\eta^2\right)^{z^*} \cdot k^{z-z^*} \cdot \frac{\mult(kd-z)}{\mult(kd)} \nonumber 
\\& \leq  ((1+3\eta^2)\exp(2\eta^{2}\log k))^{z^*}  \frac{1}{d^z}\frac{\mult(kd-z)}{\mult(kd)}\leq \left(\frac{1+\eta}{d} \right)^z \cdot \frac{\mult(kd-z)}{\mult(kd)},
\nonumber
\end{align}
where the last inequality used that $k\leq \eta^{-5}$ and $z^*\le z$.
Thus, in both cases, we have that, 
\begin{align*}
\binom{kd}{2\ell} 2^{\ell+1} (\ell+1)! 
&k^{12\ell} \frac{(kd-2\ell-z)!}{\mult(kd-z)}\leq (kd)! \cdot\left(\frac{2^{\ell+1} (\ell+1)! k^{12\ell}}{(2\ell)!}\right) \bigg(\prod_{i=0}^{z-1}\frac{1}{kd-2\ell-i}\bigg) \frac{1}{\mult(kd-z)}\\
&\leq (kd)! \max_{\ell\in [0,kd]} \left(\frac{2^{\ell+1} (\ell+1)! k^{12\ell}}{(2\ell)!}\right) \cdot \left(\frac{1+\eta}{d} \right)^z \cdot \frac{1}{\mult(kd)} \leq  d \cdot \left(\frac{1+\eta}{d} \right)^z \cdot \frac{(kd)!}{(k!)^d}. 
\end{align*}

\noindent
\textbf{Case 3:  $\ell < 100k^2d/\log d$ and $ kd-2\ell-z < 2d$.} Then, by using Stirling's formula and the fact that $k$ is a large constant, we have
\begin{align*}
\binom{kd}{2\ell} 2^{\ell+1} (\ell+1)! k^{12\ell}  
&\frac{(kd-2\ell-z)!}{\mult(kd-z)}
\leq  (kd)! \left(\frac{2^{\ell+1} (\ell+1)! k^{12\ell}}{(2\ell)!}\right) \frac{1}{(kd-2\ell)!} (kd-2\ell-z)!\\
&\leq  \max_{\ell\in [0,kd]} \left(\frac{2^{\ell+1} (\ell+1)! k^{12\ell}}{(2\ell)!}\right) \bigg(\frac{(k!)^d (kd-2\ell-z)!}{(kd-2\ell)!}\bigg) \frac{(kd)!}{(k!)^d}\\
&\leq d \bigg(\frac{(k+\sqrt{k})^z (kd-2\ell-z)^{kd-2\ell-z}}{(kd-2\ell)^{kd-2\ell}}\bigg) \frac{(kd)!}{(k!)^d}\\
&\leq d \bigg(\frac{k+\sqrt{k}}{kd-2\ell}\bigg)^z \frac{(kd)!}{(k!)^d} \leq d\left(\frac{1+\eta}{d}\right)^z \frac{(kd)!}{(k!)^d},
\end{align*} 
where the first inequality used that $\mult(kd-z)\ge 1$, the third inequality used that $k!\le ((k+\sqrt{k})/\e)^{k-3}$ for large $k$ and that $(k-3)d\le z$, and the last inequality used that $k\ge \eta^{-4}$.
The bounds established in the three cases complete the proof.
\end{proof}

\noindent
\textbf{Acknowledgements.} Part of this work was done during visits of the second author to TU Wien, and of the third author to the University of Sheffield. The authors would like to thank these institutions for their hospitality.
The second author wishes to thank Alan Sly for fruitful discussions and Michael Krivelevich for his guidance and support. The fourth author wishes to thank Andrei Kupavskii for fruitful discussions.
The authors are grateful to Remco van der Hofstad, Asaf Nachmias and Wojciech Samotij for useful comments and suggestions.

\bibliographystyle{plain}
\bibliography{perc} 
\end{document}